\documentclass[a4paper, 10pt, twoside]{article}

\usepackage[utf8]{inputenc}
\usepackage{exscale, fullpage,url}
\usepackage[centertags]{amsmath}
\usepackage{color}
\usepackage{amssymb}
\usepackage{amsthm}
\usepackage{dsfont}
\usepackage{mathrsfs}
\usepackage{upgreek}
\usepackage{tikz-cd}
\usepackage[hidelinks]{hyperref}
\usepackage{graphicx,fp}
\usepackage{mathtools}

\usepackage{calc}
\usepackage{enumitem}

\usepackage[capitalise]{cleveref}
\crefname{enumi}{}{}
\usepackage{stmaryrd}
\newif\ifpdf
\pdffalse
\pdftrue

\usepackage{thmtools}

\newtheorem{thm}{Theorem}[section]

\newcommand{\myendsymbol}{\ensuremath{\diamondsuit}}

\declaretheorem[
  style=definition,
  title=Example,
  qed={$\myendsymbol$},
  refname={example,examples},
  Refname={Example,Examples},
  sharenumber=thm,
]{exa}

\declaretheorem[
  style=definition,
  title=Definition,
  qed={$\myendsymbol$},
  sharenumber=thm,
]{dfn}

\declaretheorem[
  style=definition,
  title=Lemma,
  qed={},
  sharenumber=thm,
]{lem}

\declaretheorem[
  style=definition,
  title=Proposition,
  qed={},
  sharenumber=thm,
]{prop}

\declaretheorem[
  style=definition,
  title=Corollary,
  qed={},
  sharenumber=thm,
]{cor}

\declaretheorem[
  style=definition,
  title=Remark,
  qed={$\myendsymbol$},
  sharenumber=thm,
]{rmk}

\declaretheorem[
  style=definition,
  title=Notation,
  qed={$\myendsymbol$},
  sharenumber=thm,
]{notation}

\newcommand{\nd}{\noindent}

\newcommand{\dV}{{\mathds V}}
\newcommand{\dR}{{\mathds R}}
\newcommand{\dC}{{\mathds C}}
\newcommand{\dQ}{{\mathds Q}}
\newcommand{\dN}{{\mathds N}}

\newcommand{\dZ}{{\mathds Z}}
\newcommand{\dP}{{\mathds P}}

\newcommand{\dL}{{\mathbb L}}

\newcommand{\bD}{{\mathbb D}}

\newcommand{\cA}{\mathscr{A}}
\newcommand{\cB}{\mathcal{B}}
\newcommand{\cC}{\mathcal{C}}
\newcommand{\cD}{\mathscr{D}}

\newcommand{\cF}{\mathcal{F}}
\newcommand{\cG}{\mathcal{G}}
\newcommand{\cH}{\mathcal{H}}
\newcommand{\cI}{\mathcal{I}}

\newcommand{\cK}{\mathcal{K}}
\newcommand{\cL}{\mathcal{L}}
\newcommand{\cM}{\mathcal{M}}
\newcommand{\cN}{\mathcal{N}}
\newcommand{\cO}{\mathcal{O}}

\newcommand{\cS}{\mathcal{S}}

\newcommand{\cU}{\mathcal{U}}

\newcommand{\cY}{\mathcal{Y}}

\newcommand{\fg}{\mathfrak{g}}

\newcommand{\D}{\displaystyle}

\newcommand{\SC}{\scriptstyle}

\DeclareMathOperator{\Spec}{\textup{Spec}\,}
\DeclareMathOperator{\Proj}{\textup{Proj}\,}

\DeclareMathOperator{\Tr}{\textup{Tr}}
\DeclareMathOperator{\Der}{\textup{Der}}

\DeclareMathOperator{\End}{\textup{End}}
\DeclareMathOperator{\Sym}{\textup{Sym}}
\DeclareMathOperator{\GL}{\textup{GL}}
\DeclareMathOperator{\diag}{\textup{diag}}
\DeclareMathOperator{\vol}{\textup{vol}}

\DeclareMathOperator{\Lie}{\textup{Lie}}
\DeclareMathOperator{\DR}{\mathit{DR}}

\DeclareMathOperator{\im}{\textup{im}}
\DeclareMathOperator{\FL}{\textup{FL}}
\DeclareMathOperator{\id}{\textup{id}}

\DeclareMathOperator{\gr}{\textup{Gr}}

\DeclareMathOperator{\SL}{\textup{SL}}

\newcommand{\Supp}{\textup{Supp}}

\newcommand{\MHM}{\textup{MHM}}
\newcommand{\HM}{\textup{HM}}
\newcommand{\Mh}{{^{H,*}\!\!}\cM_L^\beta}
\newcommand{\MhD}{{^{H,!}\!\!}\cM_L^{-\beta}}
\newcommand{\Nh}{{^{H,*}\!\!}\cN_L^\beta}

\newcommand{\mbd}{\mathbb{D}}

\newcommand{\mbz}{\mathds{Z}}

\newcommand{\mcu}{\mathcal{U}}

\newcommand{\mcy}{\mathcal{Y}}

\newcommand{\ra}{\rightarrow}
\newcommand{\lra}{\longrightarrow}

\newsavebox\foobox

\newcommand{\suchthat}{\;\ifnum\currentgrouptype=16 \middle\fi|\;}

\DeclareMathOperator{\Id}{Id}
\DeclareMathOperator{\trace}{trace}
\DeclareMathOperator{\Tot}{Tot}
\DeclareMathOperator{\rk}{rk}

\DeclareMathOperator{\Pic}{Pic}

\DeclareMathOperator{\sheafHom}{\mathscr{H}\kern -3pt\textit{om}\kern 1pt}
\DeclareMathOperator{\sheafDer}{\mathscr{D}\kern -1pt\textit{er}\kern 1pt}
\DeclareMathOperator{\differential}{d\!}
\DeclareMathOperator{\ad}{ad}

\newcommand{\N}{\mathds{N}}

\renewcommand{\P}{\mathds{P}}
\newcommand{\R}{\mathds{R}}
\newcommand{\C}{\mathds{C}}
\newcommand{\Q}{\mathds{Q}}
\newcommand{\Z}{\mathds{Z}}

\renewcommand{\O}{\mathcal{O}}
\newcommand{\Ell}{\mathscr{L}}

\newcommand{\fsl}{\mathfrak{sl}}
\newcommand{\fgl}{\mathfrak{gl}}

\let\originalleft\left
\let\originalright\right
\renewcommand{\left}{\mathopen{}\mathclose\bgroup\originalleft}
\renewcommand{\right}{\aftergroup\egroup\originalright}

\newcommand\restrK[2]{{(#1)_{| {#2}}}}
\newcommand\restr[2]{{#1_{| {#2}}}}

\newcommand\defstyle[1]{\textbf{#1}}

\begin{document}
\title{\Large Tautological systems, homogeneous spaces and the holonomic rank problem}
\author{\small Paul Görlach, Thomas Reichelt, Christian Sevenheck, Avi Steiner and Uli Walther}

\renewcommand{\thefootnote}{}
\footnotetext{
\noindent
PG and CS were partially supported by DFG grant SE 1114/5-2.
 TR and AS acknowledge support by DFG grant RE 3567/1-1.
 UW was supported by NSF grant DMS-2100288 and by
 Simons Foundation Collaboration Grant for Mathematicians \#580839 and SFI-MPS-TSM-00012928.

}
\renewcommand{\thefootnote}{\arabic{footnote}}

\maketitle

\begin{abstract}
  \noindent Many hypergeometric differential systems that arise from a geometric setting can be endowed with the structure of mixed Hodge modules. We generalize this fundamental result to the tautological systems associated to homogeneous spaces by giving a functorial construction for them. As an application, we solve the holonomic rank problem for such tautological systems in full generality.
\end{abstract}

\tableofcontents

\section{Introduction} \label{sec:intro}

The purpose of this paper is to investigate differential systems that
one can naturally associate to group actions on smooth algebraic
varieties, and more specifically to representations of algebraic
groups. More precisely,  consider
the following data: a complex algebraic group $G$
acting linearly on a finite-dimensional vector space, an invariant
subvariety of this space, and a homomorphism from the Lie algebra of
$G$ into the complex numbers. To this situation is naturally
attached a \emph{tautological system}, which is an equivariant
$\mathscr{D}$-module on the dual vector space. In the case where the group $G$
  is an algebraic torus, this construct
was considered  first by Gel'fand and his collaborators
  Gindikin, Graev, Kapranov and Zelevinsky in a seminal investigation
  that originated in the study of Aomoto integrals on hyperplane
  arrangements; cf.\ \cite{GGG80,GG86,GZ86,GGZ87,GKZ89,GKZ90}.  Today,
  these \emph{GKZ-} or \emph{$A$-hypergeometric systems} are of
  fundamental importance in several areas of mathematics; see, e.g.,
  \cite{GKZReview} for an overview on the algebraic aspects of this
  theory.
  For more general Lie groups this construction seems to
go back to \cite[Section II.4]{Hot98}, and has more recently been the
main focus in a series of papers by Bloch, Huang, Lian, Song, Yau and
Zhu (\cite{TautPeriod1, Bloch_HolRankProb, TautPeriod2,
  HuangLianZhu}).

One of the main motivations for studying
  tautological systems arises in mirror symmetry, understood in the
classical sense of recovering enumerative geometry information (i.e.,
quantum cohomology) of certain symplectic varieties by period integral
computations of their mirror families (or oscillating integrals in the
non-Calabi--Yau case). Aided by the high level of
  sophistication reached in the theory of GKZ-systems, the
investigation of mirror symmetry for
  complete intersections inside toric varieties has been very
  successful and can be considered as
  settled (at least under sufficient positivity assumptions, cf.\
  \cite{Giv7, Iritani, RS12} and the respective bibliography trees).
It is, however, a longstanding and
challenging problem to establish mirror symmetry, expressed as an
equivalence of $\mathscr{D}$-modules (possibly with additional
structures, such as Hodge modules or irregular variants of them) for
non-toric varieties. Our present work is a contribution to the
fundamental properties of tautological systems in the more general
setup, with a view towards applications to mirror symmetry.

An important class of examples arise from homogeneous spaces; for a
partial list of known results on mirror symmetry in that context see
\cite{Rietsch, Marsh-Rietsch, Lam-Templier}. A common feature of these
papers is that the mirror of a Fano manifold that is
  also a homogeneous space for some group $G^\vee$ consists of a
Landau--Ginzburg potential, constructed via Lie theoretic methods from
the Langlands dual group $G$ of $G^\vee$. When restricted to a torus
inside $G$, such a potential function can be expressed as a Laurent
polynomial.  Describing, and then studying, an appropriate partial
compactification of this mirror Laurent polynomial is a major and
central problem in the area; the toric situation is considered for
example in \cite{RS12} and, from a very different point of view, in
\cite{Siebert-TropLG}.
It is hence a problem of fundamental importance to
    decribe, for a given homogeneous space $X=G/P$ embedded into a
    projective space, the differential system satisfied by periods of
  families of its hyperplane sections. Such analogue
    of a GKZ-system should yield (by dimensional reduction) the
  mirror $\mathscr{D}$-module considered in the papers by Rietsch,
  Marsh, Lam and Templier above.
Our main findings here, as they relate to homogeneous
  spaces and
paraphrased in \cref{theo:Main} below, give criteria to determine when
tautological systems arise as such $\cD$-modules in a setting where we
allow $G$ to be any linear algebraic group, and where the
representation will be in the space of sections of some equivariant
line bundle $L$ on $X$.

Our investigations show that, in order to obtain a
  nonzero tautological system for general Lie group actions,  one
needs to impose rather delicate conditions on the bundle $L$ and the
 corresponding parameter Lie algebra homomorphism. If
these conditions---which we make explicit---hold true, we show that
the corresponding tautological system has, in fact, a
functorial description  similar to the one for
  GKZ-systems given in \cite{GKZ90,SW09}, and thus naturally
underlies a mixed Hodge module. We determine its possible weights, and
we show how to compute its solution rank at any point. In particular,
we determine its holonomic rank in terms of the dimension of the
cohomology of a natural family of (complements of) hyperplane sections
of $X$. In order to put this in context, recall that
  the variation of certain relative cohomology
  groups to families of hypersurfaces in a torus is captured
  by the restriction of the GKZ-system induced by the hypersurface to
  its regular locus (see \cite[Thm.~2.13]{Reich2},
  \cite{Stienstra} and \cite{ReichWalth-Duco}) while information on the
  intersection cohomology of the associated toric projective variety
  can be estimated as well (compare, for example
  \cite[Prop.~2.14]{RS12}).

Our general rank result gives a completely general solution for
  arbitrary line bundles to the holonomic rank problem raised in
  \cite{Bloch_HolRankProb} (and adressed in \cite{HuangLianZhu} in the
  special case $L=-K_X$). Considering rather finer structure aspects,
we further show that in many cases (classified by the value of the
parameter Lie algebra homomorphism), the monodromy representation
defined by the smooth part of the tautological system is irreducible.
This, again, has well-established counterparts in the
  case of a torus action: compare
  \cite{GKZ90,W-taka,Saito11,SW12,Beukers}. It is interesting that the resulting mixed Hodge modules have only two nonzero weights, in contrast to the GKZ case of \cite{ReichWalth-Weight} where the torus dimension predicts the number of weights.

    The potential applications in mirror symmetry can be understood along the lines of what is already known in the case of toric varieties (see \cite{RS10, RS12, RS20}). Namely, the Hodge module structure that we prove to exist on the tautological system (the object called $\tau(\rho, \hat{X}, \beta)$ in \cref{theo:Main} below) can be used to construct a non-commutative Hodge structure on the partial Fourier transform of its dimensional reduction to the K\"ahler moduli space  mentioned above. Conjecturally, this one is identified with the quantum $\cD$-module of the A-model $P^\vee\backslash G^\vee$ of the given homogeneous space $G/P$ (the variation of the cohomology of hyperplane sections of which is controlled by our tautological system). This would constitute a full mirror statement, currently only known in a few special cases due to the work of Rietsch, Marsh, etc. It should be noted though that in order to carry out this program, considerable extra work is needed, starting with identifying the Hodge filtration on the tautological system (the corresponding statement in the GKZ-case is the main result in \cite{RS20}).

  Besides applications to
mirror symmetry, our results should also lay the foundations for
further studies of Hodge theory of various differential modules
constructed from representations of algebraic groups, such as
Frenkel--Gross connections (see \cite{FrenkelGross}) or generalized
Kloosterman $\cD$-modules (\cite{HeinlothNgoYun}).

In the remainder of this introduction, we will describe our main results in more detail, and we give an overview on the content of this paper. The main character, the tautological system, is defined below. In terms of notation,   for a vector space $V$ and its dual space $W:=V^\vee$, we denote the Fourier--Laplace transformation functor
$\FL^V\colon \textup{Mod}(\mathscr{D}_V)\rightarrow\textup{Mod}(\mathscr{D}_W)$  (see \cref{ssec:FLBasics} below for more details about Fourier--Laplace transformations on arbitrary vector bundles). For now, $G'$ can be any linear algebraic group, but in the
later parts we will consider a group $G$ acting transitively on a variety $X$, and $G'$ will denote the
product $\dC^*\times G$, acting on equivariant line bundles $L\rightarrow X$.
\begin{dfn} \label{def:tautSys}
  Let $\rho \colon G' \to \GL(V)$ be a finite-dimensional rational representation of an algebraic group and denote the induced Lie algebra representation by $\differential\rho \colon \mathfrak g' \to \fgl(V)$. Let $\overline{Y}$ be a $G'$-invariant closed subvariety of $V$. For a Lie algebra homomorphism $\beta \colon \mathfrak g' \to \C$, define the left $\mathscr{D}_V$-module
  \begin{equation}\label{eq:tau_hat}
  \hat{\tau}(\rho, \overline{Y}, \beta) := \mathscr{D}_{V}/\left(\mathscr{D}_{V} \cI + \mathscr{D}_{V}\left(Z_V(\xi)-\beta'(\xi) \mid \xi \in \mathfrak g'\right)\right),
  \end{equation}
    where $\cI \subseteq \cO_V$ is the vanishing ideal of $\overline{Y}$, where
    $Z_V(\xi)$ denotes the vector field on $V$ given by the infinitesimal action of $\fg'$ (see \cref{lem:equivVFOnVectSp} for a detailed discussion), and where   $\beta'(\xi) := \trace(\differential\rho(\xi))-\beta(\xi)$.

  Its Fourier--Laplace transform
  \begin{equation}\label{eq:tau}
  \tau(\rho, \overline{Y}, \beta) := \FL^V(\hat\tau(\rho, \overline{Y}, \beta))
  \end{equation}
  is a left $\mathscr{D}_{V^\vee}$-module called the \defstyle{tautological system} associated to $\rho$, $\overline{Y}$ and $\beta$.
\end{dfn}

The next statement summarizes our main results.
To state them, assume that $X$ is a smooth projective variety, and that $G$ is a reductive and connected linear algebraic group that acts transitively on $X$. Suppose that $L\rightarrow X$ is a $G$-equivariant line bundle on $X$, with sheaf of sections $\Ell$,
which we assume to be very ample. We put $G':=\dC^*\times G$, and we define an action of $G'$ on $L$ by letting the $\dC^*$-factor act via inverse scaling in the fibres of $L$ (see \cref{def:EquLineChar} for a more precise and more general description). Setting $V:=H^0(X, \mathscr{L})^\vee$, we obtain a representation $G'\rightarrow \GL(V)$. Moreover, since $L$ is very ample, the linear system $|\mathscr{L}|$ yields an embedding $g\colon X\hookrightarrow \P V$. Let $\hat{X}\subseteq V$ be the affine cone; this is a $G'$-invariant subvariety.
Notice that there is an isomorphism $L^* \cong \hat{X}\setminus \{0\}$, where $L^*$ is the complement of the zero section of $L\rightarrow X$
and we write
$\iota \colon L^* \hookrightarrow V$ for the corresponding locally closed embedding obtained by composing this isomorphism with the embedding $\hat{X}\setminus\{0\}\hookrightarrow V$.
Choose any Lie algebra homomorphism $\beta\colon \fg'=\dC\mathbf{e}\oplus \fg\rightarrow \dC$ with $\beta_{|\fg}=0$ (this is forced on $\beta$  if $G$ is semisimple, since then $[\fg,\fg]=\fg$), i.e., choose a number $\beta(\mathbf{e})\in \dC$. \begin{thm}[\cref{thm:restrictedTauHatDescription}, \cref{theo:tau_is_MHM} and \cref{cor:holRankProb}]\label{theo:Main}
In the above situation, the following statements hold true.
\begin{enumerate}
    \item
        Let $\beta(\mathbf{e}) \notin  \dZ$. We have
        \begin{enumerate}
        \item
        $$
        \tau(\rho, \hat{X}, \beta)
        =
        \left\{
        \begin{array}{lcl}
         \FL^V(\iota_+ \cO_{L^*}^{\ell/k}) &\quad & \textup{if } \mathscr L^{\otimes \ell} \cong \omega^{\otimes (-k)}_X \textup{ and }\beta(\mathbf{e}) = \ell/k,\\ \\
        0 & \quad &  \textup{else},
        \end{array}
        \right.
        $$
        where $\cO_{L^*}^{\ell/k}$ is a smooth $\mathscr{D}_{L^*}$-module of rank $1$ on $L^*$ (and we denote by $\underline{\dC}_{L^*}^{\ell/k}$
        its  associated local system) which underlies a pure complex Hodge module of weight $\dim(X)+1$.
        \item
        If $\tau(\rho, \hat{X}, \beta)\neq 0$, then it underlies a simple pure complex Hodge module of weight $\dim(X)+\dim(V^\vee)$. In particular, the  local system corresponding to the restriction of $\tau(\rho, \hat{X}, \beta)$ to the complement of its singular locus (or, phrased differently, its monodromy representation) is irreducible.
        \item
        The holonomic rank of $\tau(\rho, \hat{X}, \beta)$ equals         $$
        \dim_\dC H^{\dim(X)}_c\left(X \setminus Z(\lambda), \underline{\dC}_\lambda^{\ell/k}\right)
        $$
        for a generic $\lambda\in V^\vee=H^0(X,\mathscr{L})$, where
        $Z(\lambda)$ is the vanishing locus in $X$ of the section $\lambda$, and where
        $\underline{\dC}_\lambda^{\ell/k}$ is the local system
        $\lambda_{|X\setminus Z(\lambda)}^* \underline{\dC}_{L^*}^{\ell/k}$.
        \end{enumerate}
    \item
        Let $\beta(\mathbf{e}) \in \dZ_{>0}$. We then have
        $$
            \tau(\rho, \hat{X}, \beta) = \FL^V\left(H^0 \iota_\dag \cO_{L^*}\right),
        $$
        which underlies a (rational) mixed Hodge module (i.e. an element in $\MHM(V^\vee)$), with weights in $\left\{\dim(X)+\dim(V^\vee),\, \dim(X)+\dim(V^\vee)+1\right\}$. Its holonomic rank is given by
        $$
            \dim_\dC H^{\dim(X)}(X\setminus Z(\lambda), \dC).
        $$
  \end{enumerate}
\end{thm}
Since the above theorem is meant only as an overview of our results, we ignore the case where $\beta(\mathbf{e})\in \dZ_{\leq 0}$ here, as it is essentially uninteresting (see \cref{cor:LocColocHatTau} for more details).
In a similar spirit, we only mention the holonomic rank here, whereas
\cref{cor:holRankProb} contains finer results concerning the fibre rank (resp.\ the solution rank) of the system $\tau(\rho, \hat{X}, \beta)$ at any point. Notice further that the points 1.(a) and 1.(b) in the above theorem imply in particular that for any given equivariant line bundle $L$ that gives a non-zero tautological system, some sufficiently high power of it yields a system with irreducible monodromy representation.

There are essentially three main ingredients in the proof of the above results. First, one needs to rewrite the
Fourier--Laplace transformation entering in the definition of the tautological system in \eqref{eq:tau} as an operation that involves only functors defined in the category of mixed Hodge modules. This is done using a strategy that already appeared in \cite{Reich2}, namely, via the Radon transformation for algebraic $\mathscr{D}_{\dP^n}$-modules.
Due to the possible non-integrality of $\beta$ however, we need here a variant of this transformation.
This twisted  Radon transformation was used in \cite[Section 5.2]{RS20}, and the relevant adaptations are discussed in \cref{sec:FLOnVecBundles}.

The second main point of our investigations is to study non-vanishing criteria for tautological systems. As already  mentioned, a tautological system, as defined by \eqref{eq:tau} resp.\ \eqref{eq:tau_hat} is often the zero module, especially when the dimension of the group $G$ is larger than the dimension of the $G$-variety $Y$. This is exactly the situation that we are facing when studying tautological systems defined by homogeneous spaces. This aspect seems to have been overlooked in the previous studies of tautological systems (e.g.\ in \cite{Bloch_HolRankProb,HuangLianZhu}).
We therefore need to develop both necessary and sufficient criteria for such systems to be non-zero. As stated in our main result above, they involve both constraints on the equivariant line bundle $L$ and on the parameter homomorphism $\beta$. We develop these criteria in \cref{sec:nonVanishingCriteria,sec:FormuleBeta}, using some facts on modules over rings of twisted differential operators as well as some standard techniques from representation theory.

The third important ingredient of our construction is a localization result for the Fourier--Laplace transform $\hat{\tau}(\rho,\hat{X},\beta)$ of the tautological system, treated in \cref{subsec:LocProp} and \cref{subsec:CoLocProp}. Although the problem is similar to the  corresponding result in the GKZ-case in \cite{SW09}, the techniques are very much different. It is here where the two cases
$\beta(\mathbf{e})\in\dZ$ resp.\ $\beta(\mathbf{e})\notin\dZ$
need to be treated separately. While the latter is a relatively simple argument concerning eigenvalue decomposition for an operator derived from the Euler vector field on the space $V$, the former is more delicate. Contrary to the strategy in \cite{SW09} (using so-called Euler-Koszul homology) we study here various Lie algebroid cohomologies and prove some vanishing theorems about them.

\bigskip

\emph{Outline:}
Let us give a more specific overview over the various parts of the paper. Notice that the level of generality is decreasing, in the sense that the results in the earlier sections apply to more general situations than the main result as stated above. In particular, Section \ref{sec:nonVanishingCriteria}
contains results of general interest about
$\cD$-modules related to group actions on algebraic varieties.

We start by defining in \cref{sec:OBeta} certain Hodge modules on
line bundles $L\rightarrow X$ over smooth varieties
(or rather on the complement of the zero section $L^*$). Their underlying $\mathscr{D}$-modules (denoted by $\cO_{L^*}^\beta$) generalize the twisted structure sheaf $\mathscr{D}_{\dC^*}/\mathscr{D}_{\dC^*}(\partial_t\cdot t + \beta)$
(which would correspond to the case where the variety is a point).
Then we study their Fourier--Laplace transforms in \cref{sec:FLOnVecBundles}, show that they still underlie a mixed Hodge module on the dual bundle and discuss a (complex of) $\cD$-module(s) on the space of global sections of this dual bundle as well as a geometric interpretation of it as twisted cohomology of hyperplane sections.

In \cref{sec:nonVanishingCriteria} we address the question under which hypotheses the tautological system
$\tau(\rho,\overline{Y},\beta)$ and its Fourier transform $\hat{\tau}(\rho,\overline{Y},\beta)$ are a non-zero $\mathscr{D}_{V^\vee}$- resp.\ $\mathscr{D}_V$-module. We first consider a quite general situation of a smooth algebraic variety $Y$ endowed with the action of an algebraic group $G'$. From the vector fields induced by this group action, together with a Lie algebra homomorphism $\beta$, we construct a $\mathscr{D}_Y$-module $\mathscr{N}_Y^\beta$. Especially important is the case where $Y$ occurs as an orbit in a vector space $V$ underlying a rational representation $\rho\colon G'\rightarrow \GL(V)$. According to principles outlined above, there is a tautological system $\tau(\rho,\overline{Y},\beta)$ and its Fourier transform $\hat{\tau}(\rho,\overline{Y},\beta)$. We relate in \cref{cor:FLTautSysAsDirectImage} the restriction of $\hat{\tau}(\rho,\overline{Y},\beta)$ to $Y$ with the intrinsically defined module $\mathscr{N}_Y^\beta$. We then develop therefore a framework, based on the formalism of Lie algebroids and their universal enveloping algebras, to study the vanishing resp.\ non-vanishing of the module $\mathscr{N}_Y^\beta$. The first main result in this section is \cref{thm:nonZeroGeneralSetup} which gives a sufficient criterion for $\hat{\tau}(\rho, \overline{Y},\beta)$ to be non-zero. Afterwards, this is applied to the more specific case where the variety $Y$ is the complement of the zero section of the total space of a line bundle over a variety $X$ equipped with an action by a group $G$. Then $L$ can, with the choice of a character, be made into a $G'$-space, where $G':=\dC^*\times G$. The second main result is \cref{thm:restrictedTauHatDescription}, which not only gives sufficient and necessary non-vanishing criteria for the
system $\hat{\tau}(\rho,\overline{Y},\beta)$, but also describes this system, if it is non-zero, as a direct image of one of the modules $\cO_{L^*}^\beta$ introduced earlier in \cref{sec:OBeta}.

In \cref{sec:FormuleBeta}, we derive (via representation theoretic methods) a formula for the complex parameter value $\beta(\mathbf{e})\in \dC$
for which the tautological system $\tau(\rho, \hat{X}, \beta)$ is non-zero, at least in the case of a semisimple group $G$.
We also give a geometric interpretation of this formula and show that it is compatible with the general criterion of \cref{thm:restrictedTauHatDescription}.

In \cref{sec:tautologicalSystems}, we apply all the previous results in the case where the variety $X$ is a homogeneous space, and where the representation $\rho$ is in the dual of the space of sections of an equivariant line bundle on $X$. The affine cone of $X$ then takes the role of the $G'$-invariant space used in the definition of the tautological system $\tau(\rho,\hat{X},\beta)$. The main result is then \cref{theo:tau_is_MHM}, showing that if $\beta$ is such that $\tau(\rho,\hat{X},\beta)\neq 0$, then it underlies a pure complex Hodge module for $\beta(\mathbf{e})\notin \dZ$ and a rational mixed Hodge module for $\beta(\mathbf{e})\in \dZ_{>0}$. Moreover, we exhibit in \cref{cor:holRankProb} a functorial description of $\tau(\rho,\hat{X},\beta)$ as a
direct resp.\ as a proper direct image of a family of (complements of) hyperplane sections of $X$,  and in consequence solve the holonomic rank problem as stated in \cite{Bloch_HolRankProb} and \cite{HuangLianZhu} in this generalized situation. A major ingredient necessary for the formulation of this functorial description is to determine how the Fourier--Laplace transformation of $\tau(\rho,\hat{X},\beta)$ is related to its restriction to the complement of the origin. The answer to this question depends crucially on whether  $\beta(\mathbf{e})$  is integral or not; we treat the two cases separately in \cref{subsec:LocProp} and \cref{subsec:CoLocProp}.

\medskip

\emph{Notations:} Throughout, we work over $\C$. By \emph{variety}, we mean an integral scheme of finite type over $\C$. When we talk about \emph{points} on a variety, we mean closed points unless mentioned otherwise. Our convention for the projective space of a finite-dimensional vector space $V$ is $\P V := \Proj \Sym V^\vee$, i.e., $\P V$ parameterizes one-dimensional \emph{subspaces} of $V$.
For a smooth variety $X$, we let $\mathscr{D}_X$ be the sheaf of algebraic differential operators on $X$. If not mentioned otherwise, a $\mathscr{D}_X$-module is a
quasi-coherent $\cO_X$-module equipped with a left action
by $\mathscr{D}_X$. The category of such modules is denoted by $\textup{Mod}_{qc}(\mathscr{D}_X)$ and the corresponding bounded derived category by $D_{qc}^b(\mathscr{D}_X)$.
Similarly, let $\textup{Mod}_{h}(\mathscr{D}_X)$ and $D_{h}^b(\mathscr{D}_X)$ be the category of holonomic $\mathscr D_X$-modules and its corresponding bounded derived category, respectively.
Throughout, for a morphism $f \colon X \to Y$ between smooth varieties over $\C$, we denote by $f_+ \colon D_{qc}^b(\mathscr D_X) \to D_{qc}^b(\mathscr D_Y)$ and $f^+ \colon D_{qc}^b(\mathscr D_Y) \to D_{qc}^b(\mathscr D_X)$ the functors defined by
\[f_+ \mathscr M := Rf_*(\mathscr D_{Y \leftarrow X} \otimes^{\mathbb L}_{\mathscr D_X} \mathscr M)
\quad \text{and} \quad
f^+\mathscr N := \mathscr D_{X \to Y} \otimes^{\mathbb L}_{f^{-1} \mathscr D_Y} f^{-1} \mathscr N.\]
Moreover, we denote by
$$
    \bD \mathscr{M} := \omega^\vee_X\otimes_{\cO_X} R{\cH\!}om_{\cD_X}(\mathscr{M},\cD_X)[\dim(X)]
$$
the duality functor from $D^b_h(\cD_X)$ to itself; it respects
$\textup{Mod}_h(\cD_X)$. We then define the functors
$$
f_\dag:=\bD\circ f_+ \circ \bD
\quad\quad
\textup{and}
\quad\quad
f^\dag:=\bD\circ f^+ \circ \bD.
$$
For any variety $X$, let $\MHM(X)$ be the Abelian category of algebraic
($\dQ$-)mixed Hodge modules on $X$ (as defined in \cite{Saito1, SaitoMHM}) and $D^b \MHM(X)$ its bounded derived category. For any morphism $f\colon X\rightarrow Y$ the functors $f_+, f_\dag$ resp.\ $f^\dag[\dim(Y)-\dim(X)], f^+[\dim(X)-\dim(Y)]$ on $D^b_{h}(\mathscr{D}_X)$ resp.\
$D^b_{h}(\mathscr{D}_Y)$ lift to functors
$$
f_*, f_!\colon  D^b\MHM(X)\rightarrow D^b\MHM(Y)
\quad\quad
\textup{resp.}
\quad\quad
f^*, f^!\colon  D^b\MHM(Y)\rightarrow D^b\MHM(X).
$$
We also denote by $\bD$ the functor on $D^b \MHM(X)$ which lifts the above defined holonomic duality functor on $D^b_h(\cD_X)$.
Any object $\cM\in \MHM(X)$ is a tuple
$\cM=(\mathscr{M},F_\bullet, W_\bullet, K)$  where
$\mathscr{M}\in\textup{Mod}_h(\mathscr{D}_X)$ and $W_\bullet\mathscr{M}$ is its weight filtration. We denote by
$\HM(X,w)$ (or simply $\HM(X)$ if $w$ is clear from the context) the full subcategory of objects such that
$\gr_l(\mathscr{M})=0$ for all $l\neq w$; these are the pure Hodge modules of weight $w$.

We will need an extension of the notion of ($\dQ$-)mixed Hodge modules to the category of complex mixed Hodge modules. It can be constructed by first defining $\dR$-mixed Hodge modules,
see \cite[Section 13.5]{Mo12}. Then a filtered $\mathscr{D}$-module $(\mathscr{M},F_\bullet)$ is said to underlie a complex mixed Hodge module if it is a direct summand of an $\dR$-mixed Hodge module (\cite[Definition 3.2.1.]{DS13}). We denote the corresponding Abelian category by $\MHM(X,\dC)$, by
$D^b\MHM(X,\dC)$ its bounded derived category and by
$\HM(X,\dC,w)$ (or $\HM(X,\dC)$ for short) the category of pure complex Hodge modules of weight $w$. Many of the known constructions for $\dR$-mixed Hodge modules carry over to the categories $\MHM(X,\dC)$ and $\HM(X,\dC)$ since they are stable under taking direct summands.
The article \cite[Section 7.1 and Appendix A]{VilonenDavis} contains a more detailed discussion of complex Hodge modules.

For any variety $X$, write $a_X \colon X\rightarrow \{pt\}$ for the map to the point. We denote by
${^H\!}\dC_{pt}$ the trivial complex Hodge structure of dimension $1$. Then
$$
^{H\!}\underline{\dC}_{X} := a_X^* {^H\!}\dC_{pt}[\dim(X)],
$$
is a smooth (constant) Hodge module and indeed an object in $\MHM(X,\dC)$. Notice that our notation differs from the convention in \cite{Saito1, SaitoMHM}, where the ($\dQ$-)constant Hodge module of rank $1$ is denoted by ${^p}\dQ_X^H$.

We will further need a particular smooth (but non-constant) complex Hodge module on the one-dimensional torus $\dC^*$. Namely, for any $\beta\in \dR$,
we denote by $\cO^{\beta}_{\dC^*}$ the $\mathscr{D}_{\dC^*}$-module
$$
\cO^\beta_{\dC^*} = \mathscr{D}_{\dC^*} / \mathscr{D}_{\dC^*}(\partial_t t + \beta).
$$
We write ${^H\!}\underline{\dC}^\beta_{\dC^*}$ for the complex Hodge module with underlying $\mathscr{D}$-module equal to $\cO^\beta_{\dC^*}$ (placed in cohomological degree zero), and where
$\gr^F_p \cO^\beta_{\dC^*} = 0$ for $p \neq 0$ and $\gr^W_i \cO^\beta_{\dC^*}= 0$ for $i \neq 1$. Its corresponding perverse sheaf
is $\dV[1]$, where $\dV$ is the local system of rank $1$ on $\dC^*$ given by the monodromy with eigenvalue $e^{2\pi \sqrt{-1} \beta}$. Again, in the conventions of \cite{Saito1, SaitoMHM} this object would have been denoted by
${^p}\dC_X^{H,\beta}$.

\emph{Acknowledgements:} We thank Michel Brion for his assistance finding references for the parabolic version of Borel--Weil and for his explanations about line bundles on homogeneous varieties.
We thank Luis Narv\'aez Macarro for sharing some insights on Lie algebroids with us, especially for his help with the proof of \cref{lem:kernelGeneratedByVFs}.

We thank the anonymous referees for their careful reading of an earlier version of this paper and for their many helpful remarks.
In particular, one referee suggested that some results of this paper can alternatively be shown by using equivariant properties of the objects involved in a more systematic way. We comment on this in \cref{rmk:RemarkToReferee} and \cref{rmk:RemarkRefColoc}.

\section{Mixed Hodge modules on line bundles} \label{sec:OBeta}

As a preliminary result, we state and prove for the reader's convenience the following well known fact about the fundamental group of the complement of the zero section of a line bundle.
\begin{prop}\label{prop:PiEinsLStern}

Let $M$ be a simply connected complex manifold, i.e. $\pi_1(M) = \{e\}$. Let $ \pi: L \ra M$ be a holomorphic line bundle on $M$, and write
$$
c_1(L)=\sum_{i=1}^r \lambda_i e_i \in H^2(M,\Z)
$$
for some basis $e_1,\ldots,e_r$ of $H^2(M,\Z)$.
Denote by $L^*$ the complement of the zero section of $L$. Then $\pi_1(L^*)=\Z/k \Z$, where $k=\textup{gcd}(\lambda_1,\ldots,\lambda_r)$.
\end{prop}
\begin{proof}
We first notice that the assumption $\pi_1(M) = \{e\}$ and the universal coefficient theorem for cohomology implies that $H^2(M,\mathbb{Z})$ is free, so the statement of the proposition makes sense. Furthermore, $L^* \rightarrow M$ is a principal $\mathbb{C}^*$-bundle, hence by Milnor's construction \cite{Milnor} for the case $G= \mathbb{C}^*$ there is a classifying space $B:= BG$, a universal principal $G$-bundle $p: E:= EG \ra B$ and a map $\varphi: M \ra B$ so that $L^*$ is the pullback of $E$ along  $\varphi$.

Set $I :=[0,1]$, denote by $\ast$ the base point of $B$ and by $PB := \{ \gamma \in B^I : \gamma(0) = \ast\}$ the Moore path space over $B$. Since $B$ is path connected
we have the Moore path space fibration for $(B;*)$
\[
\Omega B \longrightarrow PB \overset{\rho}\longrightarrow B \qquad \rho(\gamma) = \gamma(1).
\]
and  analogously for $\pi: PE \ra E$. It can be shown \cite[Proposition 2.10]{FHT} that there is an action of $G \times_E PE$ on $PE$, making $ \pi:PE \rightarrow B$ a $G \times_E PE$-fibration. One gets a diagram of fibrations
$$
\begin{tikzcd}
E \arrow[ddrr, "p"']
  && PE \arrow[dd, "\pi"] \arrow[rr] \arrow[ll]
  && PB \arrow[ddll, "\rho"] \\ \\
  && B &&
\end{tikzcd}
$$
with fibers
$$
\begin{tikzcd}
G
  && G \times_E PE \arrow[ll, "\gamma"'] \arrow[rr, "\gamma'"]
  && \Omega B
\end{tikzcd}
$$
where $\gamma$ and $\gamma'$ are weak homotopy equivalences (see loc. cit.).

Pulling back this diagram along $\varphi$ one gets
$$
\begin{tikzcd}
L^* \arrow[ddrr] && M \times_B PE \arrow[ll] \arrow[rr] \arrow[dd] && H(\varphi) \arrow[ddll] \\ \\
&& M &
\end{tikzcd}
$$
where $H(\varphi)$ is by definition the homotopy fiber of $\varphi$ giving a fibration
\[
H(\varphi) \longrightarrow M \overset{\varphi}\longrightarrow B.
\]
Notice that $L^*$ and $H(\varphi)$ are weakly homotopy equivalent, in particular $\pi_1(L^*) \simeq \pi_1(H(\varphi))$.

Consider the long exact homotopy sequence  of the fibration above
\[
\pi_2(M) \lra \pi_2(B) \lra \pi_1(H(\varphi)) \lra \pi_1(M)= \{e\}.
\]

Since $\pi_1(M)$ and $\pi_1(B)$ are trivial we have $\pi_2(M) \simeq H_2(M,\mbz)$ and $\pi_2(B) \simeq H_2(B,\mbz) \simeq \mbz$. In particular, we get the exact sequence
\[
H_2(M,\mathbb{Z})_{\text{free}}\ \longrightarrow H_2(B,\mathbb{Z}) \lra \pi_1(H(\varphi)) \longrightarrow \pi_1(M) = \{e\}.
\]
By duality we get the map $\mathbb{Z} \cdot \theta \simeq H^2(B,\mathbb{Z}) \ra H^2(M,\mathbb{Z})$, where $\theta$ is a generator of $H^2(B,\mathbb{Z})$. Notice that this map is simply the pullback of cohomology classes along the classifying map $M \ra B$. In order to identify the image of $\theta$, we notice that there is a commutative diagram
$$
\begin{tikzcd}
&& B = B\mathbb{C}^* \\ \\
M \arrow[uurr] \arrow[ddrr] && \\ \\
&& BU(1) \arrow[uuuu]
\end{tikzcd}
$$
The map $BU(1) \rightarrow B$ is induced by the inclusion $U(1) \rightarrow \mathbb{C}^*$ and is a homotopy equivalence, in particular $H^*(B, \mathbb{Z}) \simeq H^*(BU(1),\mathbb{Z}) \simeq \mathbb{Z}[\theta]$ with $\deg(\theta) = 2$. The map $M \rightarrow BU(1)$ is the classifying space of the sphere bundle of $L$. By definiton of the Chern classes the pullback of $\theta$ along $M \rightarrow BU(1)$ is the first Chern class of $L$. We conclude that, with respect to the dual basis of $e_1,\ldots, e_r$, the map $H_2(M,\mathbb{Z})_{\text{free}} \rightarrow H_2(B,\mathbb{Z})$ is given by $\mathbb{Z}^r \overset{(\lambda_1, \ldots, \lambda_r)}\longrightarrow \mathbb{Z}$.

\end{proof}

From now on, let $X$ be a smooth complex quasi-projective variety and let $L \to X$ an algebraic line bundle. Unless noted otherwise, we work in the algebraic setting and denote the associated complex manifolds by $X^\mathrm{an}, L^\mathrm{an},\dots$. In particular, when assuming that $\pi_1(X^\mathrm{an})=\{e\}$, we can apply the above proposition to the case where $M:=X^\mathrm{an}$ and to the holomorphic line bundle $L^\mathrm{an}\rightarrow X^\mathrm{an}$. We therefore conclude that $\pi_1(L^{*,\mathrm{an}})=\Z/k\Z$. For the remainder of this paper, we will moreover assume that $L$ is a non-trivial line bundle on $X$ (in particular, we assume that $\textup{Pic}(X)\neq 0$). This implies in particular that the number $k$ obtained in the previous proposition is different from zero.

\begin{dfn}\label{def:Obeta-new}
Let $L\rightarrow X$ be as above, and let $k$ be as in \cref{prop:PiEinsLStern}.
Choose a rational number $\beta$ in $\frac{1}{k}\Z$. Consider the representation $\pi_1(L^{*,\mathrm{an}}) \rightarrow \dC^*$ given by sending $[1]\in\Z/k\Z \cong\pi_1(L^{*,\mathrm{an}})$ to the $k$-th root of unity
$e^{-2\pi i \beta}$. This defines a local system on $L^{*,\mathrm{an}}$ which we denote by $\underline{\dC}^\beta_{L^*}$. The corresponding $\cO_{L^*}$-module with integrable connection
(i.e., the corresponding smooth $\cD_{L^*}$-module) is denoted
by $\cO_{L^*}^\beta$. It underlies a complex smooth pure Hodge module on $L^*$ denoted by   $^{H\!}\underline{\dC}^\beta_{L^*}$.
\end{dfn}
Notice that it follows from this definition that $\cO_{L^*}^\beta\cong \cO_{L^*}^{\beta'}$ for $\beta-\beta'\in \Z$.

\begin{prop}\label{prop:ObetaExtProd}
  Locally, over an open subset $U \subseteq X$ trivializing the line bundle, $\O_{L^*}^\beta$ is isomorphic to the $\mathscr D_{\C^* \times U}$-module
  \[\mathscr D_{\C^*}/(\partial_t t + \beta) \boxtimes \O_U.\]
\end{prop}
\begin{proof}
    Write $L_U^*$ for the restriction of $L^*$ over the trivializing set $U\subseteq X$. Clearly, $U$ is also a trivializing open set for $L^*$, i.e., we have $L^*_{U}\cong \dC^*\times U$. First we note that
    $$
    \pi_1(L_U^{*,\mathrm{an}})\cong \pi_1(\dC^{*,\mathrm{an}}\times U^{\mathrm{an}})=\Z\times \pi_1(U^\mathrm{an}).
    $$
    Notice that although one can find trivializing sets of the analytic bundle $L^\mathrm{an}\rightarrow X^\mathrm{an}$ which are simply connected, this is not necessarily true for the set $U^\mathrm{an}$, which is the analytification of a (Zariski open) trivializing set of the algebraic bundle $L\rightarrow X$.

    The group homomorphism $\pi_1(L_U^{*,\mathrm{an}}) \to \pi_1(L^{*,\mathrm{an}})$ induced by the inclusion is given by
    $$
    \begin{array}{rcl}
         \pi_1(L_U^{*,\mathrm{an}}) = \Z\times \pi_1(U^\mathrm{an}) & \longrightarrow & \Z/k\Z = \pi_1(L^{*,\mathrm{an}}) \\
         (l,\gamma) & \longmapsto & \overline{l}
    \end{array}
    $$
    since $\pi_1(X^\mathrm{an})=\{e\}$.
    Now, the restriction by $L_U^{*,\mathrm{an}} \hookrightarrow L^{*,\mathrm{an}}$ of $\underline{\dC}^\beta_{L^*}$ is given by the representation $\pi_1(L_U^{*,\mathrm{an}})\rightarrow \dC^*$ obtained by composing the map $\pi_1(L_U^{*,\mathrm{an}})  \rightarrow  \pi_1(L^{*,\mathrm{an}})$ with the given representation $\pi_1(L^{*,\mathrm{an}})\rightarrow \dC^*$ defining $\underline{\dC}^\beta_{L^*}$. Hence, it sends $(1,\gamma)$ to $e^{-2\pi i \beta}$. Therefore, the restriction of $\underline{\dC}^\beta_{L^*}$ to $L^{*,\mathrm{an}}_U$ is isomorphic to $\underline{\dC}^\beta_{\C^*}\boxtimes \underline{\dC}_{U^\mathrm{an}}$, and consequently the restriction of the algebraic $\cD_{L^*}$-module $\cO_{L^*}^\beta$ to $L_U^*$ is given as stated above.
\end{proof}
The following fact about the holonomic dual of $\cO_{L^*}^\beta$ is obvious from the definition, since for smooth objects, the holonomic dual coincides with the dual bundle (with dual connection).
\begin{lem}\label{lem:OLbetaDual}
We have $\bD \cO_{L^*}^\beta\cong \cO_{L^*}^{-\beta}$
for all $\beta\in \frac{1}{k}\Z$.
\end{lem}

The next step is to consider extensions of the $\cD_{L^*}$-module $\cO^\beta_{L^*}$ (resp.\ of the corresponding Hodge module $^{H\!}\underline{\dC}^\beta_{L^*}$) over the zero section of $L\rightarrow X$. For this, we choose a covering
$X=\bigcup_{i\in I} U_i$ by Zariski open subsets over each of which the bundle $L$ (as well as its restriction $L^*$) trivializes. We write $L_{U_i}$ and (as above) $L^*_{U_i}$ for the restriction of the bundle $L$ and that of $L^*$ over the open set $U_i$. We denote by $j_i:L_{U_i}\hookrightarrow L$ resp.\ by
$\tilde{j}_i:L^*_{U_i}\hookrightarrow L^*$ be the canonical open embeddings.
By shrinking $U_i$ if necessary (so that it becomes the complement of a divisor in $X$, and so will be $L_{U_i}$ in $L$, and $L^*_{U_i}$ in $L^*$), it will be convenient to assume that both $j_i$ and $\tilde{j}_i$ are affine maps, in particular, the functors $j_{i,\star}$ and $\tilde{j}_{i,\star}$ are exact for $\star\in\{+,\dag\}$.

Moreover, we write $j_L:L^*\hookrightarrow L$ for the open embedding of the complement of the zero section into $L$. We then have the following cartesian diagram of canonical open embeddings:
\begin{equation}\label{diag:ExtObetaOpenSubs}
    \begin{tikzcd}
        L^*_{U_i} \ar[hook]{rr}{\tilde{j}_i} \ar[hook]{dd}{j_{i,L}} && L^* \ar[hook]{dd}{j_L} \\ \\
        L_{U_i} \ar[hook]{rr}{j_i} && L.
    \end{tikzcd}
\end{equation}

By construction of the module $\cO^\beta_{L^*}$, for any $i\in I$, we have an isomorphism
(depending on the choice of a trivialization of $L$ resp.\  of  $L^*$ over $U_i$)
$$
\psi_i: \cO^\beta_{\dC^*}\boxtimes\cO_{U_i} \stackrel{\cong}{\longrightarrow} \tilde{j}_i^+ \cO_{L^*}^\beta.
$$
Then we have the following statement
\begin{prop}\label{prop:CBetaHodgeExt}
\begin{enumerate}
    \item
        The $\mathscr{D}_{L}$-modules $j_{L,+} \cO^\beta_{L^{*}}$ resp.\ $j_{L,\dag} \cO^\beta_{L^{*}}$ underlie the mixed Hodge modules $j_{L,*} {^{H\!}}\underline{\dC}^\beta_{L^*}$
        resp.\ $j_{L,!} {^{H\!}}\underline{\dC}^\beta_{L^*}$ on $L$.

    \item
        If
        $\beta\notin\dZ$, then
        $$
        j_{L,*} {^{H\!}}\underline{\dC}^\beta_{L^*}  \cong j_{L,!} {^{H\!}}\underline{\dC}^\beta_{L^*}\cong j_{L,!*} {^{H\!}}\underline{\dC}^\beta_{L^*},
        $$
        which is pure of weight $\dim(X)+1$.

    \item
                For any $\beta\in\frac{1}{k}\dZ$ the following isomorphisms hold in $\MHM(L)$
        $$
        j_{L,*} {^{H\!}}\underline{\dC}^\beta_{L^*}  \cong
        j_{i,!*} j_i^{*} j_{L,*} {^{H\!}}\underline{\dC}^\beta_{L^*},
        \quad\quad\textup{and}\quad\quad
        j_{L,!} {^{H\!}}\underline{\dC}^\beta_{L^*}\cong
        j_{i,!*} j_i^{*} j_{L,!} {^{H\!}}\underline{\dC}^\beta_{L^*},
        $$
        i.e., these mixed Hodge modules are the minimal extensions of their restrictions to open sets $L_{U_i}\subset L$.
\end{enumerate}
\end{prop}
\begin{proof}
\begin{enumerate}
    \item
        This is obvious.

    \item

        Notice that we have a well defined morphism
        $$
        j_{L,\dag} \cO_{L^*}^\beta \longrightarrow j_{L,+} \cO_{L^*}^\beta.
        $$
        It suffices to show that if $\beta\notin\Z$, then this is an isomorphism in $\textup{Mod}(\cD_L)$. This is a local statement, therefore, we can reduce the proof to show that for any $i\in I$, the morphism
        $$
        j_i^+j_{L,\dag} \cO_{L^*}^\beta \longrightarrow j_i^+j_{L,+} \cO_{L^*}^\beta.
        $$
        is an isomorphism. Since $j_i^+\cong j_i^\dag$, this is equivalent by base change (see diagram \eqref{diag:ExtObetaOpenSubs}) to show that
        $$
        (j_{i,L})_\dag\tilde{j}_i^+ \cO_{L^*}^\beta \longrightarrow
        (j_{i,L})_+\tilde{j}_i^+ \cO_{L^*}^\beta
        $$
        is an isomorphism. Since $\tilde{j}_i^+ \cO_{L^*}^\beta$ is isomorphic to
        $\cO^\beta_{\dC^*}\boxtimes \cO_{U_i}$ via $\psi_i$, and since the functors $(j_{i,L})_+$ resp.\ $(j_{i,L})_\dag$ correspond to $(j_\dC\times\id_{U_i})_+$ resp.\
        $(j_\dC\times\id_{U_i})_\dag$ under this isomorphism, the statement reduces to the well-known fact that
        $j_{\dC,+}\cO_{\dC^*}^\beta\cong j_{\dC,\dag}\cO_{\dC^*}^\beta\cong j_{\dC,\dag+}\cO^\beta_{\dC^*}$ for $\beta\notin\dZ$.
    \item

        Again it suffices to show the statement on the level of $\cD_L$-modules, i.e., we need to show that for all $i\in I$ we have
        \begin{equation}\label{eq:MinExt}
        j_{L,+} \cO^\beta_{L^{*}} \cong j_{i,\dag+} j_i^+ j_{L,+} \cO^\beta_{L^{*}}
        \quad\quad\textup{and}\quad\quad
        j_{L,\dag} \cO^\beta_{L^{*}} \cong
        j_{i,\dag+} j_i^+ j_{L,\dag} \cO^\beta_{L^{*}}.
        \end{equation}

        Let us prove the first statement concerning the extension $j_{L,+}\cO_{L^*}^\beta$, the proof of the second one is similar. Fix $i\in I$. Then for any $r\in I\backslash\{i\}$, we obtain an isomorphism
        \begin{equation}\label{eq:IsoDoubleOpenSets}
            j_r^+ j_{L,+} \cO^\beta_{L^{*}} \cong j_r^+ j_{i,\dag+} j_i^+ j_{L,+} \cO^\beta_{L^{*}}
        \end{equation}
        by an argument similar to point 2. above. Namely, in order to show \eqref{eq:IsoDoubleOpenSets}, it suffices by base change (notice that all functors involved are exact, so the base change property also holds for the intermediate extension) to prove
        $$
            j_{r,L,+} \cO_{L^*_{U_r}}^\beta \cong
            j_{ir,\dag+}
            j_{ir}^+
            j_{r,L,+} \cO^\beta_{L_{U_r}^{*}}
        $$
        where now $j_{r,L}:L^*_{U_r}\hookrightarrow L_{U_r}$ and where
        $j_{ir}:L_{U_i\cap U_r}\hookrightarrow L_{U_r}$. However, since both $L^*$ and $L$ trivialize over $U_r$ and since the module $\cO_{L_{U_r}^*}^\beta$ resp.\ the extension $(j_{r,L})_+\cO_{L_{U_r}^*}^\beta$ is isomorphic  to the exterior product  $\cO_{\dC^*}^\beta\boxtimes \cO_{U_r}$ resp.\ to $j_{\dC,+}\cO_{\dC^*}^\beta\boxtimes \cO_{U_r}$ (and since obviously
        $\cO_{U_r}$ is the minimal extension of $\cO_{U_i\cap U_r}$), we obtain the existence of the isomorphism \eqref{eq:IsoDoubleOpenSets}. Now it is a tedious but straightforward check that these isomorphisms are compatible on intersections of trivializing open sets, hence they yield the desired isomorphism
        $$
            j_{L,+} \cO^\beta_{L^{*}} \cong j_{i,\dag+} j_i^+ j_{L,+} \cO^\beta_{L^{*}}.
        $$
\end{enumerate}
\end{proof}

\section{Fourier--Laplace transformation} \label{sec:FLOnVecBundles}

The purpose of this section is twofold: First we recall a few basic properties of general Fourier-Laplace transformations on (not necessarily trivial) vector bundles. We then apply these constructions to study a (complex of) $\mathscr{D}$-module(s) that generically computes cohomology groups of hyperplane sections of projective varieties. These results are used later in \cref{sec:tautologicalSystems} for the special case of
homogeneous spaces and their corresponding tautological systems.

\subsection{Fourier--Laplace transformation on vector bundles} \label{ssec:FLBasics}

\begin{dfn}
  Given a vector bundle $E \to X$ on a smooth variety $X$, we consider the canonical projections $p_1 \colon E \times_X E^\vee \to E$ and $p_2 \colon E \times_X E^\vee \to E^\vee$. Let $\alpha \colon E \times_X E^\vee \to \C \times X \to \C$ be the natural pairing and denote $\mathscr K := \alpha^+(\mathscr D_\C/\mathscr D_\C(\partial_t+1))$.
  The \defstyle{Fourier--Laplace transformation} is the functor $\FL_X^E \colon D_{qc}^b(\mathscr D_E) \to D_{qc}^b(\mathscr D_{E^\vee})$ given as
  \[\FL_X^E(\mathscr M) := p_{2,+} (p_1^+ \mathscr M \otimes^{\mathbb L}_{\O_{E\!\times\!E^\vee}} \mathscr K).\]
\end{dfn}

We first note a well-known fact concerning the behavior of the Fourier-Laplace transformation with respect to the holonomic duality functor.
\begin{lem}\label{lem:DualFourier}
We have
$$
c^+ \circ \FL_X^E \circ \bD \cong \bD \circ \FL_X^E
$$
as functors from $D^b_c(\cD_E) \to D^b_c(\cD_{E^\vee})$,
where $c:E^\vee \to E^\vee$ is the
automorphism given by fiberwise negation.
\end{lem}
\begin{proof}
This can be shown exactly as in \cite[Corollaire~2.2.2.1., 4)]{Daia} (see especially loc.cit., Proposition~2.2.3.2)
\end{proof}

We proceed with the following two basic properties that follow rather directly from the projection formula and base change.

\begin{lem} \label{lem:FLBundleMorphism}
  Let $\varphi \colon E \to F$ be a morphism of vector bundles over $X$ and denote by $\varphi^\vee \colon F^\vee \to E^\vee$ the induced morphism of dual vector bundles. Then for $\star\in\{+,\dag\}$ we have
  \[\FL_X^F \circ \varphi_\star \cong \varphi^{\vee,\star} \circ \FL_X^E\]
  as functors $D_{c}^b(\mathscr D_E) \to D_{c}^b(\mathscr D_{F^\vee})$.
\end{lem}

\begin{proof}
  We only show the case $\star=+$ from which the case $\star=\dag$ follows directly using \cref{lem:DualFourier}.
  We denote $\mathscr K^E := (\alpha^E)^+(\mathscr D_\C/\mathscr D_\C(\partial_t+1))$ and $\mathscr K^F := (\alpha^F)^+(\mathscr D_\C/\mathscr D_\C(\partial_t+1))$, where $\alpha^E \colon E \times_X E^\vee \to \C$ and $\alpha^F \colon F \times_X F^\vee \to \C$ are the natural pairings. Moreover, denote by $q_1 \colon E \times_X F^\vee \to E$ and $q_2 \colon E \times_X F^\vee \to F^\vee$ the projections onto the first and second factor. Consider the commutative diagram
  \[\begin{tikzcd}[column sep=huge]
    &
    E \ar{r}{\varphi}
    & F
    \\
    E \times_X E^\vee \ar{ur}{p^E_1} \ar{d}{p^E_2}
    & E \times_X F^\vee \ar{u}[swap]{q_1} \ar{d}[swap]{q_2} \ar{l}{\id_E \times \varphi^\vee} \ar{r}{\varphi \times \id_{F^\vee}}
    & F \times_X F^\vee \ar{dl}{p_2^F} \ar{u}{p_1^F}
    \\
    E^\vee
    & F^\vee \ar{l}{\varphi^\vee}
    & \phantom{F \times_X F^\vee},
  \end{tikzcd}\]
  whose squares are cartesian. For every $\mathscr M \in D^b_{c}(\mathscr D_E)$, we have:
  \begin{align*}
    &\FL_X^F(\varphi_+ \mathscr M) & \\
    &= p_{2,+}^F (p_1^{F,+} \varphi_+ \mathscr M \otimes^{\mathbb L}_{\O_{F\!\times\!F^\vee}} \mathscr K^F) & \\
    &\cong p_{2,+}^F \big((\varphi \times \id_{F^\vee})_+ q_1^+ \mathscr M \otimes^{\mathbb L}_{\O_{F\!\times\!F^\vee}} \mathscr K^F\big) & \text{\small (base change)}\\
    &\cong p_{2,+}^F (\varphi \times \id_{F^\vee})_+ \big(q_1^+ \mathscr M \otimes^{\mathbb L}_{\O_{E\!\times\!F^\vee}} (\varphi \times \id_{F^\vee})^+ \mathscr K^F\big) & \text{\small (projection formula)}\\
    &\cong q_{2,+} \big(q_1^+ \mathscr M \otimes^{\mathbb L}_{\O_{E\!\times\!F^\vee}} (\varphi \times \id_{F^\vee})^+ \mathscr K^F\big) & \text{\small ($q_2 = p_2^F \circ (\varphi \times \id_{F^\vee})$)}\\
    &\cong q_{2,+} \big(q_1^+ \mathscr M \otimes^{\mathbb L}_{\O_{E\!\times\!F^\vee}} (\id_E \times \varphi^\vee)^+ \mathscr K^E\big) & \text{\small ($\alpha^E \circ (\id_E \times \varphi^\vee) = \alpha^F \circ (\varphi \times \id_{F^\vee})$)}\\
    &\cong q_{2,+} \big((\id_E \times \varphi^\vee)^+ p_1^{E,+} \mathscr M \otimes^{\mathbb L}_{\O_{E\!\times\!F^\vee}} (\id_E \times \varphi^\vee)^+ \mathscr K^E\big) & \text{\small ($q_1 = p_1^E \circ (\id_E \times \varphi^\vee)$)}\\
    &\cong \varphi^{\vee,+} p^E_{2,+} (p_1^{E,+} \mathscr M \otimes^{\mathbb L}_{\O_{E\!\times\!E^\vee}} \mathscr K^E) & \text{\small (base change)}\\
    &= \varphi^{\vee,+} \FL_X^E(\mathscr M). & \qedhere
  \end{align*}
\end{proof}

\begin{lem} \label{lem:FLCartesian}

  Consider a cartesian square
  \[\begin{tikzcd}
    E \ar{r}{g} \ar{d} \ar[phantom]{dr}{\times} & F \ar{d} \\
    X \ar{r} & Y,
  \end{tikzcd}\]
  where the vertical arrows are vector bundles over smooth varieties. Denote by $g^\vee \colon E^\vee \to F^\vee$ the corresponding morphism of dual vector bundles. Then for
  $\star=\{+,\dag\}$ we have
  \[\FL_Y^F \circ g_\star \cong g^\vee_\star \circ \FL_X^E \qquad \text{and} \qquad
  \FL_X^E \circ g^\star \cong g^{\vee,\star} \circ \FL_Y^F\]
  as functors $D_{c}^b(\mathscr D_E) \to D_{c}^b(\mathscr D_{F^\vee})$ and $D_{c}^b(\mathscr D_F) \to D_{c}^b(\mathscr D_{E^\vee})$, respectively.
\end{lem}

\begin{proof}
  Again we restrict to the case
  $\star=+$, and invoke duality to deduce the corresponding statements for $\star=\dag$.
  We use notations as in the proof of \cref{lem:FLBundleMorphism}. Note that we have the following commutative diagram with cartesian squares:
  \[\begin{tikzcd}
    E \ar{d}{g} & E \times_X E^\vee \ar{l}[swap]{p_1^E} \ar{r}{p_2^E} \ar{d}{g \times g^\vee} & E^\vee \ar{d}{g^\vee} \\
    F & F \times_Y F^\vee \ar{l}[swap]{p_1^F} \ar{r}{p_2^F} & F^\vee. \\
  \end{tikzcd}\]
  For every $\mathscr M \in D_{qc}^b(\mathscr D_E)$, we have:
  \begin{align*}
    &\FL_Y^F(g_+ \mathscr M) & \\
    &= p_{2,+}^F (p_1^{F,+} g_+ \mathscr M \otimes^{\mathbb L}_{\O_{F\!\times_{\!Y}\!F^\vee}} \mathscr K^F) & \\
    &\cong p_{2,+}^F \big((g \times g^\vee)_+ p_1^{E,+} \mathscr M \otimes^{\mathbb L}_{\O_{F\!\times_{\!Y}\!F^\vee}} \mathscr K^F\big) & \text{\small (base change)} \\
    &\cong p_{2,+}^F (g \times g^\vee)_+ \big(p_1^{E,+} \mathscr M \otimes^{\mathbb L}_{\O_{E\!\times_{\!X}\!E^\vee}} (g \times g^\vee)^+\mathscr K^F\big) & \text{\small (projection formula)} \\
    &\cong g^\vee_+ p_{2,+}^E \big(p_1^{E,+} \mathscr M \otimes^{\mathbb L}_{\O_{E\!\times_{\!X}\!E^\vee}} (g \times g^\vee)^+\mathscr K^F\big) & \text{\small ($p_2^F \circ (g \times g^\vee) = g^\vee \circ p_2^E$)} \\
    &\cong g^\vee_+ p_{2,+}^E (p_1^{E,+} \mathscr M \otimes^{\mathbb L}_{\O_{E\!\times_{\!X}\!E^\vee}} \mathscr K^E) & \text{\small ($\alpha^E = \alpha^F \circ (g \times g^\vee)$)} \\
    &= g_+^\vee \FL_X^E(\mathscr M). &
  \end{align*}
  Similarly, for $\mathscr N \in D_{qc}^b(\mathscr D_F)$, we get:
  \begin{align*}
    &\FL_X^E(g^+ \mathscr N) & \\
    &= p_{2,+}^E (p_1^{E,+} g^+ \mathscr N \otimes^{\mathbb L}_{\O_{E\!\times_{\!X}\!E^\vee}} \mathscr K^E) & \\
    &\cong p_{2,+}^E \big((g \times g^\vee)^+ p_1^{F,+} \mathscr N \otimes^{\mathbb L}_{\O_{E\!\times_{\!X}\!E^\vee}} \mathscr K^E\big) &\text{\small ($g \circ p_1^E = p_1^F \circ (g \times g^\vee)$)} \\
    &\cong p_{2,+}^E \big((g \times g^\vee)^+ p_1^{F,+} \mathscr N \otimes^{\mathbb L}_{\O_{E\!\times_{\!X}\!E^\vee}} (g \times g^\vee)^+\mathscr K^F\big) &\text{\small ($\alpha^E = \alpha^F \circ (g \times g^\vee)$)} \\
    &\cong g^{\vee,+} p_{2,+}^F \big(p_1^{F,+} \mathscr N \otimes^{\mathbb L}_{\O_{F\!\times_{\!Y}\!F^\vee}} \mathscr K^F\big) &\text{\small (base change)} \\
    &= g^{\vee,+} \FL_Y^F(\mathscr N). & \qedhere
  \end{align*}
\end{proof}

In the following, we wish to relate Fourier--Laplace transforms on
vector bund\-les with classical Fourier--Laplace transforms on a finite-dimensional vector space  (which is the special case of a vector bundle over a point). For this, we consider the following situation: Let $\pi \colon E \to X$ be a vector bundle on a smooth variety and denote by $\mathcal E$ its sheaf of sections, i.e., $E = \Tot(\mathcal E) := \Spec_{\O_X} \Sym^\bullet \mathcal E^\vee$. Let $W \subseteq \Gamma(X,\mathcal E)$ be a non-zero finite-dimensional vector space of global sections of $E$ and let $V := W^\vee$ be its dual vector space. There are natural bundle morphisms $ev \colon W \times X \to E$ and $ev^\vee \colon E^\vee \to V \times X$, where $E^\vee$ denotes the dual vector bundle to $E$.

\begin{prop} \label{prop:FLTransfer}
  Let $W$ be a finite-dimensional space of global sections of a vector bundle $E \to X$ on a smooth variety. Let $V$ denote its dual vector space. If $a_V \colon V \times X \to V$ and $a_W \colon W \times X \to W$ denote the projections onto the first factors, we have
  \[\FL^V(a_{V,+}ev^\vee_+ \mathscr M) \cong a_{W,+} ev^+ \FL_X^{E^\vee}(\mathscr M)\]
  for all $\mathscr M \in D_{qc}^b(\mathscr D_{E^\vee})$.
\end{prop}
\begin{proof}
  The claim follows from \cref{lem:FLBundleMorphism} and \cref{lem:FLCartesian} considering the diagram
  \[\begin{tikzcd}
    E^\vee \ar{r}{ev^\vee} \ar{dr} & V \times X \ar{d} \ar{r}{a_V} & V \ar{d} \\
    & X \ar[phantom]{ur}{\times} \ar{r} & \Spec \C.
  \end{tikzcd}\]
\end{proof}

\subsection{Fourier--Laplace transform of extensions of \texorpdfstring{$\O_{L^{*}}^\beta$}{O\_\{L\textasciicircum*\}\textasciicircum\textbackslash beta}}

We now determine the Fourier--Laplace transform of the $\mathscr{D}_{L^*}$-modules $\O_{L^*}^\beta$ defined in \cref{sec:OBeta}, where $L^*$ is the complement of the zero section of a line bundle $\pi \colon L \to X$.
We denote by $j_L \colon L^* \hookrightarrow L$ and $j_{L^\vee} \colon L^{\vee,*} \hookrightarrow L^\vee$  the open embeddings from the complements of the zero section into $L$ and $L^\vee$, respectively.

\begin{prop} \label{prop:FLofObeta}
  Let $\beta \in \C$ with $k\beta \in \Z$. Then
  \[\FL_X^{L}(j_{L,+} \O_{L^{*}}^{\beta}) \cong j_{L^\vee,\dag} \O_{L^{\vee,*}}^{-\beta}.\]
\end{prop}

\begin{proof}
We have trivially
$$
    \cO^\beta_{L^{*}} \cong \tilde{j}_{i,\dag+} \tilde{j}_i^+ \cO^\beta_{L^{*}},
$$
(recall diagram \eqref{diag:ExtObetaOpenSubs} for the maps involved in this isomorphism) since $\cO_{L^*}^\beta$ is a smooth $\cD_{L^*}$-module. The commutativity of
diagram \eqref{diag:ExtObetaOpenSubs} yields
$$
j_{L,+}\tilde{j}_{i,\dag+} \tilde{j}_i^+\cO_{L^*}^\beta \cong j_{i,\dag+} j_{i,L,+} \tilde{j}_i^+\cO_{L^*}^\beta
\quad\quad
\textup{and}
\quad\quad
j_{L,\dag}\tilde{j}_{i,\dag+} \tilde{j}_i^+\cO_{L^*}^\beta \cong j_{i,\dag+} j_{i,L,\dag} \tilde{j}_i^+\cO_{L^*}^\beta,
$$
therefore, we obtain
$$
j_{L,\ast} \cO_{L^*}^\beta \cong j_{i,\dag+} j_{i,L,\ast} \tilde{j}_i^+\cO_{L^*}^\beta
$$
for $\ast\in\{+,\dag\}$.
Similar statements hold for $\cO_{L^{\vee,*}}^{-\beta}$ and its extensions to $\cD_{L^{\vee,*}}$-modules (they involve the canonical open embeddings $j_i^\vee:L^\vee_{U_i}\hookrightarrow L^\vee$ on the dual bundle).

By \cref{lem:FLCartesian}, we have $\FL^L_X \circ j_{i,\ast}=j^\vee_{i,\ast}\circ\FL^{L_{U_i}}_{U_i}$ for $\ast\in\{+,\dag\}$. Moreover, since
all of the four functors $\FL^L_X, j_{i,\ast}, j^\vee_{i,\ast}, \FL^{L_{U_i}}_{U_i}$
are exact, we also obtain $\FL^L_X \circ j_{i,\dag+}=j^\vee_{i,\dag+}\circ\FL^{L_{U_i}}_{U_i}$.
It follows that
$$
\begin{array}{rcl}
\D \FL^L_X(j_{L,+} \O_{L^{*}}^{\beta})
&\cong &
\D  j^\vee_{i,\dag+} \FL^{L_{U_i}}_{U_i} \left( \tilde{j}_{L,+} \tilde{j}_{i}^+ \cO^\beta_{L^{*}} \right)
\\ \\
&\cong &
\D  j^\vee_{i,\dag+} \FL^{L_{U_i}}_{U_i} \left( \tilde{j}_{L,+} \left(\cO_{\dC^*}^\beta\boxtimes\cO_{U_i}\right) \right)
\\ \\
&\cong&\D
j^\vee_{i,\dag+} \FL^{L_{U_i}}_{U_i} \left(j_{\dC,+}\cO_{\dC^*}^\beta\boxtimes\cO_{U_i}\right)
\\ \\
&\cong&\D
j^\vee_{i,\dag+} \left( \FL^{\dC} (j_{\dC,+}\cO_{\dC^*}^\beta)\boxtimes\cO_{U_i}\right)
\\ \\
&\cong&\D
j^\vee_{i,\dag+} \left( j_{\dC,\dag}\cO_{\dC^*}^{-\beta}\boxtimes\cO_{U_i}\right)
\\ \\
&\cong&\D
j^\vee_{i,\dag+} \left( \tilde{j}_{L^\vee,\dag}\left(\cO_{\dC^*}^{-\beta}\otimes \cO_{U_i}\right)\right)
\\ \\
&\cong&\D
j^\vee_{i,\dag+} \left( \tilde{j}_{L^\vee,\dag} \tilde{j}_i^{\vee,+}\cO^{-\beta}_{L^{\vee,*}}\right)
\\ \\
&\cong&\D
j_{L^\vee,\dag}\cO^{-\beta}_{L^{\vee,*}}.
\end{array}
$$
\end{proof}

\begin{cor} \label{cor:FLofOBetaIsMHM}
  Let $k \in \Z$ and let $\beta \in \R$ with $k\beta \in \Z$. Then the Fourier--Laplace transform on $L$ of the $\mathscr{D}_{L}$-module $j_{L,+}\O_{L^*}^\beta$ can be equipped with the structure of a complex mixed Hodge module which is pure of weight $\dim(X)+1$ if $\beta\notin \dZ$.
\end{cor}

\begin{proof}
We have just seen in the previous \cref{prop:FLofObeta} that
$$
\FL_X^{L}(j_{L,+} \O_{L^*}^{\beta}) \cong j_{L^\vee,\dag} \O_{L^{\vee,*}}^{-\beta}.
$$
On the other hand, we know by \cref{prop:CBetaHodgeExt} that $j_{L^\vee,\dag} \O_{L^{\vee,*}}^{-\beta}$ underlies the  the object
$$
j_{L^\vee,!} \, ^{H\!}\underline{\dC}^{-\beta}_{L^{\vee,*}} \in \MHM(L^\vee,\dC),
$$
and that it is pure if $\beta\notin\dZ$.
\end{proof}

\subsection{Twisted cohomology of hyperplane sections}\label{sec:TwistCOhomHypSec}

In this subsection, we describe a complex of $\cD$-modules that generically computes certain twisted cohomologies of hyperplane sections of our variety $X$ (resp.\ the complement of those). We show that it underlies an object in the derived category of mixed Hodge modules. In the more specific situation studied later in \cref{sec:tautologicalSystems}, when $X$ arises as a homogeneous space, these $\cD$-modules will appear as tautological systems.

With the notations from before, we fix a non-zero finite-dimensional subspace $W$ of $\Gamma(X,\Ell)$.
Let $V := W^\vee$ denote its dual vector space. The linear system $W$ on $X$ defines a rational map $g \colon X \dashrightarrow \P V$. The natural evaluation morphism
\begin{equation}\label{eq:evMorph}
ev \colon W \times X \to L,
\end{equation}
is a morphism of vector bundles over $X$ and it induces a dual bundle morphism
\[ev^\vee \colon L^\vee \to V \times X.\]
The following diagram commutes:
\[\begin{tikzcd}
  L^\vee \ar{r}{ev^\vee} \ar{d}[swap]{\pi^\vee} & V \times X \ar[dashed]{d} \\
  X \ar[dashed]{r}{g \times \id_X} &  \P V \times X.
\end{tikzcd}\]

If the linear system $W$ is base-point-free, then $g \colon X \to \P V$ is a morphism and $ev^\vee$ restricts to a morphism
\[\widetilde{ev}^{\vee} \colon L^{\vee,*} \to (V \setminus \{0\}) \times X\]
of complements of zero sections. In this case, we have the following commutative diagram:
\[\begin{tikzcd}
  L^\vee \ar{r}{ev^\vee} & V \times X \ar{r}{a_V} \ar[phantom]{dr}{\times} & V\\
  L^{\vee,*}   \ar{r}{\widetilde{ev}^{\vee}} \ar{d} \ar[hook]{u}{j_{L^\vee}} & (V \setminus \{0\}) \times X \ar{d} \ar[hook]{u}{j \times \id_X} \ar{r}
  \ar[phantom]{dr}{\times}
  & V \setminus \{0\} \ar{d} \ar[hook]{u}{j} \\
  X \ar{r}{g \times \id_X} & \P V \times X \ar{r} &  \P V.&
  \end{tikzcd}\]
If, moreover, the linear system $W$ separates points and tangent directions (in particular, $\Ell$ is very ample in this case), then $g \colon X \to \P V$ is a locally closed embedding. In this case, $L^{\vee,*}$ is isomorphic to $\hat{X} \setminus \{0\}$, where $\hat{X} \subseteq V$ is the affine cone over $g(X) \subseteq \P V$, and $L^\vee$ is the blow-up of $\hat X$ in the origin: $L^\vee \cong \operatorname{Bl}_{\{0\}}\hat{X}$.
We denote further by $\cY:=ev^{-1}(0)$ the inverse image of the zero section of $L$, by $\cU:=(W \times X)\backslash \cY$ its complement, and we write $a_\cY:\cY\rightarrow W$ resp.\ $a_\cU:\cU\rightarrow W$ for the restrictions of the
projection $a_W \colon W\times X \to W$ to $\cY$ resp.\ to $\cU$.
\begin{prop} \label{prop:directImageOfOBetaMHM}
  Assume $L$ to be very ample and let $W \subseteq H^0(X,\Ell)$ be a finite-dimensional linear system defining a locally closed embedding $g \colon X \hookrightarrow \P V$, where $V := W^\vee$.  Let $\hat\iota \colon L^{\vee,*} \cong \hat{X} \setminus \{0\} \hookrightarrow V$ denote the locally closed embedding of the punctured affine cone over $X$ into $V$. Then we have the following.
  \begin{enumerate}
      \item
          For all $\beta \in \C$ with $k\beta \in \Z$, the complexes of $\mathscr D_W$-modules
          \[
            \FL^V(\hat\iota_+ \O_{L^{\vee,*}}^\beta)
            \quad\quad\textup{and}\quad\quad
            \FL^V(\hat\iota_\dag \O_{L^{\vee,*}}^{-\beta})
          \]
          underlie elements of $D^b\MHM(W,\dC)$
                    that we denote by $\Mh$ and by $\MhD$, respectively. We have
          $$
            \Mh \cong {^{H,*}\!\!}\cM_L^{\beta+\ell}
            \quad\quad\textup{and}\quad\quad
            \MhD \cong {^{H,!}\!\!}\cM_L^{-\beta+\ell}
          $$
          for any $\ell\in \dZ$.
    \item
        For $\beta\in \dZ$, the complexes $\FL^V(\hat\iota_+ \O_{L^{\vee,*}}^\beta)$ and $\FL^V(\hat\iota_\dag \O_{L^{\vee,*}}^{-\beta})$ underlie elements in $D^b\MHM(W)$
        that we denote unambiguously by ${^{H,*}\!\!}\cM_L$ resp.\ by ${^{H,!}\!\!}\cM_L$.

    \item
      For $\beta \notin \dZ$, we have an isomorphism
      $\Mh \cong {^{H,!}\!\!}\cM_L^{\beta}$. If $X$ is projective, then the cohomology modules
      $H^i(\Mh)$ are pure Hodge modules of weight $\dim(X)+\dim(W)+i$.

    \item     Let $\beta\in\dZ$ and assume again that $X$ is projective. Then for any $k\in \dZ$,
    there exist morphisms in the abelian category of mixed Hodge modules
    \[
        H^k(a_{\mcy,*} \, ^{H\!}\underline{\dC}_{\mcy}) \lra H^{k} \left({^{H,*}\!\!}\cM_L\right) \quad \text{resp.} \quad H^{k} \left({^{H,!}\!\!}\cM_L\right) \lra H^k(a_{\mcy,*} \, ^{H\!}\underline{\dC}_{\mcy})(-1)
    \]
    with constant kernel of weight $k+ \dim X + \dim W -1$ resp.\ $k+ \dim X + \dim W$  and constant cokernel of weight $k+ \dim X + \dim W $ resp.\ $k+ \dim X + \dim W+1$.
    In particular there are the following weight estimates for ${^{H,*}\!\!}\cM_L$ and ${^{H,!}\!\!}\cM_L$:
    \begin{align*}
        \gr^W_\ell (H^k( {^{H,*}\!\!}\cM_L)) = 0 &\quad \text{for} \quad \ell \neq k+ \dim W +\dim X -1 ,k+ \dim W +\dim X,\\
        \gr^W_\ell (H^k( {^{H,!}\!\!}\cM_L)) = 0 &\quad \text{for} \quad \ell \neq k+ \dim W +\dim X ,k+ \dim W +\dim X+1.
    \end{align*}
\end{enumerate}
\end{prop}

\begin{rmk}
 Notice that by a result of Chen--Dirks \cite[Theorem 1.4]{CD23} one can easily show point 1.\ of the above proposition, i.e., the fact that $\FL^V(\hat\iota_+ \O_{L^{\vee,*}}^\beta)$ and $\FL^V(\hat\iota_\dag \O_{L^{\vee,*}}^{-\beta})$
 underlie mixed Hodge modules. Namely, it is shown in loc.\ cit.\ that the monodromic Fourier-Laplace transform of a monodromic mixed Hodge module carries again the structure of a monodromic mixed Hodge module, and since the $\cD_V$-modules $\hat\iota_+ \O_{L^{\vee,*}}^\beta$ and $\hat\iota_\dag \O_{L^{\vee,*}}^\beta$ do underlie monodromic mixed Hodge modules on $V$ under the assumptions of \cref{prop:directImageOfOBetaMHM}, the result of Chen--Dirks can be applied. However, the proof of point 1.\ of \cref{prop:directImageOfOBetaMHM} below will yield a particular presentation of $\FL^V(\hat\iota_+ \O_{L^{\vee,*}}^\beta)$ and $\FL^V(\hat\iota_\dag \O_{L^{\vee,*}}^{-\beta})$ by standard functors (see Formula \eqref{eq:ProofOf38}), which
  is needed later  in Proposition \ref{prop:TautIsDirectImage}, which in turn is a cornerstone in the solution of the holonomic rank problem (see \cref{cor:holRankProb} in \cref{subsec:TautMHM} as well as \cref{theo:Main} in the introduction).
\end{rmk}

\begin{proof}[Proof of \cref{prop:directImageOfOBetaMHM}]
\begin{enumerate}
    \item
          We start by showing the statement for ${^{H,*}\!\!}\cM_L^\beta$. For this purpose, we combine \cref{prop:FLTransfer} and \cref{prop:FLofObeta} to get a purely functorial description of this complex of $\mathscr{D}_W$-modules not involving Fourier--Laplace transforms, namely
          \begin{equation}\label{eq:ProofOf38}
          \begin{array}{rcll}
            &&\FL^V(\hat\iota_+ \O_{L^{\vee,*}}^{\beta}) \\
            &\cong& \FL^V(a_{V,+}ev^\vee_+ j_{{L^{\vee}},+} \O_{L^{\vee,*}}^\beta) \\
            &\cong& a_{W,+} ev^+ \FL_X^{L^\vee}(j_{{L^\vee},+} \O_{L^{\vee,*}}^\beta) & {\textup{\small\cref{prop:FLTransfer}}} \\
            &\cong& a_{W,+} ev^+ j_{L,\dag} \O_{L^{*}}^{-\beta} & {\textup{\small\cref{prop:FLofObeta}}}\\
            &\cong& a_{W,+} ev^\dag j_{L,\dag} \O_{L^{*}}^{-\beta} & {ev \textup{ is smooth}}\\
            &= &a_{W,+} \, ev^\dag[\dim L - \dim (W\times X)] \, j_{L,\dag} \O_{L^{*}}^{-\beta}[\dim W-1],
          \end{array}
        \end{equation}
                where the last equality is due to the obvious dimension count $\dim(L)=\dim(X)+1$.

          Since $\O_{L^*}^{-\beta}$ underlies the complex pure Hodge module $^{H\!}\underline{\dC}^{-\beta}_{L^*}$ (see \cref{def:Obeta-new}), we obtain that
          \begin{equation}\label{eq:defMbeta}
              \Mh := a_{W,*} ev^*  j_{L,!}\, ^{H\!}\underline{\dC}^{-\beta}_{L^*}[\dim W-1]
              \in D^b\MHM(W,\dC).
          \end{equation}

        Define
        $$
        \MhD:=\left(\bD\, \Mh \right)(\dim(W \times X))
                $$
        where $\bD$ is the duality functor in $\MHM(W,\dC)$ as recalled in the introduction. Clearly, the complex of $\cD_W$-modules that underlies $\MhD$ is then $\bD \FL^V(\hat{\iota}_+\cO^\beta_{L^{\vee,*}})$, where this time $\bD$ is the holonomic duality functor on $\cD_W$-modules.

        On the spaces $V$ and $L^{\vee,*}$, we consider the isomorphisms $c_V$ and $c_{L^{\vee,*}}$ given by multiplication by $-1$ (in all variables for $c_V$ and fibrewise for $c_{L^{\vee,*}}$). Then since the Fourier transformation $\FL^V$ and the holonomic duality commute up to the action of $c_V$ (i.e., since
        $\bD\, \circ \FL^V\cong \FL^V\circ\bD\circ c_V^+$), we obtain the following isomorphisms in $D^b(\cD_W)$ for the complex of $\cD_W$-modules underlying $\MhD$:
        \begin{align*}
            \bD \FL^V(\hat{\iota}_+ \cO^\beta_{L^{\vee,*}}) &\simeq   \FL^V \bD \, c^+_V (\hat{\iota}_+ \cO^\beta_{L^{\vee,*}}) \\
            &\simeq   \FL^V \bD  (\hat{\iota}_+ c^+_{L^{\vee,*}} \cO^\beta_{L^{\vee,*}}) & {\textup{since } c_V\circ\hat{\iota}=\hat{\iota}\circ c_{L^{\vee,*}}\textup{ by definition of } \hat{\iota}}\\
            &\simeq   \FL^V \bD  (\hat{\iota}_+  \cO^\beta_{L^{\vee,*}})
            &{\exists\textup{ isomorphism } c^+_{L^{\vee,*}}\cO^\beta_{L^{\vee,*}}\cong \cO^\beta_{L^{\vee,*}}} \\
            & \simeq  \FL^V  (\hat{\iota}_\dag  \bD\, \cO^\beta_{L^{\vee,*}}) &
            {\bD\, \hat{\iota}_+ \cong \hat{\iota}_\dag\,\bD}\\
                    & \simeq  \FL^V  (\hat{\iota}_\dag   \cO^{-\beta}_{L^{\vee,*}})
            & {\bD\, \cO^\beta_{L^{\vee,*}}\cong \cO^{-\beta}_{L^{\vee,*}} \textup{ by \cref{lem:OLbetaDual}}}.
        \end{align*}
        This shows that the underlying complex of $\cD_W$-modules of         $\MhD$ is $\FL^V  (\hat{\iota}_\dag   \cO^{-\beta}_{L^{\vee,*}})$, as claimed.

        The second statement follows directly from the fact that $\cO_{L^*}^\beta\cong\cO_{L^*}^{\beta'}$ for $\beta-\beta'\in \Z$.
    \item
        For $\beta\in \dZ$, we have $\cO_{L^*}^{-\beta}=\cO_{L^*}$, which underlies an element in $\MHM(L^*)$, and by the above argument we get that
        ${^{H,*}\!\!}\cM_L, {^{H,!}\!\!}\cM_L \in D^b\MHM(W)$.
   \item
      Recall from \eqref{eq:ProofOf38} above that
      $$
        \FL^V(\hat{\iota}_+\cO_{L^{\vee,*}}^{-\beta}) \cong a_{W,+} ev^\dag j_{L,\dag} \cO_{L^*}^\beta.
      $$
      Applying the holonomic duality functor yields
      $$
        \bD \FL^V(\hat{\iota}_+\cO_{L^{\vee,*}}^{-\beta})
        \cong
        a_{W,+} ev^\dag j_{L,+} \bD \, \cO_{L^*}^\beta,
        \cong
        a_{W,+} ev^\dag j_{L,+} \cO_{L^*}^{-\beta},
      $$
      since $a_{W,\dag}\cong a_{W,+}$ ($a_W$ is proper) and since  $ev^+\cong ev^\dag$ ($ev$ is smooth). Now if $\beta\notin\dZ$, by using \cref{prop:CBetaHodgeExt}, we have $j_{L,+}\cO_{L^*}^{-\beta} \cong j_{L,\dag}\cO_{L^*}^{-\beta}$, and thus we obtain
      $$
        \bD \FL^V(\hat{\iota}_+\cO_{L^{\vee,*}}^{-\beta})
        \cong
        a_{W,+} ev^\dag j_{L,\dag} \cO_{L^*}^{-\beta}
        \cong
        \FL^V(\hat{\iota}_+\cO_{L^{\vee,*}}^{\beta}),
      $$
      from which we deduce an isomorphism
      $$
        {^{H,!}\!\!}\cM_L^{\beta} \cong \Mh
      $$
      in $D^b\MHM(W,\dC)$.

      Moreover, under the assumption that $\beta\notin \dZ$, we have seen in \cref{cor:FLofOBetaIsMHM} that $j_{L,!}\, ^{H\!}\underline{\dC}^{-\beta}_{L^*}$ is pure (of weight $\dim(X)+1$). Since the morphism $ev$ is smooth, and since $a_W$ is projective here, the second assertion thus follows from \cite[Th\'eor\`eme~1]{Saito1}.

    \item Recall that we denoted by  $j_L \colon L^* \ra L$ the inclusion of the complement of the zero section and denote by $i_L \colon X \ra L$ the inclusion of the zero section of $L$. There is the following adjunction triangle
    \[
        j_{L,!} \,j_L^{-1} \, ^{H\!}\underline{\dC}_{L}\lra \, ^{H\!}\underline{\dC}_{L} \lra i_{L,!} i^{-1}_L \, ^{H\!}\underline{\dC}_{L} \overset{+1}\lra .
    \]
    Since $i^{-1}_L \, ^{H\!}\underline{\dC}_{L} = ^{H\!}\underline{\dC}_{X}[1]$ we get the triangle
    $$
        i_{L,!} \, ^{H\!}\underline{\dC}_{X} \lra j_{L,!} \,j_L^{-1} \, ^{H\!}\underline{\dC}_{L}\lra \, ^{H\!}\underline{\dC}_{L}  \overset{+1}\lra .
    $$

     Since the map $j_L$ is affine, the functor $j_{L,!}$ from $\MHM(L^*)$ to $\MHM(L)$ is exact and $H^0(j_{L_!} {^{H\!}}\underline{\dC}_{L^*})$ is the only non-zero cohomology. Therefore we get the short exact sequence
     \begin{equation}
     \label{eq:adjtr}
     0 \lra i_{L,!} \, ^{H\!}\underline{\dC}_{X} \lra H^0(j_{L,!}  \, ^{H\!}\underline{\dC}_{L^*})\lra \, ^{H\!}\underline{\dC}_{L}  \lra 0.
     \end{equation}
     We have the following diagram with cartesian squares

     $$
        \begin{tikzcd}
        \mcy \arrow[rr, "i_{\mcy}", hook] \arrow[dd, "ev_{|\cY}", swap] && X \times W \arrow[dd, "ev"] && \mcu \arrow[ll, hook'] \arrow[dd, "ev_{|\cU}"] \\ \\
        X \arrow[rr, "i_L", hook] && L && L^* \arrow[ll, "j_L", hook', swap]
        \end{tikzcd}
     $$
     Applying the exact functor $ev^*[\dim W -1]$ to the short exact sequence \eqref{eq:adjtr} we get the short exact sequence
     \begin{equation}\label{eq:adjtri2}
     0 \lra i_{\mcy,!} \, ^{H\!}\underline{\dC}_{\mcy} \lra H^0(ev^* j_{L,!}  \, ^{H\!}\underline{\dC}_{L^*}[\dim W -1])\lra \, ^{H\!}\underline{\dC}_{X \times W}  \lra 0.
     \end{equation}
     Notice that $i_{\mcy,!} \, ^{H\!}\underline{\dC}_{\mcy}$ is pure of weight $\dim X + \dim W -1$ and that $^{H\!}\underline{\dC}_{X \times W}$ is pure of weight $\dim X + \dim W$. We apply the functor $a_{W,*}$ to \eqref{eq:adjtri2} and get
     \begin{equation}\label{eq:4termseq}
     H^{k-1}(a_{W,*}\, ^{H\!}\underline{\dC}_{X \times W} ) \ra H^k(a_{\mcy,*} \, ^{H\!}\underline{\dC}_{\mcy}) \ra H^{k} \left({^{H,*}\!\!}\cM_L\right) \ra H^k(a_{W,*}\, ^{H\!}\underline{\dC}_{X \times W} ).
     \end{equation}

     Since $H^k(a_{\mcy,*}\, ^{H\!}\underline{\dC}_{\mcy})$ is pure of weight $k+ \dim X + \dim W -1$ and the constant mixed Hodge module $H^k(a_{W,*} ^{H\!}\underline{\dC}_{X \times W} )$ is pure of weight $k +\dim X + \dim W$ we conclude that
     \[
     \gr^W_\ell (H^k( {^{H,*}\!\!}\cM_L)) = 0 \quad \text{for} \quad \ell \neq k+ \dim W +\dim X -1 ,k+ \dim W +\dim X.
     \]
     and there exists a map $H^k(a_{\mcy,*} \, ^{H\!}\underline{\dC}_{\mcy}) \ra H^{k} \left({^{H,*}\!\!}\cM_L\right)$ with constant kernel and cokernel.
     Applying  $\mbd$ to the sequence \eqref{eq:4termseq}  and doing a Tate-twist by $-(\dim X \times W)$ we get for $m = -k$
     \[
     H^m(a_{W,*}\, ^{H\!}\underline{\dC}_{X \times W}) \ra H^m\left({^{H,!}\!\!}\cM_L\right) \ra H^m(a_{\mcy,*} \, ^{H\!}\underline{\dC}_{\mcy})(-1) \ra H^{m+1}(a_{W,*}\, ^{H\!}\underline{\dC}_{X \times W} ).
     \]
     Since $H^k(a_{\mcy,*}\, ^{H\!}\underline{\dC}_{\mcy})$ is pure of weight $k+ \dim X + \dim W +1$ we conclude that
     \[
     \gr^W_\ell (H^m( {^{H,!}\!\!}\cM_L)) = 0 \quad \text{for} \quad \ell \neq m+ \dim W +\dim X ,m+ \dim W +\dim X+1.
     \]
     and there exists a map $ H^{k} \left({^{H,!}\!\!}\cM_L\right) \ra H^k(a_{\mcy,*} \, ^{H\!}\underline{\dC}_{\mcy})(-1)$ with constant kernel and cokernel.

                                   \end{enumerate}
\end{proof}

We will discuss next a natural geometric interpretation of the complex of mixed Hodge modules $\Mh$ resp.\ $\MhD$.

For this purpose, fix some value
$\lambda\in W$. Then, by definition, we have $\lambda\in \Gamma(X, \Ell)$, and interpreting this global section as a morphism $\lambda \colon X \hookrightarrow L$, we can consider the image $L_\lambda:=\im(\lambda)\subseteq L$. We identify the zero section of the projection $\pi\colon  L\twoheadrightarrow X$ inside $L$ with $X$ and recall that $L^{*}:=L\setminus X$ denotes the complement of the zero section. We denote by $H_\lambda:=
L_\lambda \cap X \subseteq X$ the zero locus of the section $\lambda$ (which was called $Z(\lambda)$ in \cref{theo:Main}) and by $U_\lambda := X\setminus H_\lambda$
its complement in $X$.

Notice that the full family of zero loci of sections of $L$ is given by $\cY:=ev^{-1}(0) \to  W$, $(s,\lambda)\mapsto \lambda$, i.e.  the fibre of this map over a point $\lambda\in W$ is exactly the hypersurface $H_\lambda$.
Similarly, we have
$\cU    =(W\times X)\setminus \cY = \bigcup_{\lambda\in W} U_\lambda$, the evaluation morphism $ev$ from Formula
\eqref{eq:evMorph} then restricts to a morphism
$$
ev_{|\cU} \colon  \cU \longrightarrow L^*.
$$
Recall that we defined the constant (complex) pure Hodge module $^{H\!}\underline{\dC}^\beta_{L^*}$  in  \cref{def:Obeta-new}. We then put
$$
^{H\!}\underline{\dC}^\beta_{\cU} := {ev_{|\cU}^*}
\,{^{H\!}}\underline{\dC}^\beta_{L^*}[\dim W -1]
\in\HM(\cU,\dC).
$$
Moreover, for $\lambda\in W$ as above, consider
the restriction $\lambda_{|U_\lambda}\colon U_\lambda
\hookrightarrow L^*$. We put, for $\beta\in \dQ$,
$$
{}^{H\!}\underline{\dC}^\beta_\lambda:=\lambda^*_{|U_\lambda}\, ^{H\!}\underline{\dC}^\beta_{L^*}[-1]\in \HM(U_\lambda,\dC).
$$

\begin{prop}\label{prop:TautIsDirectImage} We continue with the setup of \cref{prop:directImageOfOBetaMHM} and additionally assume that $X$ is projective. Let $\beta \in \C$ with $k\beta \in \Z$. Then the following statements hold true:
\begin{enumerate}
    \item
      Let $a_{\cU}\colon \cU\rightarrow W$ be the restriction of the projection $a_W\colon W\times X\rightarrow W$. Then we have an isomorphism
      $$
        a_{\cU,!} \, ^{H\!}\underline{\dC}^{-\beta}_{\cU}
        \cong         \Mh
        \quad\quad
        \textup{and}
        \quad\quad
        a_{\cU,*} \, ^{H\!}\underline{\dC}^{\beta}_{\cU}( 2 \dim W + 2\dim X)
        \cong \MhD
      $$
     in $ D^b\MHM(W,\dC)$.
    \item
      For any $m \in \dN$, and any $\lambda\in W$       we have isomorphisms of (complex) mixed Hodge structures
      $$
      \begin{array}{rcl}
          H^m(i_\lambda^* \, \Mh[-\dim W]) & \cong & H_c^{\dim(X) + m}(U_\lambda,\,^{H\!}\underline{\dC}^{-\beta}_\lambda), \\ \\
          H^m(i_\lambda^! \, \MhD[\dim W]) & \cong & H^{\dim(X)+m}(U_\lambda,\,^{H\!}\underline{\dC}^{\beta}_\lambda) (\dim W + 2 \dim X). \\ \\
      \end{array}
            $$
            \end{enumerate}
\end{prop}

\begin{proof}

  In the course of the proof, we will make repeatedly use of the base change property for algebraic mixed Hodge modules, as stated in \cite[Section 4.4.3]{SaitoMHM}.

  \begin{enumerate}
      \item
          This is almost immediate by considering the following cartesian diagram
          $$
          \begin{tikzcd}
            && \cU \ar{rr}{ev_{|\cU}} \ar[hook]{dd}{j_{\cU}} \ar[swap]{lldd}{a_\cU} && L^* \ar[hook]{dd}{j_{L}}\\ \\
            W && X\times W \ar{rr}{ev} \ar{ll}{a_W} && L
          \end{tikzcd}
          $$
          which yields (using \eqref{eq:defMbeta})
                                                                                                                              $$
            \begin{array}{rcl}
            \Mh
            &\stackrel{(*)}{\cong}&
            a_{W,!}\, ev^*\,{j_{L,!}}\, ^{H\!}\underline{\dC}^{-\beta}_{L^{*}}[\dim W-1]\\
            &\stackrel{(**)}{\cong}&
            a_{W,!}\, j_{\cU,!}\, {ev_{|\cU}^*} ^{H\!}\underline{\dC}^{-\beta}_{L^{*}}[\dim W -1] \\[0.5em]
            &=&
            a_{\cU,!}\, {ev_{|\cU}^*} ^{H\!}\underline{\dC}^{-\beta}_{L^{*}}[\dim W -1]\\
            &=&
            {a_{\cU,!}}\, ^{H\!}\underline{\dC}^{-\beta}_{\cU},
            \end{array}
            $$
          where the isomorphism $(*)$ holds because $a_W$ is proper (since $X$ is projective)
          and $ev$ is smooth, and where $(**)$ follows by base change. We then apply the duality functor $\bD$ on $D^b\MHM(W,\dC)$ on both sides of
          $a_{\cU,!} \, ^{H\!}\underline{\dC}^{-\beta}_{\cU}\cong \Mh$ to obtain that $a_{\cU,*} \, ^{H\!}\underline{\dC}^{\beta}_{\cU}(2\dim X + 2\dim W)
          \cong \MhD$.
      \item
          Write
          $$
          \Nh:=ev^* j_{L,!}\, ^{H\!}\underline{\dC}^{-\beta}_{L^{*}}[\dim W-1] \in \MHM(X\times W,\dC),
          $$
          then by the proof of the previous
          \cref{prop:directImageOfOBetaMHM} we have that
          $$
          \Mh \cong a_{W,!}\, \Nh.
          $$

          Now consider the cartesian diagram
          $$
          \begin{tikzcd}
            X\times \left\{\lambda\right\} \ar[hook]{rr}{i_\lambda^X} \ar[twoheadrightarrow]{dd}{a^X} && X\times W \ar[twoheadrightarrow]{dd}{a_W}\\ \\
            \left\{\lambda\right\} \ar[hook]{rr}{i_\lambda }&& W.
          \end{tikzcd}
          $$
          Then
          $$
          \begin{array}{rclcl}
            i_\lambda^*\,  \Mh            & \cong & i_\lambda^* {a_{W,!}}\, \Nh            \\ \\
            & \cong & a^X_! i^{X,*}_\lambda \, \Nh=
            a^X_!\,i^{X,*}_\lambda ev^* j_{L,!}\, ^{H\!}\underline{\dC}^{-\beta}_{L^{*}}[\dim W-1]
            && (\textup{base change}) \\ \\
            &\cong & a_!^X\, \lambda^* j_{L,!}\, ^{H\!}\underline{\dC}^{-\beta}_{L^{*}}[\dim W-1]  && (ev \circ i^X_\lambda =\lambda) \\ \\
            &\cong & a_*^X\, \lambda^* j_{L,!}\, ^{H\!}\underline{\dC}^{-\beta}_{L^{*}}[\dim W-1]   && (a^X \textup{ proper}).
          \end{array}
          $$
          Now we consider the diagram
          $$
          \begin{tikzcd}
            X \ar{rr}{\lambda} && L \\ \\
            U_\lambda \ar[hook]{uu}{j}
            \ar{rr}{\lambda_{|U_\lambda}}&& L^{*}
            \ar[hook]{uu}{j_{L}},
          \end{tikzcd}
          $$
          then base change yields $\lambda^*\, j_{L,!} \cong j_!\, \lambda_{|U_\lambda}^*$, so we get an isomorphism of objects in $D^b\MHM(\{\lambda\},\dC)$ (which we identify with the derived category of complex mixed Hodge structures).
          $$
          \begin{array}{rl}
          i_\lambda^*\,  \Mh                     &\cong
          a_*^X\, j_!\, \lambda_{|U_\lambda}^*\, ^{H\!}\underline{\dC}^{-\beta}_{L^{*}}[\dim W-1]
          \cong
          a_*^X\, j_!\, ^{H\!}\underline{\dC}^{-\beta}_\lambda[\dim W]\\
          &\cong
          a_!^X\, j_!\, ^{H\!}\underline{\dC}^{-\beta}_\lambda[\dim W]
          =a^{U_\lambda}_!\, ^{H\!}\underline{\dC}^{-\beta}_\lambda[\dim W],
          \end{array}
          $$
          where $a^{U_\lambda}\colon U_\lambda\twoheadrightarrow \{\lambda\}$ and where we have used $a^X_*=a^X_!$
          since $X$ is projective. We apply $H^m(-)$ to both sides to obtain
          an isomorphism of complex mixed Hodge structures
          $$
          H^m(i_\lambda^* \, \Mh          [- \dim W]) \cong
          H^m(a^{U_\lambda}_!\, ^{H\!}\underline{\dC}^{-\beta}_\lambda)
          =H^{m+ \dim X}_c(U_\lambda, {^{H\!}}\underline{\dC}^{-\beta}_\lambda).
        $$
                                recall that we use the convention $^{H\!}\underline{\dC}_{X} := a_X^* {^H\!}\dC_{pt}[\dim(X)]$.

For the second statement, we apply the duality functor $\bD$ in $D^b\MHM(W,\dC)$ to the isomorphism
$i_\lambda^*\,  \Mh  \cong a^{U_\lambda}_!\, ^{H\!}\underline{\dC}^{-\beta}_\lambda[\dim W]$ just proved, which gives
$$
i_\lambda^!\,  \MhD  \cong a^{U_\lambda}_*\, ^{H\!}\underline{\dC}^{\beta}_\lambda(\dim W + 2 \dim X)[-\dim W],
$$
and then by taking cohomology again we find that
$$
          H^m(i_\lambda^! \, \MhD[\dim W])  \cong H^{\dim(X)+m}(U_\lambda,\,^{H\!}\underline{\dC}^{\beta}_\lambda)(\dim W + 2 \dim X),
$$
as required. \qedhere
\end{enumerate}
\end{proof}

In the subsequent sections of this article, we will investigate to which extent tautological systems for homogeneous spaces $X$ are examples of the $\mathscr D$-modules underlying ${}^{H\!}\mathcal M_L^\beta$ for particular line bundles $L$ and values $\beta$.

\section{Non-vanishing criteria for tautological systems} \label{sec:nonVanishingCriteria}

The definition of a tautological system does not always describe a non-zero $\mathscr D$-module. In fact, for tautological systems arising from projective homogeneous spaces, this fails in a striking way, as we will see below in \cref{ssec:canSheafOnLineBdl}.
In that setup, tautological systems $\tau(\rho,\overline{Y},\beta)$ will only be non-zero for very particular representations $\rho$ and specific choices of $\beta$. In those cases however, tautological systems are particularly interesting.
The aim of this section is therefore to develop general criteria for vanishing resp.\ non-vanishing of tautological systems.

\subsection{\texorpdfstring{$\mathscr D$}{D}-modules from group actions} \label{ssec:VFfromGroupActions}

Here we consider the action of a linear algebraic group $G'$ on a variety $Y$.
The main case of interest (which will be discussed from \cref{ssec:canSheafOnLineBdl} on)
arises when we are given an action of a reductive linear algebraic group $G$ on a smooth variety $X$, and an equivariant line bundle $\mathscr{L}$ on $X$.  Denoting by $G'$ the group $G \times \C^*$, we let $G'$ act on  $Y$, which we take to be the total space $L$  (or the complement $L^*$ of the zero section) of $\mathscr{L}$. For the purpose of clarity, it is however useful to first treat the case of an arbitrary variety $Y$ admitting a $G'$-action, where $G'$ is any linear algebraic group. This is the point of view that we are going to adapt in Sections~\ref{ssec:VFfromGroupActions} to~\ref{sec:Twists}.

We begin by recalling some facts concerning group actions on smooth varieties. They are mainly included for the reader's convenience and in order to fix notations. The proofs are rather elementary and will therefore be omitted.

\begin{lem} \label{lem:defEquivVF}
  Let $G'$ be an algebraic group acting on a smooth variety $Y$. Then there is a unique Lie algebra homomorphism
  \[Z_Y \colon \mathfrak g' \to \Gamma(Y,\Theta_Y)\]
  associating to every element $\xi$ of the Lie algebra $\mathfrak g'$ of $G'$ a vector field $Z_Y(\xi)$ on $Y$ with the following point-wise description: At a point $y \in Y$, the tangent vector of the vector field $Z_Y(\xi)$ is given by $\differential \varphi^y(\xi)$, where $\varphi^y \colon G' \to Y$, $g \mapsto g^{-1} \cdot y$, and $\xi$ is understood as a tangent vector to $G'$ at the point $1 \in G'$.
\end{lem}

In the complex analytic category, the vector field $Z_Y(\xi)$ may be defined as the derivation
  \[Z_Y(\xi)(f) = \frac{d}{dt} \restr{f\left(\exp(t\xi)^{-1} \cdot {(-)}\right)}{t=0}.\]
If the $G'$-variety $Y$ considered is clear from the context, we will drop the index and just write $Z(\xi)$. In the literature, the vector field $Z(\xi)$ is sometimes denoted by $L_\xi$, see e.g.\ \cite[II.2]{Hot98}.

\begin{exa}
  Consider the action of $G'$ on itself by left-multiplication (i.e., $Y = G'$). Then $-Z_{G'}(\xi)$ is the right-invariant vector field associated to $\xi \in \mathfrak g'$. If, for example, $G' = (\C^*)^d$ and $\xi \in \C^d = \mathfrak g'$, then
  \[Z_{(\C^*)^d}(\xi) = -\sum_{i=1}^d \xi_i \, t_i \partial_{t_i},\]
  where $(t_1,\dots,t_d)$ are the standard coordinates on $(\C^*)^d$.
\end{exa}

For group actions on finite-dimensional vector spaces, we also have the following description:

\begin{lem} \label{lem:equivVFOnVectSp}
  Let $\rho \colon G' \to \GL(V)$ be a finite-dimensional rational representation of an algebraic group $G'$. The induced left action of $G'$ on $\C[V] = \bigoplus_{d\geq 0} \Sym^d V^\vee$ describes a morphism of algebraic groups $G' \to \GL_\C(\C[V])$ whose induced Lie algebra homomorphism $\mathfrak g' \to \End_\C(\C[V])$ makes the following diagram commute:
  \[\begin{tikzcd}
    \mathfrak g' \ar{rr} \ar{dr}[swap]{Z_V} & & \End_\C(\C[V]) \\
    & \Der(\C[V]) \ar[hookrightarrow]{ur} & .
  \end{tikzcd}\]
  Explicitly, if we fix coordinates $x_1, \ldots, x_n$ on $V$ and consider the associated Lie algebra representation $\differential\rho \colon \mathfrak g' \to \mathfrak{gl}(V) = \mathfrak{gl}(n,\C) = \C^{n\times n}$, then
  \[Z_V(\xi) = -\sum_{i,j=1}^n \differential\rho(\xi)_{ji} \, x_i \partial_{x_j}\]
  for all $\xi \in \Lie(G')$.
\end{lem}

\begin{exa}
  Let $G' = (\C^*)^d$ be a $d$-dimensional torus acting linearly on an $n$-di\-men\-sion\-al vector space $V$. We identify $V$ with $\C^n$ by picking a basis that diagonalizes the action, i.e., $t = (t_1,\dots,t_d) \in (\C^*)^d$ acts on $x = (x_1,\dots,x_n) \in \C^n$ by
  \[
  t \cdot x = (t^{\alpha_1} x_1, \dots, t^{\alpha_n} x_n) \qquad \text{with }\alpha_1, \ldots, \alpha_n \in \Z^d.
  \]
  If $\xi \in \Z^d = \Lie\left((\C^*)^d\right)$ is the $i$-th standard basis vector $e_i$, we get the vector field
  \[
  Z_V(e_i) = -\sum_{j=1}^n (\alpha_j)_i x_j \partial_{x_j}
  \]
  on $V$. These are the vector fields showing up in GKZ-systems associated to the given torus action.
\end{exa}

\begin{exa} \label{ex:P1VF}
  Let $X \subseteq \P^k$ be the rational normal curve of degree $k$, i.e., the image of
  \[\P^1 \xrightarrow{|\O(k)|} \P^k, \qquad [x_0:x_1] \mapsto \big[{\textstyle \binom{k}{i} \, x_0^{k-i} x_1^i \mid i = 0,\dots,k}\big],\]
  and let $Y := \hat X \setminus \{0\}$ be the punctured affine cone over $X$ in $V := \C^{k+1}$. The group $\SL(2)$ acts on $V = H^0(\P^1,\O(k))^\vee = \Sym^k (\C^2)$, the $k$-th symmetric power of the standard $\SL(2)$-representation, and we extend this to an action of $G' := \SL(2) \times \C^*$ by letting the $\C^*$-factor act by scaling on $V$.
    The Lie algebra $\mathfrak g'$ is generated by $E_{12}, E_{21}, E_{11}-E_{22} \in \fsl(2)$ and the generator $\mathbf e$ of $\Lie(\C^*) \cong \C$.
    The induced vector fields on $V$ are
  \[
  \begin{array}{c@{\hspace{3em}}c}
    \begin{aligned}
      Z_V(E_{12}) &= -\sum_{i=1}^k i \, z_i \partial_{z_{i-1}}, \\
      Z_V(E_{11}-E_{22}) &= -\sum_{i=0}^k  (k-2i) \, z_i \partial_{z_i},
    \end{aligned}
    &
    \begin{aligned}
      Z_V(E_{21}) &= -\sum_{i=1}^k (k-i+1) \, z_{i-1} \partial_{z_i}, \\
      Z_V(\mathbf e) &= -\sum_{i=0}^k z_i \partial_{z_i},
    \end{aligned}
  \end{array}
  \]
  where $z_0,\dots,z_k$ denote the coordinates on $V = \C^{k+1}$. Note that the minus signs appear because we differentiate the \emph{contragredient} action on the coordinate ring of $V$.

  On the $G'$-invariant subset $Y$, these vector fields restrict to the vector fields $Z_Y(\xi)$. In local charts, these can be expressed as follows: We may cover $Y$ by the two open subsets $U_0$ and $U_1$ given by the non-vanishing of $x_0^k \in V^\vee$ and $x_1^k \in V^\vee$, respectively. Identifying
  \begin{align*}
    U_0 &\cong \C^* \times \C, \qquad
    \lambda \cdot {\textstyle(1,ks, \binom{k}{2}s^2,\dots,ks^{k-1},s^k)} \quad \mapsfrom \ (\lambda,s), \\
    U_1 &\cong \C^* \times \C, \qquad
    \mu \cdot  {\textstyle(t^k,kt^{k-1},\binom{k}{2}t^{k-2},\dots,kt,1)} \ \, \mapsfrom \ (\mu,t),
  \end{align*}
  the vector fields induced from the $G'$-action on $Y$ are:
  \[
\begin{array}{c@{\hspace{3em}}c}
  \begin{aligned}
  \restr{Z_Y(E_{12})}{U_0} &= -ks \lambda \partial_\lambda + s^2 \partial_s, \\
  \restr{Z_Y(E_{21})}{U_0} &= -\partial_s, \\
  \restr{Z_Y(E_{11}-E_{22})}{U_0} &= -k \lambda \partial_\lambda + 2 s \partial_s, \\
  \restr{Z_Y(\mathbf e)}{U_0} &= -\lambda \partial_\lambda,
  \end{aligned} &
  \begin{aligned}
    \restr{Z_Y(E_{12})}{U_1} &= -\partial_t,\\
    \restr{Z_Y(E_{21})}{U_1} &= -kt \mu \partial_\mu + t^2 \partial_t,\\
    \restr{Z_Y(E_{11}-E_{22})}{U_1} &= k \mu \partial_\mu - 2 t \partial_t, \\
    \restr{Z_Y(\mathbf e)}{U_1} &= -\mu \partial_\mu.
  \end{aligned}
\end{array}\]
Note that these local expressions coincide on the intersection $U_0 \cap U_1$ under the gluing
$\C^* \times \C^* \xrightarrow{\cong} \C^* \times \C^*$, $(\lambda,s) \mapsto (\lambda s^k, s^{-1}) = (\mu,t)$.
\end{exa}

Using the vector fields defined in \cref{lem:defEquivVF}, we introduce the following $\mathscr{D}$-modules:

\begin{dfn} \label{def:NBeta}
  Let $G'$ be an algebraic group acting on a smooth variety $Y$ and let $\beta \colon \mathfrak g' \to \C$ be a Lie algebra homomorphism. Then we define the left $\mathscr{D}_Y$-module
  \[\mathscr N_Y^\beta := \omega_Y^\vee \otimes_{\O_Y} \mathscr{D}_Y/(Z_Y(\xi)-\beta(\xi) \mid \xi \in \mathfrak g')\mathscr{D}_Y. \qedhere\]
\end{dfn}

Recall that on a local coordinate system, the right-left transformation $\omega_Y^\vee \otimes (\cdot)$ is given by transposition of operators, see e.g.~\cite[§1.2]{Hotta}, but this  does not globalize in general.

\begin{rmk}\label{rem:LocSysRkOne}
In case that the action of $G'$ on $Y$ is transitive, it is easy to see that if $\mathscr N_Y^\beta\neq 0$, then it is a smooth $\cD_Y$-module of rank $1$ (namely, the vector fields $Z_Y(\xi)$, when $\xi$ runs through $\fg'$, generate the tangent bundle of $Y$). As already mentioned, our main case of interest is when $Y=L^*$ for some $\dC^*\times G$-equivariant line bundle $L$ on a homogeneous $G$-space $X$. Then $G':=\C^*\times G$ clearly acts transitively on $L^*$, and therefore $\mathscr N_Y^\beta$ corresponds to a rank $1$ local system on $L^*$. By the discussion in \cref{sec:OBeta}, we then know that $\mathscr N_Y^\beta$ must be one of the $\cD_{L^*}$-modules introduced in \cref{def:Obeta-new}. Under the hypothesis that $G$ is semisimple, one can also show  that a non-zero $\mathscr{N}_Y^\beta$ is isomorphic to
$\cO_{L^*}^{\beta(\mathbf e)}$, where $\mathbf e$ is the generator of the Lie algebra of $\C^*$, as we will discuss later in \cref{prop:NBetaEqualsOBeta}. However, one of the main points in this section is that very often, the module $\mathscr N_Y^\beta$ (and, if $Y$ is an orbit in a representation space $V$ of $G'$, the tautological system $\tau(\rho,\overline{Y},\beta)$) is zero, and then it is certainly not isomorphic to $\cO_{L^*}^{\beta}$. We will develop below criteria that guarantee the non-vanishing of the modules $\mathscr N_Y^\beta$ resp.\ of tautological systems (see \cref{prop:NBetaEqualsOBeta} and \cref{thm:restrictedTauHatDescription} below).
\end{rmk}

\begin{exa}
  Let $G' = T = (\C^*)^d$ be a $d$-dimensional torus acting on itself. We identify Lie algebra homomorphisms $\beta \colon \C^d = \mathfrak g' \to \C$ with vectors $\beta \in \C^d$. Then
  \[\mathscr N_T^\beta = \omega_T^\vee \otimes_{\O_T} \mathscr{D}_T/(-t_i \partial_{t_i}-\beta_i \mid i = 1,\dots,d)\mathscr{D}_T \cong \mathscr{D}_T/\mathscr{D}_T(\partial_{t_i} t_i-\beta_i \mid i = 1,\dots,d).\]
  This $\mathscr D_T$-module was called $\O_T^{-\beta}$ in \cite{RS20}. \end{exa}

\begin{exa} \label{ex:P1NBeta}
  We reconsider the action of $G' = \SL(2) \times \C^*$ on the punctured affine cone $Y$ over the rational normal curve of degree $k$ from \cref{ex:P1VF} and use the notations from before. Every Lie algebra homomorphism $\beta \colon \mathfrak g' \to \C$ is given by $\restr{\beta}{\fsl(2)} \equiv 0$ and $\beta(\mathbf{e}) = \beta_0 \in \C$. By the computations in \cref{ex:P1VF}, in the local chart $U_0 \cong \C^* \times \C \subseteq Y$, the $\mathscr D_Y$-module $\mathscr{N}_Y^\beta$ can be expressed as
  \begin{align*}
    \restrK{\mathscr N_Y^\beta}{U_0} &\cong
    \omega_{U_0}^\vee \otimes_{\O_{U_0}} \mathscr{D}_{U_0}/(
    -ks \lambda \partial_\lambda + s^2 \partial_s, \;
    -\partial_s, \;
    -k \lambda \partial_\lambda + 2 s \partial_s, \;
    -\lambda \partial_\lambda - \beta_0)\mathscr{D}_{U_0} \\
    &\cong \mathscr{D}_{U_0}/\mathscr{D}_{U_0}(
    ks \partial_\lambda \lambda - \partial_s s^2, \;
    \partial_s, \;
    k \partial_\lambda \lambda - 2 \partial_s s, \;
    \partial_\lambda \lambda - \beta_0) \\
    &\cong \mathscr{D}_{U_0}/\mathscr{D}_{U_0}(
    ks \lambda \partial_\lambda - s^2 \partial_s + (k-2)s, \;
    \partial_s, \;
    k \lambda \partial_\lambda - 2 s \partial_s + (k-2), \;
    \lambda \partial_\lambda + 1 - \beta_0) \\
    &\cong \mathscr{D}_{U_0}/\mathscr{D}_{U_0}(
    \partial_s, \;
    \lambda \partial_\lambda + 1 - \beta_0, \; k(-1+\beta_0)+(k-2)) \\
    &\cong
    \begin{cases}
      \mathscr{D}_{\C^*}/\mathscr{D}_{\C^*}(\partial_\lambda \lambda - \beta_0) \boxtimes \mathscr{D}_{\C}/\mathscr{D}_{\C} \cdot \partial_s
            \quad &\text{if } \beta_0 = 2/k, \\
      0 &\text{otherwise}
    \end{cases}
  \end{align*}
  and similarly for the other local chart $U_1$ of $Y$. In particular, for one specific value for $\beta(\mathbf e)$, we obtain a non-zero $\mathscr D_Y$-module that will be of interest to us.

  Note that in contrast, if we define the cyclic \emph{left} module
  \[\tilde{\mathscr N}_Y^\beta := \mathscr{D}_Y/\mathscr{D}_Y(Z_Y(\xi)-\beta(\xi) \mid \xi \in \mathfrak g'),\]
  then, in this example, we get
  \begin{align*}
    \restrK{\tilde{\mathscr N}_Y^\beta}{U_0} &\cong
    \mathscr{D}_{U_0}/\mathscr{D}_{U_0}(
    -ks \lambda \partial_\lambda + s^2 \partial_s, \;
    -\partial_s, \;
    -k \lambda \partial_\lambda + 2 s \partial_s, \;
    -\lambda \partial_\lambda - \beta_0) \\
    &\cong \mathscr{D}_{U_0}/\mathscr{D}_{U_0}(
    \partial_s, \;
    \lambda \partial_\lambda + \beta_0, \;
    k \beta_0)
    \cong
    \begin{cases}
      \O_{U_0} &\text{if } \beta_0 = 0, \\
      0 &\text{otherwise}.
    \end{cases}
  \end{align*}
  For $k = 2$, we have $\mathscr N_Y^{(\beta_0=1)} = \tilde{\mathscr N}_Y^{(\beta_0=0)}$, but in general they do not agree with each other. In fact, one can show that although $\mathscr{N}_Y^\beta$ is \emph{locally} a cyclic left $\mathscr D_Y$-module, it does not admit a \emph{global} description as a cyclic left $\mathscr D_Y$-module for $k \geq 3$.
\end{exa}

The main reason we wish to consider the $\mathscr D_Y$-module $\mathscr N_Y^\beta$ defined via the right-left-transformation of a cyclic right-module is the following behavior under equivariant closed embeddings:

\begin{prop} \label{prop:directImageOfNbeta}
  Let $G'$ be an algebraic group and let $i \colon Y_1 \hookrightarrow Y_2$ be a $G'$-equivariant closed embedding between smooth $G'$-varieties $Y_1,Y_2$.
    Then, for all Lie algebra homomorphisms $\beta \colon \mathfrak g' \to \C$, we have
  \[i_+ \mathscr N_{Y_1}^\beta \cong   \omega_{Y_2}^\vee \otimes_{\O_{Y_2}} \mathscr{D}_{Y_2}/\big(\mathcal{I} + (Z_{Y_2}(\xi)-\beta(\xi) \mid \xi \in \mathfrak g')\big)\mathscr{D}_{Y_2}\]   where $\mathcal {I} \subseteq \O_{Y_2}$ is the ideal sheaf of $Y_1$ in $Y_2$.
\end{prop}

\begin{proof}
  Since $i \colon Y_1 \hookrightarrow Y_2$ is a closed embedding, the functor $i_*$ is exact and the transfer module $\mathscr{D}_{Y_1 \to Y_2}$ is a flat $\mathscr{D}_{Y_1}$-module. Therefore, the direct image of $\mathscr N_{Y_1}^\beta$ under $i$ is given by
  \[i_+  \mathscr N_{Y_1}^\beta \cong \omega_{Y_2}^\vee \otimes_{\O_{Y_2}} i_*\big(\mathscr{D}_{Y_1}/(Z_{Y_1}(\xi)-\beta(\xi) \mid \xi \in \mathfrak g') \mathscr D_{Y_1} \otimes_{\mathscr{D}_{Y_1}} \mathscr{D}_{Y_1\to Y_2}\big).\]
  Hence, the claim is that
  \begin{align*}
    &\mathscr{D}_{Y_1}/(Z_{Y_1}(\xi)-\beta(\xi) \mid \xi \in \mathfrak g') \mathscr D_{Y_1} \otimes_{\mathscr{D}_{Y_1}} \mathscr{D}_{Y_1\to Y_2} \\
    &\qquad \cong i^{-1} \big(\mathscr{D}_{Y_2}/((Z_{Y_2}(\xi)-\beta(\xi)\mid \xi \in \mathfrak g') \mathscr{D}_{Y_2} + \mathcal I \mathscr{D}_{Y_2}) \big).
  \end{align*}
  as right $i^{-1}\mathscr{D}_{Y_2}$-modules.
  Note that $\mathscr{D}_{Y_1\to Y_2} \cong i^{-1}(\mathscr{D}_{Y_2}/\mathcal I \mathscr{D}_{Y_2})$ as right $i^{-1} \mathscr{D}_{Y_2}$-modules, since $i$ is a closed embedding. Under the left $\mathscr{D}_{Y_1}$-module structure on $\mathscr{D}_{Y_1\to Y_2}$, vector fields on $Y_1$ act via the push-forward homomorphism
  \[\differential i \colon \Theta_{Y_1} \to i^* \Theta_{Y_2} = \O_{Y_1} \otimes_{i^{-1} \O_{Y_2}} i^{-1} \Theta_{Y_2} \cong i^{-1}(\O_{Y_2}/\mathcal{I} \otimes_{\O_{Y_2}} \Theta_{Y_2}).\]

  We note that the push-forward of the vector field $Z_{Y_1}(\xi)$ on $Y_1$ agrees with the restriction of the vector field $Z_{Y_2}(\xi)$ on $Y_2$ to $Y_1$, i.e., $\differential i(Z_{Y_1}(\xi)) = 1 \otimes Z_{Y_2}(\xi)$. Indeed, this follows from the construction of $Z_{Y_1}(\xi)$ and $Z_{Y_2}(\xi)$, using the commutativity of
  \[\begin{tikzcd}
    G' \times Y_1 \ar{r}{\varphi_1} \ar{d}{\id_{G'} \times i} & Y_1 \ar{d}{i} \\
    G' \times Y_2 \ar{r}{\varphi_2} & Y_2,
  \end{tikzcd}\]
  where $\varphi_1, \varphi_2$ are the morphisms given by the $G'$-actions.

  This shows that $Z_{Y_1}(\xi) \in \Der(\O_{Y_1})$ acts on the right $i^{-1} \mathscr D_{Y_2}$-module $\mathscr D_{Y_1 \to Y_2} \cong i^{-1}(\mathscr D_{Y_2}/\mathcal I \mathscr{D}_{Y_2})$ by left-multiplication with $Z_{Y_2}(\xi)$. This implies the claimed description as a cyclic right $i^{-1} \mathscr{D}_{Y_2}$-module of $\mathscr{D}_{Y_1}/(Z_{Y_1}(\xi)-\beta(\xi) \mid \xi) \mathscr D_{Y_1} \otimes_{\mathscr{D}_{Y_1}} \mathscr{D}_{Y_1\to Y_2}$, concluding the proof.
\end{proof}

The $\mathscr{D}$-modules in \cref{prop:directImageOfNbeta} look similar to the $\beta$-twistedly equivariant $\mathscr{D}$-modules considered in \cite[II.2]{Hot98}, yet they are different: Instead of considering a cyclic \emph{left} module obtained by quotienting out a $G'$-stable ideal and the vector fields induced by the group action (twisted with $\beta$), we instead consider the \emph{right} module constructed in the same way and apply a right-left transformation to obtain a left $\mathscr{D}$-module. The behavior under direct images of closed embeddings in \cref{prop:directImageOfNbeta} is the reason why for our purposes we work with the definition via right modules in \cref{def:NBeta}.

We next consider the situation where $Y$ is an orbit
of a rational representation $\rho$ of our group $G'$ in a given vector space $V$. Recall from our basic  \cref{def:tautSys} that under this hypothesis, we can define, for any Lie algebra homomorphism $\beta\colon \fg'\rightarrow \dC$, the $\mathscr{D}_V$-module $\hat{\tau}(\rho, \overline{Y},\beta)$ (as well as its Fourier-Laplace transform $\tau(\rho, \overline{Y},\beta)$ which was called tautological system in \cref{def:tautSys}). The next result tells us about a technically easy but imporant relation of this $\hat{\tau}(\rho, \overline{Y},\beta)$ to the $\mathscr{D}_Y$-module $\mathscr{N}_Y^\beta$ considered above.
\begin{cor} \label{cor:FLTautSysAsDirectImage}
  Let $\rho \colon G' \to \GL(V)$ be a finite-dimensional rational representation of an algebraic group and let $\beta \colon \mathfrak g' \to \C$ be a Lie algebra homomorphism. Let $Y \subseteq V$ be a $G'$-orbit, let $\overline{Y}$ be its closure and let $\partial Y := \overline Y \setminus Y$.
  Then
  \[j^+ \hat{\tau}(\rho, \overline{Y}, \beta) \cong i_+ \mathscr{N}_{Y}^{\beta},\]
  where $Y = \overline Y \setminus \partial Y \xhookrightarrow{i} U:=V \setminus \partial Y \xhookrightarrow{j} V$.

  In particular, if  $\hat{\tau}(\rho, \overline{Y}, \beta)$ is localized at $\partial Y$ (meaning $j_+ j^+ \hat{\tau}(\rho,\overline{Y},\beta) \cong \hat{\tau}(\rho,\overline{Y},\beta)$), then it is the direct image of $\mathscr{N}_{Y}^{\beta}$ under the locally closed embedding $Y \hookrightarrow V$.
\end{cor}

\begin{proof}
  We apply \cref{prop:directImageOfNbeta} to the $G'$-spaces $Y_1 := Y$, $Y_2 := V \setminus \partial Y$ and the closed embedding $i \colon Y_1 \hookrightarrow Y_2$ to see that
  \[
  i_+ \mathscr{N}_{Y}^{\beta} \cong \omega_{U}^\vee \otimes_{\O_U} \mathscr{D}_U/\big(\cI + \{Z_U(\xi)-\beta(\xi)\}\big)\mathscr{D}_U.
  \]
  Choosing coordinates $x_1,\dots,x_n$ on $V$, we may by \cref{lem:equivVFOnVectSp} express the vector field $Z_U(\xi)$ as the derivation $-\sum_{i,j=1}^n \differential\rho(\xi)_{ji} \, x_i \partial_{x_j}$. The right-left transformation $\omega_U^\vee \otimes_{\O_U} (\cdot)$ is then explicitly given by transposing operators:
  \[
  i_+ \mathscr{N}_{Y}^\beta \cong \mathscr{D}_U/\mathscr{D}_U \big(\cI + \{Z_U(\xi)^T-\beta(\xi)\}\big).
  \]
  An explicit computation of the transposed vector fields yields:
  \[
  Z_V(\xi)^T = \sum_{i,j=1}^n \differential\rho(\xi)_{ji} \, \partial_{x_j} x_i
  = \sum_{i,j=1}^n \differential \rho(\xi)_{ji} \, x_i \partial_{x_j} + \sum_{i=1}^n \differential\rho(\xi)_{ii} = -Z_V(\xi)+\trace(\differential \rho(\xi)),
  \]
  hence (using $Z_V(\xi)_{|U}=Z_U(\xi)$) we have that $i_+ \mathscr{N}_{Y}^\beta \cong j^+ \hat{\tau}(\rho, \overline{Y}, \beta)$.
\end{proof}

\begin{exa}[GKZ-systems]

  Consider a torus representation $\rho \colon (\C^*)^n \to \GL(n,\C)$ that is given by $\rho(t_1,\dots,t_d) = \diag(t^{\alpha_1}, \dots, t^{\alpha_n})$ with $\alpha_i \in \Z^d$. Let $\overline{Y} \subseteq \C^n$ be the orbit closure of the point $(1,\dots,1) \in \C^n$; this is a (not necessarily normal) affine toric variety. The $\mathscr{D}_{\C^n}$-module $\hat{\tau}(\rho, \overline{Y}, \beta)$ is the Fourier--Laplace transform $\widehat{\mathscr{M}}_{A}(-\beta)$ of the GKZ-system $\mathscr{M}_{A}(-\beta)$ (see, e.g.  \cite{GKZReview} for an overview and for the notation used here), where $A$ is the $d \times n$-matrix whose $i$-th column is $\alpha_i$ and $\beta \colon \Lie\left((\C^*)^d\right) = \Z^d \to \C$ is identified with the vector $\left(\beta(e_i)\right)_{i=1,\dots,d} \in \C^d$.

  In this case, \cref{cor:FLTautSysAsDirectImage} applied to $Y = \overline{Y} \cap (\C^*)^n$ says that $\widehat{\mathscr{M}}_{A}(-\beta)$ is the direct image of $\mathscr{N}_{(\C^*)^d}^\beta = \O_{(\C^*)^d}^{-\beta}$ under the locally closed embedding $(\C^*)^d \cong Y \hookrightarrow \C^n$, whenever $\mathscr{M}_{A}(-\beta)$ is localized at the intersection of $Y$ with the union of coordinate hyperplanes of $\C^n$. This was observed in \cite{SW09}, where an explicit combinatorial characterization of the localization property in terms of $A$ and $\beta$ was proved using Euler--Koszul complexes.
\end{exa}

\begin{exa}
  Reconsider from \cref{ex:P1VF} the punctured affine cone $Y$ over the rational normal curve of degree $k$. This may be identified with the complement of the zero section in the line bundle $L = \Tot(\O_{\P^1}(k)) \to \P^1$. The calculation in \cref{ex:P1NBeta} shows that $\mathscr N_Y^\beta = \O_{L^*}^{-\beta(\mathbf e)}$ if $\beta(\mathbf e) = 2/k$ and $\mathscr N_Y^\beta = 0$ otherwise. \cref{cor:FLTautSysAsDirectImage} shows that the restriction of the FL-transformed tautological system
  \begin{align*}
    \hat \tau(\rho,\overline Y, \beta) = \mathscr D_V / \mathscr D_V \cdot &\left\{\binom{k}{i_2} \binom{k}{j_2} z_{i_1} z_{j_1}- \binom{k}{i_1} \binom{k}{j_1} z_{i_2} z_{j_2} \mid i_1+j_1=i_2+j_2\right\} \\
    &\cup \left\{-\sum_{i=1}^k i \, z_i \partial_{z_{i-1}}, \; -\sum_{i=1}^k (k-i+1) \, z_{i-1} \partial_{z_i}, \; \right. \\
    &\phantom{{}\cup{} \bigg\{}\left. -\sum_{i=0}^k  (k-2i) \, z_i \partial_{z_i}, \; -\sum_{i=0}^k z_i \partial_{z_i}-(k+1)+\beta(\mathbf e)\right\}
  \end{align*}
  to the complement of the origin in $V$ is
  \[
  \restr{\hat \tau(\rho,\overline Y,\beta)}{V\setminus \{0\}} =
  \begin{cases}
   i_+ \O_{L^*}^{-\beta(\mathbf{e})} &\text{if } \beta(\mathbf e) = 2/k, \\
   0 &\text{otherwise}.
  \end{cases}
  \]
\end{exa}

\subsection{\texorpdfstring{$\cD$}{D}-modules from equivariant line bundles} \label{ssec:DFromAMod}

The non-vanishing of tautological systems is by \cref{cor:FLTautSysAsDirectImage} closely tied to the non-vanishing of the $\mathscr{D}$-modules $\mathscr{N}_Y^\beta$. The aim of this section is to study criteria for $\mathscr N_Y^\beta$ to be (non-)zero. In order to do so, we introduce another construction of $\cD$-modules on a $G'$-variety $Y$, which also depends on the choice of an equivariant line bundle on $Y$. It will turn out (this is the main result of \cref{sec:Twists} below) that the modules $\mathscr N_Y^\beta$ can be expressed in exactly this way.
    For such $\cD$-modules defined by equivariant line bundles, it is possible to develop non-vanishing criteria using an interpretation via modules over rings of twisted differential operators (called $\cA$-modules below).

We continue with the setup of the previous section. Let $G'$ be a connected linear algebraic group acting on a smooth connected algebraic variety $Y$. Denote by $\mathfrak g'$ the Lie algebra of $G'$ and by $\mathcal U(\mathfrak g')$ its universal enveloping algebra. Every element $\xi$ of $\mathfrak g'$ induces a vector field $Z_Y(\xi) \in \Gamma(Y,\Theta_Y)$ by \cref{lem:defEquivVF}, and this map extends to a homomorphism of $\O_Y$-modules
\[Z_Y \colon \O_Y \otimes_\C \mathfrak g' \to \Theta_Y\]
via $Z_Y(f \otimes \xi) = f Z_Y(\xi)$ for $f \in \O_Y$, $\xi \in \mathfrak g'$.

\begin{dfn} \label{def:A}
  Given the $G'$-variety $Y$, we define
    \[\cA_Y := \O_Y \otimes_\C \mathcal U(\mathfrak g'),\]
  which has the structure of an associative $\C$-algebra with multiplication given by
  \[(f_1 \otimes \xi_1) \cdot (f_2 \otimes \xi_2) = f_1 f_2 \otimes \xi_1 \xi_2 + f_1 Z_Y(\xi_1)(f_2) \otimes \xi_2.\]
\end{dfn}

The $\O_Y$-module homomorphism $Z_Y$ extends to a homomorphism of associative $\C$-algebras
\[\tilde{Z}_Y \colon \cA_Y \to \mathscr D_Y.\]
For any left $\cA_Y$-module $\mathcal M$, we may consider the left $\mathscr{D}_Y$-module obtained by scalar extension
  \[\mathscr{D}_Y \otimes_{\cA_Y} \mathcal{M}.\]
On the other hand, note that the homomorphism $\tilde{Z}_Y$ induces a forgetful functor from the category of left $\mathscr{D}_Y$-modules to the category of left $\cA_Y$-modules.

The associative algebra $\cA_Y$ is the universal enveloping algebra of the Lie algebroid $(\O_Y \otimes_\C \mathfrak g',\, Z_Y)$ on $Y$, see \cite[1.8.4.Example]{BeilinsonBernstein_Jantzen}. This is the reason why, in many ways, modules over $\cA_Y$ behave similarly to modules over the algebra $\mathscr{D}_Y$ (which can be viewed as the universal enveloping algebra of the Lie algebroid $\Theta_Y$). For example, the tensor product of two left $\cA_Y$-modules over $\O_Y$ is again naturally a left $\cA_Y$-module, while the tensor product of a left and a right $\cA_Y$-module over $\O_Y$ naturally becomes a right $\cA_Y$-module. Applying basic results on modules over universal enveloping algebras of Lie algebroids \cite[Appendice]{CN05} to $\tilde{Z}_Y \colon \cA_Y \to \mathscr{D}_Y$, we obtain the following elementary properties:

\begin{lem}[{\cite[Théorème A.6 and Corollaire A.2]{CN05}}] \label{lem:tensorProductsOnAandD}
  Let $\mathcal M$ be a left $\cA_Y$-module. Let $\mathcal N$ (resp.\ $\mathcal N'$) be a left (resp.\ right) $\mathscr{D}_Y$-module. Then there are natural isomorphisms
  \begin{enumerate}
    \item $\mathscr D_Y \otimes_{\cA_Y} (\mathcal M \otimes_{\O_Y} \mathcal N) \cong (\mathscr D_Y \otimes_{\cA_Y} \mathcal M) \otimes_{\O_Y} \mathcal N$ as left $\mathscr{D}_Y$-modules, \label{item:leftCNMN}
    \item $(\mathcal M \otimes_{\O_Y} \mathcal N') \otimes_{\cA_Y} \mathscr{D}_Y \cong (\mathscr{D}_Y \otimes_{\cA_Y} \mathcal M) \otimes_{\O_Y} \mathcal N'$ as right $\mathscr D_Y$-modules. \label{item:rightCNMN}
  \end{enumerate}
  Here, on the left hand sides, $\mathcal N$ and $\mathcal N'$ are considered as $\cA_Y$-modules via $\tilde Z_Y\colon \cA_Y \to \mathscr D_Y$.
\end{lem}

From now on, we will only be interested in the case that $G'$ acts transitively on $Y$. In this case, the $\O_Y$-module homomorphism $Z_Y \colon \O_Y \otimes_\C \mathfrak g' \to \Theta_Y$ is surjective, hence the same is true for $\tilde{Z}_Y \colon \cA_Y \to \mathscr D_Y$, so
\[\mathscr{D}_Y \cong \cA_Y/\ker \tilde Z_Y.\]

We observe that the kernel of $\tilde Z_Y$ (which is a two-sided ideal in $\cA_Y$) is generated as a left ideal in $\cA_Y$ by the kernel of $Z_Y$:

\begin{lem}
\label{lem:kernelGeneratedByVFs}
  If $G'$ acts transitively on $Y$, then
  \[\ker\big(\tilde{Z}_Y \colon \cA_Y \to \mathscr{D}_Y\big) = \cA_Y \cdot \ker\big(Z_Y \colon \O_Y \otimes \mathfrak g' \to \Theta_Y\big).\]
\end{lem}

\begin{proof}
  We check the claim locally. For this, let $p \in Y$ be an arbitrary point and let $U \subseteq Y$ be an open neighborhood of $p$ admitting a local coordinate system $(x_1,\dots,x_n)$, so that $\Theta_U = \bigoplus_{i=1}^n \O_U \partial_{x_i}$. We claim that by further shrinking the open set $U$, we may choose
  an appropriate $\O_U$-basis $\theta_1,\dots,\theta_m$ of the free $\O_U$-module $\O_U \otimes \mathfrak g'$ such that the surjective homomorphism of $\O_U$-modules
  \[\restrK{Z_Y}{U} \colon \O_U \otimes \mathfrak g' \to \Theta_U\]
  is given by
  \[\theta_i \mapsto \begin{cases} \partial_{x_i} &\text{if } i \leq n \\ 0 &\text{if } i > n. \end{cases}\]

  Indeed, $\restrK{Z_Y}{U}$ is a surjective homomorphism of free $\O_U$-modules of finite rank and we may represent it by an $n \times m$-matrix $A$ (with $m \geq n$) by choosing \emph{any} $\O_U$-basis of $\O_U \otimes \mathfrak g'$. By surjectivity of $Z_Y$, some $n \times n$-minor of $A$ does not vanish at the point $p$. After permuting the chosen $\O_U$-basis of $\O_U \otimes \mathfrak g'$, we may assume that the non-vanishing set $V \subseteq U$ of the minor given by the first $n$ columns is an open neighborhood of $p$. Writing
  \[A = (A_1 \, | \, A_2) \qquad \text{with } A_1 \in \operatorname{Mat}(n\times n,\O_U), A_2 \in \operatorname{Mat}(n\times (m-n),\O_U),\]
  we have $A_1 \in \GL(n,\O_V)$. Changing the $\O_U$-basis on $\restr{(\O_U \otimes \mathfrak g')}{V} = \O_V \otimes \mathfrak g'$ corresponds to right-multiplying $A$ with an element of $\GL(m,\O_V)$. Then
  \[\begin{pmatrix}
    A_1 & A_2
  \end{pmatrix} \cdot
  \begin{pmatrix}
    A_1^{-1} & -A_1^{-1}A_2 \\
    0 & \Id_{m-n}
  \end{pmatrix} =
  \begin{pmatrix}
    \Id_n & 0
  \end{pmatrix}\]
  shows that a choice of $\theta_1,\dots, \theta_m$ as desired exists.

  Now, every section of $\cA_U$ can be expressed as a sum of elements of the form $f \theta_1^{a_1} \theta_2^{a_2} \dots \theta_m^{a_m}$ with $f \in \O_U$, $a_1,\dots,a_m \in \N$, each of which gets mapped under $\restrK{\tilde{Z}_Y}{U}$ to
  \[f \theta_1^{a_1} \theta_2^{a_2} \dots \theta_m^{a_m} \mapsto \begin{cases}
    f \partial_{x_1}^{a_1} \partial_{x_2}^{a_2} \dots \partial_{x_n}^{a_n} &\text{if } a_{n+1} = \dots = a_m = 0, \\
    0 &\text{otherwise.}
  \end{cases}\]
  From this, we can see that every section of $\cA_U$ getting mapped to zero under $\restrK{\tilde{Z}_Y}{U}$ is an element of \[\cA_U \cdot \{\theta_{n+1},\dots,\theta_m\} = \cA_U \cdot \ker(\restrK{Z_Y}{U}). \qedhere\]
\end{proof}

\cref{lem:kernelGeneratedByVFs} is in fact a special case of a more general fact about Lie algebroids: If $\varphi \colon \mathscr F_1 \twoheadrightarrow \mathscr F_2$ is a surjective homomorphism of two locally free Lie algebroids of finite rank on the same variety $Y$, then the kernel of the induced homomorphism of universal enveloping algebras $\tilde{\varphi} \colon \mathcal U(\mathscr F_1) \twoheadrightarrow \mathcal U(\mathscr F_2)$ is generated by $\ker \varphi$ as a left $\mathcal U(\mathscr F_1)$-ideal. A similar proof to the above carries over.

\paragraph{Equivariant line bundles as \texorpdfstring{$\cA_Y$}{A\_Y}-modules:} If $E \to Y$ is a $G'$-equivariant line bundle and we denote by $\mathscr{E}$ its sheaf of sections, then for every open subset $U \subseteq Y$, the Lie algebra $\mathfrak g'$ acts on $\Gamma(U,\mathscr{E})$. This makes $\mathscr{E}$ a left $\cA_Y$-module. We will be particularly interested in the left $\mathscr D_Y$-module
$\mathscr{D}_Y \otimes_{\cA_Y} \mathscr{E}$ arising from this.

\begin{rmk}
  If $U \subseteq Y$ is an open subset not invariant under $G'$, then $G'$ does not act on $U$. Yet, we still get $Z_U \colon \O_U \otimes_\C \mathfrak g' \to \Theta_U$, allowing us to define $\cA_U$.
    While $\restr{\mathscr{E}}{U}$ is not $G'$-equivariant, it still is a left $\cA_U$-module, and we may consider $\mathscr{D}_U \otimes_{\cA_U} \restr{\mathscr{E}}{U}$. This suggests a generalized viewpoint, where we replace the $G'$-action on $Y$ by a $\mathfrak g'$-action on $\O_Y$, and replace $G'$-equivariant line bundles with line bundles carrying a left $\cA_Y$-module structure.
\end{rmk}

Next, we examine when equivariant line bundles give rise to non-zero $\mathscr{D}$-modules.

\begin{prop} \label{prop:nonZeroIndependentOfPower}
  Assume $G'$ acts transitively on $Y$. Let $\mathscr{E}$ be a $G'$-equivariant line bundle on $Y$. Then the following are equivalent:
    \begin{enumerate}
      \item $\mathscr D_Y \otimes_{\cA_Y} \mathscr E \neq 0$, \label{item:nonZero}
      \item $\mathscr E \to \mathscr D_Y \otimes_{\cA_Y} \mathscr E$ is an isomorphism of left $\cA_Y$-modules, \label{item:natIsom}
      \item $\mathscr E^{\otimes k} \to \mathscr D_Y \otimes_{\cA_Y} \mathscr E^{\otimes k}$ is an isomorphism of left $\cA_Y$-modules for some $k \in \Z_{>0}$, \label{item:natIsomSomePower}
      \item $\mathscr E^{\otimes k} \to \mathscr D_Y \otimes_{\cA_Y} \mathscr E^{\otimes k}$ is an isomorphism of left $\cA_Y$-modules for all $k \in \Z_{>0}$. \label{item:natIsomAllPowers}
    \end{enumerate}
\end{prop}

\begin{proof}
  First, we show that the first two items are equivalent:
  By transitivity of the group action, $\tilde Z_Y \colon \cA_Y \to \mathscr{D}_Y$ is surjective, hence the natural homomorphism of $\cA_Y$-modules $\mathscr E \to \mathscr D_Y \otimes_{\cA_Y} \mathscr E$ is also surjective. Since the support of $\mathscr D_Y \otimes_{\cA_Y} \mathscr E$ is a $G'$-invariant subset of $Y$, by transitivity we must either have $\mathscr D_Y \otimes_{\cA_Y} \mathscr E = 0$ or $\Supp(\mathscr D_Y \otimes_{\cA_Y} \mathscr E) = Y$. Since $\mathscr E$ is a line bundle on $Y$, the only quotient of the $\O_Y$-module $\mathscr{E}$ with support equal to $Y$ is $\mathscr{E}$ itself. This shows \cref{item:nonZero}~$\Leftrightarrow$~\cref{item:natIsom}.

  The implication \cref{item:natIsomAllPowers}~$\Rightarrow$~\cref{item:natIsomSomePower} is trivial. To show the implication \cref{item:natIsom}~$\Rightarrow$~\cref{item:natIsomAllPowers}, we assume for contradiction that there is some $k \geq 2$ for which the claim does not hold and assume $k$ to be minimal. Applying \cref{lem:tensorProductsOnAandD}.\cref{item:leftCNMN} to $\mathcal M := \mathscr E^{\otimes (k-1)}$ and $\mathcal N := \mathscr D_Y \otimes_{\cA_Y} \mathscr E$ gives:
    \[\mathscr D_Y \otimes_{\cA_Y} \mathscr E^{\otimes k} \cong (\mathscr D_Y \otimes_{\cA_Y} \mathscr E^{\otimes (k-1)}) \otimes_{\O_Y} (\mathscr D_Y \otimes_{\cA_Y} \mathscr E) \cong \mathscr{E}^{\otimes (k-1)} \otimes_{\O_Y} \mathscr{E} = \mathscr{E}^{\otimes k}\]
  as left $\cA_Y$-modules (by minimality of $k$). This is a contradiction to the choice of $k$.

  It remains to show the implication \cref{item:natIsomSomePower}~$\Rightarrow$~\cref{item:natIsom}. Consider the two-sided ideal
  \[\mathcal I := \ker\left(\tilde Z_Y \colon \cA_Y \twoheadrightarrow \mathscr{D}_Y\right)\]
  of $\cA_Y$.
  Note that the natural homomorphism $\mathscr{E} \to \mathscr{D}_Y \otimes_{\cA_Y} \mathscr{E}$ of left $\cA_Y$-modules is an isomorphism if and only if $\mathcal{I}$ annihilates $\mathscr{E}$. Using \cref{lem:kernelGeneratedByVFs}, it suffices to prove that $\mathscr{E}$ is annihilated by $\ker(Z_Y)$. Let $s \in \Gamma(U,\mathscr{E})$ be a non-zero local section of $\mathscr{E}$ and let $P \in \Gamma(U,\ker(Z_Y)) \subseteq \O_U \otimes \mathfrak g'$.
      By assumption \cref{item:natIsomSomePower}, we have $\mathscr E^{\otimes k} \cong \mathscr D_Y \otimes_{\cA_Y} \mathscr E^{\otimes k}$ as left $\cA_Y$-modules for some $k \geq 1$, meaning that $\mathscr{E}^{\otimes k}$ is annihilated by $\mathcal I$. In particular, the local section $s^k \in \Gamma(U,\mathscr{E}^{\otimes k})$ is annihilated by $P$, so $P \cdot s^k = 0$. On the other hand, we have
  \[P \cdot s^k = k s^{k-1} (P \cdot s).\]
  Since $Y$ is an irreducible variety, we deduce that $P \cdot s = 0$. This concludes the proof.
\end{proof}

\begin{cor} \label{prop:torsionGivesNonZero}
  Assume $G'$ acts transitively on $Y$. Let $\mathscr{E}$ be a torsion element of the equivariant Picard group $\Pic^{G'}(Y)$, i.e., $\mathscr{E}^{\otimes k} \cong \O_Y$ as equivariant line bundles for some $k \in \Z_{>0}$. Then
  the natural homomorphism
  \[\mathscr{E} \to \mathscr{D}_Y \otimes_{\cA_Y} \mathscr{E}\]
  of left $\cA_Y$-modules is an isomorphism.
    \end{cor}

\begin{proof}
  By \cref{prop:nonZeroIndependentOfPower}, it suffices to consider the case that $\mathscr{E} = \O_Y$ as equivariant line bundles. The Lie algebra $\mathfrak g'$ acts trivially on the $1$-section of $\O_Y$, hence
  \[\mathscr E \cong \cA_Y/\cA_Y(\xi \mid \xi \in \mathfrak g')\]
  as left $\cA_Y$-modules. Tensoring with $\mathscr{D}_Y$ over $\cA_Y$ gives
  \[\mathscr{D}_Y \otimes_{\cA_Y} \mathscr E \cong \mathscr{D}_Y/\mathscr{D}_Y(Z_Y(\xi) \mid \xi \in \mathfrak g') = \mathscr{D}_Y/\mathscr{D}_Y \Theta_Y \cong \O_Y \cong \mathscr E.\]
  Here, we use that the vector fields $Z_Y(\xi)$ for $\xi \in \mathfrak g'$ generate the tangent bundle $\Theta_Y$, as the action of $G'$ on $Y$ is transitive.
\end{proof}

\begin{rmk} \label{rem:torsionGivesNonZero}
  Note from the proof above that the equivalences of 2.,\ 3.\ and 4.\ in \cref{prop:nonZeroIndependentOfPower} hold more generally for any line bundle $\mathscr{E}$ with a left $\cA_Y$-module structure, not necessarily arising from $G'$-equivariant structure on $\mathscr{E}$. The equivalence with 1.\ moreover holds whenever $\mathscr D_Y \otimes_{A_Y} \mathscr E$ is known to have $G'$-invariant support (as will be the case for example if we know that some positive power of $\mathscr E$ underlies a $G'$-equivariant line bundle, or if $\mathscr E$ is a twist of a $G'$-equivariant line bundle by a Lie algebra homomorphism as we will consider in \cref{sec:Twists}).
\end{rmk}

\cref{prop:torsionGivesNonZero} shows in particular that $\mathscr{D}_Y \otimes_{\cA_Y} \mathscr{E} \neq 0$ for $G'$-equivariant torsion line bundles. Under certain assumptions on $Y$, the converse is also true:

\begin{prop} \label{prop:nonTorsionGivesZero}
  Let $G'$ act transitively on $Y$ and assume that there is an open cover $Y = \bigcup_{i \in I} U_i$ such that for each $i \in I$, there is a subgroup $N_i$ of $G$ acting freely and transitively on $U_i$. Then
    \[\mathscr{D}_Y \otimes_{\cA_Y} \mathscr{E} \neq 0
      \quad \Leftrightarrow \quad
      \mathscr{E} \cong \O_Y \text{ as $G'$-equivariant line bundles}.\]
\end{prop}

We remark that under the assumptions on $Y$ in \cref{prop:nonTorsionGivesZero}, there are no non-trivial equivariant torsion line bundles on $Y$.

\begin{proof}
  One implication is given by \cref{prop:torsionGivesNonZero}. For the converse, we assume that $\mathscr{D}_Y \otimes_{\cA_Y} \mathscr{E} \neq 0$. Since $\mathscr{E}$ is $G'$-equivariant, the support of this $\mathscr{D}_Y$-module is a non-empty $G'$-invariant subset of $Y$, hence (by transitivity of the group action)
  \begin{equation} \label{eq:invariantSupport}
    \Supp( \mathscr{D}_Y \otimes_{\cA_Y} \mathscr{E}) = Y.
  \end{equation}
  In particular, the restriction to $U_i$ is a non-zero $\mathscr{D}_{U_i}$-module for each $i \in I$.

  Denote by $E^*$ the complement of the zero section of $E = \Tot(\mathscr E) \xrightarrow{\pi} Y$. For $i \in I$, the choice of a point $w_i \in E^*$ such that $p_i := \pi(w_i) \in U_i$ determines a local section $s_i \in \Gamma(U_i, \mathscr E)$ geometrically given by
  \begin{align*}
    s_i \colon \quad U_i \ \xrightarrow{\cong} \ &N_i \ \to \ \pi^{-1}(U_i) \\
    g \cdot p_i \ \mapsfrom \ &g \quad \mapsto \ g \cdot w_i.
  \end{align*}
  Here, we use that $N_i \to U_i$, $g \mapsto g \cdot p_i$ is an isomorphism by Zariski's Main Theorem, since it is bijective (as $N_i$ is assumed to act freely and transitively on $U_i$) and $U_i \subseteq Y$ is normal. Since $E^*$ is invariant under the action of $G'$ on $E$, the local section $s_i$ does not vanish on $U_i$, hence $\restr{\mathscr{E}}{U_i} = \O_{U_i} s_i$.

  By definition, $s_i$ is an $N_i$-invariant section of $\restr{\mathscr{E}}{U_i}$, hence $\xi \cdot s_i = 0$ holds for all $\xi \in   \Lie(N_i) =: \mathfrak n_i$. Since $N_i$ acts transitively on $U_i$, the $\O_{U_i}$-module homomorphism $\O_{U_i} \otimes \mathfrak n_i \to \Theta_{U_i}$ is surjective, so from the above we may deduce that $\Theta_{U_i}$ annihilates the cyclic $\mathscr{D}_{U_i}$-module $\restr{(\mathscr{D}_Y \otimes_{\cA_Y} \mathscr{E})}{U_i}$ generated by $1 \otimes s_i$.

  Take any $\xi \in \mathfrak g'$. Then $\xi \cdot s_i = f \cdot s_i$ for some $f \in \Gamma(U_i,\O_{U_i})$. But then $f$ annihilates $\restr{(\mathscr{D}_Y \otimes_{\cA_Y} \mathscr{E})}{U_i}$, as $f \cdot (1 \otimes s_i) = 1 \otimes (\xi \cdot s_i) = \restr{Z_Y(\xi)}{U_i} \cdot (1 \otimes s_i) = 0$. Because of \eqref{eq:invariantSupport}, this forces
    \[\xi \cdot s_i = 0 \qquad \text{for all }\xi \in \mathfrak g'.\]

  On $U_{ij} := U_i \cap U_j$ for $i,j \in I$, the non-vanishing local sections $s_i$ and $s_j$ only differ by an invertible function:
    \[\restrK{s_i}{U_{ij}} = \alpha_{ij} \restrK{s_j}{U_{ij}}, \qquad \alpha_{ij} \in \Gamma(U_{ij},\O_{U_{ij}}^\times).\]
  Since
  \[0 = \xi \cdot \restrK{s_i}{U_{ij}} = \xi \cdot (\alpha_{ij} \restrK{s_j}{U_{ij}}) = \restr{Z_Y(\xi)}{U_{ij}}(\alpha_{ij}) \restrK{s_j}{U_{ij}} + \alpha_{ij} (\xi \cdot \restrK{s_j}{U_{ij}}) = \restr{Z_Y(\xi)}{U_{ij}}(\alpha_{ij}) \restrK{s_j}{U_{ij}},\]
  we see that $\alpha_{ij} \neq 0$ is annihilated by all vector fields on $U_{ij}$ (since $\Theta_Y$ is globally generated by the image of $Z_Y \colon \O_Y \otimes \mathfrak g' \to \Theta_Y$). Therefore, $\alpha_{ij} \in \C^*$.

  We may now fix some $k \in I$ and define non-vanishing sections
  \[\tilde s_i := \alpha_{ki}^{-1} s_i \in \Gamma(U_i,\mathscr{E}) \qquad \text{for all } i \in I\]
  which are still annihilated by the action of $\mathfrak g'$.
  Then $\tilde s_i$ and $\tilde s_j$ agree on $U_{ij}$ for all $i,j \in I$, so they glue to a global non-vanishing section $\tilde s \in \Gamma(Y,\mathscr{E})$ annihilated by $\mathfrak g'$. This section defines an isomorphism $\mathscr{E} \cong \O_Y$ of left $\cA_Y$-modules and hence of $G'$-equivariant line bundles.
\end{proof}

\subsection{Twist by characters and non-vanishing of tautological systems} \label{sec:Twists}

Next we relate the construction from the previous section  to the $\mathscr D$-modules $\mathscr N_Y^\beta$ from \cref{def:NBeta}. Recall from \cref{cor:FLTautSysAsDirectImage} that these $\mathscr D$-modules describe restrictions of Fourier-transformed tautological systems and hence we obtain in \cref{thm:nonZeroGeneralSetup} below a non-vanishing result for tautological systems $\tau(\rho, \overline Y, \beta)$ based on the non-vanishing of $\mathscr N_Y^\beta$.

To start with, we need to consider twists of equivariant line bundles by characters:

\begin{dfn}\label{def:EquLineChar}
  Let $\chi \colon G' \to \C^*$ be a character. We define a $G'$-equivariant line bundle $\O_Y\{\chi\}$ on $Y$ by equipping the trivial line bundle $\O_Y$ with a $G'$-equivariant structure such that the action of $G'$ on $\Tot(\O_Y\{\chi\}) = \C \times Y$ is given by $g \cdot (\lambda,y) = (\chi(g)\lambda, \, g \cdot y)$.

  For any $G'$-equivariant line bundle $\mathscr{E}$, consider the $G'$-equivariant line bundle
  \[\mathscr{E}\{\chi\} := \mathscr{E} \otimes_{\O_Y} \O_Y\{\chi\},\]
  which has the same underlying $\O_Y$-module, but a different equivariant structure.
\end{dfn}

One easily checks that
$\operatorname{Hom}(G', \C^*) \to \Pic^{G'}(Y)$, $\chi \mapsto \O_Y\{\chi\}$ is a group homomorphism.

\begin{rmk}
  For a given equivariant line bundle $\mathscr{E}$ whose $G'$-action is given on $E := \Tot(\mathscr{E})$ as $\varphi \colon G' \times E \to E$, the $G'$-action on $\Tot(\mathscr{E}\{\chi\}) = E$ is given by
  \[G' \times E \to E, \qquad (g,e) \mapsto \mu\big(\chi(g),\varphi(g,e)\big),\]
  where $\mu \colon \C^* \times E \to E$ denotes the natural $\C^*$-action on $E$ by scaling fibers.
\end{rmk}

We have seen before that every $G'$-equivariant line bundle on $Y$ is a left $\cA_Y$-module, so for every character $\chi \colon G' \to \C^*$, we get the left $\cA_Y$-module
\[\O_Y\{\chi\} \cong \cA_Y/\cA_Y (\xi - \differential\chi(\xi) \mid \xi \in \mathfrak g'),\]
where $\differential\chi \colon \mathfrak g' \to \C$ is the Lie algebra homomorphism induced by $\chi$.

Note that the left $\cA_Y$-module structure on a $G'$-equivariant line bundle $\mathscr{E}$ results just from the infinitesimal action of $\mathfrak g'$ on local sections of $\mathscr{E}$. Therefore, it is natural to make the following more general definition:
\begin{dfn} For any Lie algebra homomorphism $\beta \colon \mathfrak g' \to \C$, we define the left $\mathscr A_Y$-module
\[\O_Y\{\beta\} := \cA_Y/\cA_Y (\xi - \beta(\xi) \mid \xi \in \mathfrak g').\]
If $\mathscr E$ is a left $\cA_Y$-module, then we denote by $\mathscr E \{\beta\}$ the left $\cA_Y$-module $\mathscr{E} \otimes_{\O_Y} \O_Y\{\beta\}$.
\end{dfn}
This may in general not be a $G'$-equivariant line bundle. Note that $\O_Y\{\chi\} \cong \O_Y\{\differential \chi\}$ as $\cA_Y$-modules for $\chi \colon G' \to \C^*$ inducing $\differential\chi \colon \mathfrak g' \to \C$. Similarly to before, given a left $\cA_Y$-module $\mathscr E$, we denote by $\mathscr E \{\beta\}$ the left $\cA_Y$-module $\mathscr{E} \otimes_{\O_Y} \O_Y\{\beta\}$. If $\mathscr E$ is a line bundle with a left $\cA_Y$-module structure, we denote $\mathscr (\mathscr E\{\beta\})^\vee := \mathscr E^\vee\{-\beta\}$.

Recall from \cref{def:NBeta} that for any Lie algebra homomorphism $\beta \colon \mathfrak g' \to \C$ on a smooth connected $G'$-variety $Y$, we defined the left $\mathscr{D}_Y$-module
  \[\mathscr{N}_{Y}^{\beta} := \omega_Y^\vee \otimes_{\O_Y} \mathscr{D}_Y / (Z_Y(\xi)-\beta(\xi) \mid \xi \in \mathfrak g') \mathscr{D}_Y.\]
Our next aim is to show the following result describing this $\mathscr{D}_Y$-module as arising from a $\mathfrak g'$-module structure on the anticanonical bundle on $Y$:

\begin{prop} \label{prop:interpretationOfNBeta}
  Let $\beta \colon \mathfrak g' \to \C$ be a Lie algebra homomorphism. Considering $\omega_Y$ with its natural $G'$-equivariant structure, there is an isomorphism of left $\mathscr{D}_Y$-modules
  \[\mathscr{D}_Y \otimes_{\cA_Y} (\omega_Y\{\beta\})^\vee \cong \mathscr{N}_{Y}^{\beta}.\]
\end{prop}

For the proof of \cref{prop:interpretationOfNBeta}, we need some technical remarks on left-right transforms of $\cA_Y$-modules that we carry out first:

The line bundle $\alpha_Y := \bigwedge^{\dim G'} (\O_Y \otimes_\C \mathfrak g')^\vee$ on $Y$ has the structure of a right $\cA_Y$-module which is given by the negated Lie derivative: A Lie algebra element $\xi \in \mathfrak g'$ acts on an alternating form $\omega$ by mapping it to the alternating form $\omega \cdot \xi$ given by
\[(\omega \cdot \xi)(\theta_1,\dots,\theta_m) = -Z_Y(\xi)(\omega(\theta_1,\dots,\theta_m)) + \sum_{i=1}^m \omega(\theta_1,\dots,[\xi,\theta_i],\dots,\theta_m)\]
for any $\theta_1,\dots,\theta_m \in \O_Y \otimes_\C \mathfrak g'$. This defines transformations between left and right $\cA_Y$-modules giving rise to an equivalence of categories
\begin{align*}
  \operatorname{Mod}(\cA_Y) &\xrightarrow{\qquad \cong \qquad} \operatorname{Mod}(\cA_Y^{\text{op}}), \\
  \mathcal M &\qquad\longmapsto \qquad \alpha_Y \otimes_{\O_Y} \mathcal M, \\
  \alpha_Y^\vee \otimes_{\O_Y} \mathcal M' &\qquad\longmapsfrom \qquad \mathcal M'.
\end{align*}

\begin{rmk} \label{rem:explicitLeftRightTransform}
  If $\xi_1,\dots,\xi_m$ form a $\C$-basis of $\mathfrak g'$, then
  \[\alpha_Y = \O_Y \, \xi_1^* \wedge \dots \wedge \xi_m^*\]
  The right action on $\alpha_Y$ is given by
  \[(f\, \xi_1^* \wedge \dots \wedge \xi_m^*) \cdot \xi = \big(\!\trace(\ad(\xi))f-Z_Y(\xi)(f)\big) \, \xi_1^* \wedge \dots \wedge \xi_m^* \qquad \text{for } \xi \in \mathfrak g'.\]
  In general, if $\mathcal M$ is a left $\cA_Y$-module, then the right $\cA_Y$-module structure on $\alpha_Y \otimes_{\O_Y} \mathcal M$ is given by
  \[(\xi_1^* \wedge \dots \wedge \xi_m^* \otimes s) \cdot \xi = \xi_1^* \wedge \dots \wedge \xi_m^* \otimes \big(\!\trace(\ad(\xi))-\xi\big) \cdot s \quad \text{for } \xi \in \mathfrak g', s \in \mathcal M.\]
\end{rmk}

The canonical bundle $\omega_Y$ on $Y$ is a right $\mathscr{D}_Y$-module and hence, via $\tilde Z_Y \colon \cA_Y \to \mathscr D_Y$, it also has the structure of a right $\cA_Y$-module. On the other hand, the action of $G'$ on $Y$ extends naturally to an action on the tangent bundle on $Y$, so $\omega_Y = \bigwedge^{\dim Y} \Theta_Y^\vee$ is naturally a $G'$-equivariant line bundle, which induces a left $\cA_Y$-module structure. The next lemma states that these left and right module structures on $\omega_Y$ relate to each other via the transformation above:

\begin{lem} \label{lem:rightAModuleStructuresOnOmega}
  Let $\delta := \trace \circ \ad \colon \mathfrak g' \to \C$ and let $\xi_1,\dots,\xi_m$ form a $\C$-basis of $\mathfrak g'$. There is an isomorphism of right $\cA_Y$-modules
  \begin{align*}
    \omega_Y &\xrightarrow{\cong} \alpha_Y \otimes_{\O_Y} \omega_Y\{\delta\} \\
    s &\mapsto \xi_1^* \wedge \dots \wedge \xi_m^* \otimes s,
  \end{align*}
  where on the left hand side, $\omega_Y$ is endowed with its right $\cA_Y$-module structure induced from the homomorphism $\tilde Z_Y \colon \cA_Y \to \mathscr D_Y$, and on the right hand side, we consider $\omega_Y$ with its left $\cA_Y$-module structure by viewing it as a $G'$-equivariant line bundle.
\end{lem}

\begin{proof}
  Denote $a := \xi_1^* \wedge \dots \wedge \xi_m^* \in \Gamma(Y,\alpha_Y)$ and recall that $\xi_1,\dots,\xi_m$ are a $\C$-basis of $\mathfrak g'$. Since $a$ is a non-vanishing global section of the line bundle $\alpha_Y$, the homomorphism $\omega_Y \to \alpha_Y \otimes_{\O_Y} \O_Y\{\delta\} \otimes_{\O_Y} \omega_Y$, $s \mapsto a \otimes 1 \otimes s$ is an isomorphism of $\O_Y$-modules, hence it suffices to show:
  \[(a \otimes 1 \otimes s) \cdot \xi \stackrel{!}{=} a \otimes 1 \otimes (s \cdot \xi)\]
  for $s \in \omega_Y$, $\xi \in \mathfrak g'$.

  The right $\cA_Y$-module structure on $\omega_Y$ (inherited from the right $\mathscr{D}_Y$-module structure) is given by
  \[(s \cdot \xi)(\theta_1,\dots,\theta_n) = -Z_Y(\xi)(s(\theta_1,\dots,\theta_n))+\sum_{i=1}^n s(\theta_1,\dots,[Z_Y(\xi),\theta_i],\dots,\theta_n)\]
  for $\xi \in \mathfrak g'$, $s \in \omega_Y$, $\theta_1,\dots, \theta_n \in \Theta_Y$. On the other hand, the right $\cA_Y$-module $\alpha_Y$ satisfies $a \cdot \xi = \delta(\xi) a$, so the right $\cA_Y$-module structure on $\alpha_Y \otimes_{\O_Y} \O_Y\{\delta\}$ satisfies
  \[(a \otimes 1) \cdot \xi = 0.\]

  The left $\cA_Y$-module structure on $\omega_Y$ results from the left $\cA_Y$-module structure on the $G'$-equivariant vector bundle $\Theta_Y$ given by
  \[\xi \cdot \theta = [Z_Y(\xi),\theta] \qquad \text{for all } \xi \in \mathfrak g'\]
  The induced left $\cA_Y$-module structure on $\bigwedge^n \Theta_Y$ is given by
  \[\xi \cdot (\theta_1 \wedge \dots \wedge \theta_n) = \sum_{i=1}^n \theta_1 \wedge \dots \wedge [Z_Y(\xi),\theta_i] \wedge \dots \wedge \theta_n\]
  for $\xi \in \mathfrak g'$. Passing to the dual line bundle $\omega_Y$, we get
  \begin{align*}(\xi \cdot s)(\theta_1,\dots,\theta_n) &= Z_Y(\xi)(s(\theta_1,\dots,\theta_n))-\sum_{i=1}^n s(\theta_1,\dots,[Z_Y(\xi),\theta_i],\dots,\theta_n) \\
    &= -(s \cdot \xi)(\theta_1,\dots,\theta_n) \qquad \text{for all }\xi \in \mathfrak g'.
  \end{align*}
  The right $\cA_Y$-module structure on $\alpha_Y \otimes_{\O_Y} \O_Y\{\delta\} \otimes_{\O_Y} \omega_Y$ resulting from this satisfies
  \[(a \otimes 1 \otimes s) \cdot \xi
  = ((a \otimes 1) \cdot \xi) \otimes s - (a \otimes 1) \otimes (\xi \cdot s) = a \otimes 1 \otimes (s \cdot \xi)\]
  for $\xi \in \mathfrak g'$.
\end{proof}

We can now turn to the proof of the description of $\mathscr N_Y^\beta$ via the (twisted) equivariant
anti-canonical bundle:

\begin{proof}[Proof of \cref{prop:interpretationOfNBeta}]
  Equivalently to the claim, we may show that there is an isomorphism between the corresponding right $\mathscr{D}_Y$-modules
  \[\omega_Y \otimes_{\O_Y} \big(\mathscr{D}_Y \otimes_{\cA_Y} (\omega_Y\{\beta\})^\vee\big) \stackrel{!}{\cong} \mathscr{D}_Y / (Z_Y(\xi)-\beta(\xi) \mid \xi \in \mathfrak g') \mathscr{D}_Y.\]
  By \cref{lem:tensorProductsOnAandD}.\cref{item:rightCNMN}, we have an isomorphism of right $\mathscr{D}_Y$-modules
  \[\omega_Y \otimes_{\O_Y} \big(\mathscr{D}_Y \otimes_{\cA_Y} (\omega_Y\{\beta\})^\vee\big) \cong (\omega_Y \otimes_{\O_Y} (\omega_Y\{\beta\})^\vee) \otimes_{\cA_Y} \mathscr{D}_Y,\]
  where, on the right hand side, the first occurrence of $\omega_Y$ is equipped with the right $\cA_Y$-module structure inherited from its right $\mathscr{D}_Y$-module structure. Combining this with \cref{lem:rightAModuleStructuresOnOmega}, we obtain:
  \begin{align*}
    \omega_Y \otimes_{\O_Y} \big(\mathscr{D}_Y \otimes_{\cA_Y} (\omega_Y\{\beta\})^\vee\big)
    &\cong (\alpha_Y \otimes_{\O_Y} \omega_Y\{\delta\} \otimes_{\O_Y} (\omega_Y\{\beta\})^\vee) \otimes_{\cA_Y} \mathscr{D}_Y \\
    &\cong (\alpha_Y \otimes_{\O_Y} \omega_Y \otimes_{\O_Y} \omega_Y^\vee \otimes_{\O_Y} \O_Y\{\delta-\beta\}) \otimes_{\cA_Y} \mathscr{D}_Y
  \end{align*}
  where $\delta := \trace \circ \ad \colon \mathfrak g' \to \C$ and $\omega_Y$ is now     considered as a left $\cA_Y$-module via its natural structure as a $G'$-equivariant line bundle. Since $\omega_Y \otimes_{\O_Y} \omega_Y^\vee \cong \O_Y$ as $G'$-equivariant line bundles (and therefore also as left $\cA_Y$-modules), we conclude:
  \[\omega_Y \otimes_{\O_Y} \big(\mathscr{D}_Y \otimes_{\cA_Y} \omega_Y^\vee\{\beta\}\big) \cong (\alpha_Y \otimes_{\O_Y} \O_Y\{\delta-\beta\}) \otimes_{\cA_Y} \mathscr{D}_Y.\]
  Recall that $\O_Y\{\delta-\beta\} \cong \cA_Y/\cA_Y(\xi-(\delta-\beta)(\xi) \mid \xi \in \mathfrak g')$, so by \cref{rem:explicitLeftRightTransform}, we have
  \[\alpha_Y \otimes_{\O_Y} \O_Y\{\delta-\beta\} \cong \cA_Y/(\xi-\beta(\xi) \mid \xi \in \mathfrak g') \cA_Y\]
  as right $\cA_Y$-modules. Tensoring with $\mathscr{D}_Y$ over $\cA_Y$ by means of the homomorphism $\tilde{Z}_Y \colon \cA_Y \to \mathscr D_Y$ yields the claimed result.
\end{proof}

By using all the constructions and results of this section,  we get the following non-vanishing theorem for tautological systems:

\begin{thm} \label{thm:nonZeroGeneralSetup}
  Let $\rho \colon G' \to \GL(V)$ be a finite-dimensional rational representation. Let $Y \subseteq V$ be a $G'$-orbit and let $\overline{Y}$ be its closure. Let $\beta \colon \mathfrak g' \to \C$ be a Lie algebra homomorphism. If $(\omega_Y\{\beta\})^{\otimes k} \cong \O_Y$ for some $k\in \Z$ as left $\cA_Y$-modules,
    then $\hat{\tau}(\rho, \overline{Y}, \beta) \neq 0$, and hence also $\tau(\rho, \overline{Y}, \beta) \neq 0$.
\end{thm}

\begin{proof}
  By \cref{cor:FLTautSysAsDirectImage}, we have $i_+ \mathscr N_Y^{\beta} \cong \restr{\hat\tau(\rho,\overline Y, \beta)}{U}$, where $i$ denotes the closed embedding of $Y$ into $U := V \setminus \partial Y$ for $\partial Y := \overline Y \setminus Y$. With \cref{prop:interpretationOfNBeta}, we conclude that
  \[\restr{\hat{\tau}(\rho,\overline Y,\beta)}{U} \cong i_+\big(\mathscr{D}_{Y} \otimes_{\cA_{Y}} (\omega_{Y}\{\beta\})^\vee\big)\]
  as left $\mathscr{D}_{L^*}$-modules. To show that the right hand side is non-zero, it suffices to see that we have $\mathscr{D}_{Y} \otimes_{\cA_{Y}} (\omega_{Y}\{\beta\})^\vee \neq 0$.
  But this follows from \cref{prop:torsionGivesNonZero} respectively \cref{rem:torsionGivesNonZero}, because we assumed that $(\omega_{Y}\{\beta\})^{\otimes k} \cong \O_Y$ as left $\cA_Y$-modules.
\end{proof}

\subsection{Application to projective homogeneous spaces} \label{ssec:canSheafOnLineBdl}

We now apply the previous results to the following setup: Consider a smooth projective variety $X$ with a transitive action of a reductive connected linear algebraic group $G$ (i.e., $X$ is a \emph{homogeneous space}). Let $L \to X$ be a $G$-equivariant line bundle on $X$ with sheaf of sections $\Ell$. We consider $G' := G \times \C^*$ and denote the Lie algebras involved by $\mathfrak g' := \Lie(G')$, $\mathfrak g := \Lie(G)$ and $\Lie(\C^*) = \C \mathbf{e}$, so
\[\mathfrak g' = \mathfrak g \oplus \C \mathbf{e}.\]
We view $\Ell$ as a $G'$-equivariant line bundle on $X$ by letting the $\C^*$-factor of $G'$ act trivially on $X$ and by inverse scaling on the fibers of $L \to X$. Note that then $G'$ acts transitively on $L^*$.
Denote by $L^* \subseteq L$ the complement of its zero section. The morphisms to $X$ are denoted $\pi^L \colon L \to X$ and $\pi^{L^*} \colon L^* \to X$.

\begin{lem}\label{lem:transitCond}
  Every point of $X$ admits an open neighborhood on which a subgroup of $G$ acts freely and transitively. The same holds for the action of $G'$ on $L^*$.
\end{lem}

\begin{proof}
  For any point $p \in X$, the stabilizer $P := \{g \in G \mid g \cdot p = p\}$ describes the variety as a quotient:
   \[G/P \xrightarrow{\cong} X, \qquad gP \mapsto g \cdot p.\]
  Since $X$ is projective, the subgroup $P \subseteq G$ is parabolic (this can be taken as the definition of parabolic subgroups, see, e.g., \cite[\S21.3]{Hum75}). Let $N^{-} \subseteq G$ be the unipotent radical of the opposite parabolic subgroup to $P$ in $G$. Then $N^{-} \cap P = 1$, which shows that $N^{-}$ acts freely and transitively on the $N^{-}$-orbit $N^{-} \cdot p$. On the other hand, we have $\Lie(N^{-}) \oplus \Lie(P) = \mathfrak g$ as $\C$-vector spaces, so $N^{-} \cdot p$ is of dimension $\dim G - \dim P = \dim X$, hence it is an open neighborhood of $p$ in $X$.

  For the $G'$-action on $L^*$, take any point $w \in L^*$ and consider $p = \pi^{L^*}(w)$ in the argument above. Then $\pi^{L^*,-1}(U)$ is an open neighborhood of $w$ in $L^*$ on which $\C^* \times N^{-} \subseteq G'$ acts transitively and freely.
\end{proof}

In the following, we remark that the assumptions on $X$ guarantee that we are in the setup of \cref{sec:OBeta}.

\begin{rmk}
Projective homogeneous spaces $X \cong G/P$ are smooth Fano varieties (we recall a short representation-theoretic argument in \cref{lem:Fano} below). As such, it has the property that the underlying complex manifold $X^{\mathrm{an}}$ is simply-connected (see e.g.~\cite[Corollary~4.29]{Debarre-2001}). In particular, we may apply \cref{prop:PiEinsLStern} to $X^{\mathrm{an}}$ to get $k \in \Z_{>0}$ with $\pi_1(L^{*,\mathrm{an}}) \cong \Z/k\Z$ and we may consider the $\mathscr D_{L^*}$-modules $\O_{L^*}^{\ell/k}$ for $\ell \in \Z$ as in \cref{def:Obeta-new}.

Additionally, the Fano property of $X$ implies $H^i(X,\O_X) = 0$ for all $i > 0$ by the Kodaira vanishing theorem, hence in particular $\Pic(X) = H^1(X,\O_X^\times) \cong H^2(X^{\mathrm{an}},\Z)$. In algebraic terms, the integer $k$ in the statement of \cref{prop:PiEinsLStern} is therefore the largest positive integer such that $\Ell$ admits a $k$-th root in the Picard group. Moreover, $\Pic(X) \cong H^2(X^{\mathrm{an}},\Z)$ has no torsion, as observed in the proof of \cref{prop:PiEinsLStern}.
\end{rmk}

\begin{lem} \label{lem:invariantSection}
  Let $\mathcal M$ be a $G'$-equivariant line bundle on $X$. Then
  \[\pi^{L^*,*} \mathcal M \cong \O_{L^*} \quad \Leftrightarrow \quad \mathcal M \cong \Ell^{\otimes r} \text{ for some $r \in \Z$},\]
  where both sides are isomorphisms of $G'$-equivariant line bundles.
\end{lem}

\begin{proof}
  For the implication ``$\Leftarrow$", it suffices to consider the case $r = 1$. Note that the $G'$-equivariant structure on $\pi^{L^*,*} \Ell$ corresponds to the diagonal $G'$-action on $\Tot(\pi^{L^*,*} \Ell) = L^* \times_X L \xrightarrow{pr_1} L^*$. Note further that the map
  \[s \colon L^* \xrightarrow{\Delta} L^* \times_X L^* \hookrightarrow L^* \times_X L\]
  is a $G'$-invariant global section of the line bundle $\pi^{L^*,*} \Ell$ that vanishes nowhere on $L^*$. Then
  \[\O_{L^*} \to \pi^{L^*,*} \Ell, \qquad 1 \mapsto s\]
  is an isomorphism of $G'$-equivariant line bundles.

  To show the implication ``$\Rightarrow$", let $\mathcal M$ be a $G'$-equivariant line bundle on $X$ with $\pi^{L^*,*} \mathcal M \cong \O_{L^*}$. This means that there is a $G'$-invariant global section of $\pi^{L^*,*} \mathcal M$ which we may view as a morphism
  \[s \colon L^* \to \Tot(\pi^{L^*,*} \mathcal M) = L^* \times_X M,\]
  where $M = \Tot(\mathcal M)$ and where $s$ is the identity on the first component. Since the section is non-vanishing, it is therefore given by $s = \id_{L^*} \times \varphi$ for some morphism \[\varphi \colon L^* \to M^*\]
  over $X$. The $G'$-invariance of the section $s$ translates to $G'$-equivariance of the morphism $\varphi$. Recall that the $\C^*$-factor of $G'=G \times \C^*$ acts trivially on $X$ and by inverse scaling on the fibers of $\pi^{L} \colon L \to X$. Notice that on $M$, the $\C^*$-action must also be given fiberwise, so there exists some $r \in \Z$ such that $\C^* \subseteq G'$ acts on fibers of $\pi^{M}\colon M \to X$ by scaling with $(-r)$-th powers. In particular, the $\C^*$-equivariance of $\varphi \colon L^* \to M^*$ implies the following fiberwise description: Over $p \in X$, we have
  \begin{align*}
    \varphi_p \colon L_p^* \xrightarrow{\cong} \C^* &\xrightarrow{\lambda \mapsto \lambda^r} \C^* \xrightarrow{\cong} M_p^* \\[-0.25em]
    \lambda w \ \mapsfrom \ \lambda \ &\phantom{\xrightarrow{\lambda \mapsto \lambda^r}{}} \ \ \mu \ \mapsto \ \mu \varphi(w)
  \end{align*}
  for any choice of $w \in L_p^*$.
    From this, we can conclude $\mathcal M \cong \Ell^{\otimes r}$, for instance as follows: Take an open cover $X = \bigcup_{i \in I} U_i$ trivializing $\Ell$ as $\restr{\Ell}{U_i} = \O_{U_i} s_i$ for some choice of non-vanishing local sections $s_i \in \Gamma(U_i,\Ell)$. On $U_{ij} := U_i \cap U_j$, we have $s_i = \alpha_{ij} s_j$ for some $\alpha_{ij} \in \Gamma(U_{ij}, \O_{U_{ij}}^\times)$ and the collection $(\alpha_{ij})_{i,j \in I}$ forms a {\v{C}}ech cocycle whose class in $H^1(X,\O_X^\times) \cong \Pic(X)$ defines the isomorphism class of $\Ell$. Viewing the sections $s_i$ geometrically as morphisms $U_i \to L^*$, we may compose them with $\varphi$ to get non-vanishing local sections $\varphi \circ s_i \in \Gamma(U_i, \mathcal M)$. On $U_{ij}$, we then have $\varphi \circ s_i = (\alpha_{ij}^r) (\varphi \circ s_j)$ since $\varphi$ is given on fibers of $X$ by taking $r$-th powers. This shows that the class of $\mathcal M$ in the Picard group of $X$ is the class of the {\v{C}}ech cocycle $(\alpha_{ij}^r)_{i,j \in I}$ in $H^1(X,\O_X^\times) \cong \Pic(X)$. On the other hand, the same is true for the line bundle $\Ell^{\otimes r}$ (as can e.g.\ be seen in the same way, using the morphism $L^* \to L^{\otimes r,*}$ given by fiberwise $r$-th powers). Hence, $\mathcal M \cong \Ell^{\otimes r}$   and the $G'$-equivariance of this isomorphism follows from the $G'$-equivariance of $\varphi$.
  \end{proof}

\begin{lem} \label{lem:canonicalBundleOnL}
  There is an isomorphism
  \[\omega_L \cong \pi^{L,*} \omega_X \otimes_{\O_L} \pi^{L,*} \Ell^\vee\]
  of $G'$-equivariant line bundles on $L$. In particular,
  $\omega_{L^*} \cong \pi^{L^*,*} \omega_X$
  as $G'$-equivariant line bundles on~$L^*$.
\end{lem}

\begin{proof}
The second claim follows directly from the first claim by pulling back the line bundles to $L^*$ and using \cref{lem:invariantSection}. Hence, it suffices to prove the formula for $\omega_L$.

Let $\mathscr{F}$ be a $G'$-equivariant vector bundle on $X$ and put  $F:=\textup{Tot}(\mathscr{F})$ with projection $\pi^F\colon F\rightarrow X$.
The variety $F$ is then equipped with a $G'$-action. We first claim that there is an isomorphism
$$
\Theta_{F/X} \cong \pi^{F,*}\mathscr{F}.
$$
of $G'$-equivariant vector bundles on $F$,
and a corresponding isomorphism $\Omega^1_{F/X}\cong \pi^{F,*}\mathscr{F}^\vee$ of dual vector bundles. Namely, any section $s\in\Gamma(U,\mathscr{F})$ can be
considered as an element $s\in \Gamma(U,{\cH\!}om_{\cO_X}(\mathscr{F}^\vee,\cO_X))$,
and it extends via the Leibniz rule as a section of
$\Gamma(U,{\cD\!}er_{\cO_X}({\cS\!}ym_{\cO_X}(\mathscr{F}^\vee))$. This yields a $G'$-equivariant morphism of $\cO_X$-modules
$$
\mathscr{F} \longrightarrow
{\cD\!}er_{\cO_X}({\cS\!}ym_{\cO_X}(\mathscr{F}^\vee))=
{\cD\!}er_{\cO_X}(\pi^F_*\cO_F).
$$
It is also injective,
since for any $s\neq 0$, there is some section of $\mathscr{F}^\vee$ that
is not killed by $s$, so that $s$ is not the zero derivation in ${\cD\!}er_{\cO_X}(\pi^F_*\cO_F)$. Since both $\mathscr{F}$ and ${\cD\!}er_{\cO_X}(\pi^F_*\cO_F)$
are locally free of the same rank, it follows that the cokernel of the inclusion
$\mathscr{F} \hookrightarrow {\cD\!}er_{\cO_X}(\pi^F_*\cO_F)$, if not zero, must be a torsion sheaf on $X$, but this is impossible since this map is equivariant, and so is its cokernel.
We conclude that there is an isomorphism
$\mathscr{F} \cong {\cD\!}er_{\cO_X}(\pi^F_*\cO_F)$ of $G'$-equivariant vector bundles on $X$. Applying the functor $\pi^{F,*}$ then
yields an isomorphism
$$
\pi^{F,*}\mathscr{F} \stackrel{\cong}{\longrightarrow} \pi^{F,*}
{\cD\!}er_{\cO_X}(\pi^F_*\cO_F) \stackrel{(\star)}{\cong}
{\cD\!}er_{\pi^{F,-1}\cO_X}(\cO_F) \cong \Theta_{F/X},
$$
of $G'$-equivariant bundles on $F$, as required. Notice that the isomorphism $(\star)$ in the above displayed formula holds since the map $\pi^F$ is affine.

We apply this to the special case $F=L$, i.e., $\rk(\mathscr{F})=1$, to  obtain the $\cO_L$-isomorphism
\begin{equation}\label{eq:relativeTang}
\omega_{L/X}\cong \pi^{L,*}\Ell^\vee,
\end{equation}which again is $G'$-equivariant.

Consider the cotangent sequence
$$
0 \longrightarrow  \pi^{L,*} \Omega^1_X \longrightarrow \Omega^1_L
\longrightarrow \omega_{L/X}\longrightarrow 0,
$$
which, since $\pi^L  \colon L\rightarrow X$ is $G'$-equivariant, is an exact sequence of $G'$-equivariant vector bundles on $L$. Applying $\bigwedge^{\dim(X)+1}_{\cO_L}(-)$ to this sequence, we get the following isomorphism of $G'$-equivariant line bundles on $L$:
$$
\omega_L \cong \pi^{L,*}\omega_X \otimes_{\cO_L} \omega_{L/X}.
$$
Plugging in the isomorphism from Equation \eqref{eq:relativeTang} yields
$$
\omega_L \cong \pi^{L,*} \omega_X \otimes_{\O_L} \pi^{L,*} \Ell^\vee,
$$
as required.

\end{proof}

\begin{prop} \label{prop:NBetaEqualsOBeta}
  Let
  $\beta \colon \mathfrak g' \to \C$ be a Lie algebra homomorphism with $\restr{\beta}{\mathfrak g} \equiv 0$. Let $k \in \Z_{>0}$ be such that $\pi_1(L^{*,\mathrm{an}}) \cong \Z/k\Z$ (see \cref{sec:OBeta}). Then
  \[
  \mathscr N_{L^*}^\beta \cong
  \begin{cases}
    \O_{L^*}^{\ell/k} & \text{if $\exists \ell \in \Z: \beta(\mathbf e) = \ell/k$ and $\Ell^{\otimes \ell} \cong \omega_X^{\otimes (-k)}$ as $G$-equivariant line bundles,} \\
    0 & \text{otherwise.}
  \end{cases}\]
\end{prop}

\begin{proof}
  Recall from \cref{def:NBeta} that
  $$
      \mathscr N_{L^*}^\beta := \omega_{L^*}^\vee \otimes_{\O_{L^*}} \mathscr{D}_{L^*}/(Z_{L^*}(\xi)-\beta(\xi) \mid \xi \in \mathfrak g')\mathscr{D}_{L^*}.
  $$
  We first assume that $\mathscr N_{L^*}^\beta \neq 0$.  Since $G'$ acts transitively on $L^*$, the vector fields $Z_{L^*}(\xi)$, when $\xi$ runs through $\fg'$, generate the tangent bundle of $L^*$. This implies that $\mathscr N_{L^*}^\beta$ is a smooth $\mathscr{D}_{L^*}$-module of rank one, i.e., corresponds to a local system on $L^{*,\mathrm{an}}$.
  By the discussion in \cref{sec:OBeta}, we therefore have an isomorphism of $\cD_{L^*}$-modules
  $\mathscr N_{L^*}^\beta\cong \cO_{L^*}^{\ell/k}$ for some $\ell\in\Z$.     Let $U\subseteq X$ be a Zariski open affine coordinate set (in the algebraic sense, see, e.g. \cite[Definition A.5.2]{Hotta}) such that $L$ trivializes over $U$. Then
  $$
  \mathscr N_{L^*|\dC^*\times U}^\beta \cong
  \cD_{\dC^*\times U}/\cD_{\dC^*\times U}(Z_{L^*}(\xi)_{|\dC^*\times U}^T\mid \xi\in\fg)+
  \cD_{\dC^*\times U}(-\partial_t t -\beta(\mathbf{e})),
  $$
  where we denote by $(-)^T$ the transpose of a differential operator written in the chosen local coordinates.

  Now using \cref{prop:ObetaExtProd} the isomorphism $\cO_{L^*|\dC^*\times U}^{\ell/k}\cong \mathscr N_{L^*|\dC^*\times U}^\beta$ can be made explicit, and then it follows easily that $\beta(\mathbf{e})-\ell/k$ must be an integer, but this yields an isomorphism $ \mathscr N_{L^*}^\beta \cong \cO_{L^*}^{\beta(\mathbf{e})}$ by the remark after \cref{def:Obeta-new}. In particular, this shows that when $\mathscr N_{L^*}^\beta\neq 0$, we must have $\beta(\mathbf{e})\in\frac{1}{k}\Z$.

  In order to show the remaining statements, we therefore assume that $\beta(\mathbf e) = \ell/k \in \Q$ for some $\ell\in\Z$. By \cref{prop:interpretationOfNBeta}, we know
  \[\mathscr{N}_{L^*}^{\beta} \cong \mathscr{D}_{L^*} \otimes_{\cA_{L^*}} (\omega_{L^*}\{\beta\})^\vee.\]
  Because of \cref{lem:transitCond}, we may apply \cref{prop:nonTorsionGivesZero} to decide when a $\mathscr D_{L^*}$-module of the form ${\mathscr D_{L^*} \otimes_{\cA_{L^*}} \mathscr E}$ is non-zero. However, notice that this criterion only applies to equivariant line bundles $\mathscr E$, while $(\omega_{L^*}\{\beta\})^\vee$ is for non-integral $\beta(\mathbf e)$ only a line bundle with an $\cA_{L^*}$-module structure. However, the $k$-th tensor power
  \[(\omega_{L^*}\{\beta\})^{\otimes (-k)} \cong \omega_{L^*}^{\otimes (-k)}\{-k\beta\},\]
  underlies an equivariant line bundle, since $k\beta =\differential\chi_{\ell}$, where $\chi_{\ell} \colon G'=G \times \C^* \to \C^*$ is the character given by $(g,t) \mapsto t^{\ell}$. With \cref{prop:nonZeroIndependentOfPower}/\cref{rem:torsionGivesNonZero} and \cref{prop:nonTorsionGivesZero}, we see that
  \[\mathscr N_{L^*}^\beta \neq 0
  \ \Leftrightarrow \ \mathscr{D}_{L^*} \otimes_{\cA_{L^*}} (\omega_{L^*}\{\beta\})^\vee \neq 0
  \ \Leftrightarrow \ \mathscr{D}_{L^*} \otimes_{\cA_{L^*}} \omega_{L^*}^{\otimes (-k)}\{-\differential\chi_\ell\} \neq 0
  \ \Leftrightarrow \ \omega_{L^*}^{\otimes (-k)}\{-\differential\chi_\ell\} \cong \O_{L^*}.\]

  Now $\omega_{L^*} \cong \pi^{L^*,*} \omega_X$ as $G'$-equivariant line bundles by \cref{lem:canonicalBundleOnL} and \cref{lem:invariantSection}. Hence,
  \[\omega_{L^*}^{\otimes (-k)}\{-\differential\chi_\ell\} \cong \pi^{L^*,*} \big(\omega_X^{\otimes (-k)}\{-\differential\chi_\ell\}\big).\]
  By using this, and invoking \cref{lem:invariantSection} again, we see that
  \[\mathscr N_{L^*}^\beta \neq 0 \quad \Leftrightarrow \quad \omega_X^{\otimes (-k)}\{-\differential\chi_\ell\} \cong \Ell^{\otimes r} \text{ as $G'$-equivariant line bundles for some $r\in \Z$}.\]
  Since the $\C^*$-factor of $G'$ acts trivially on $X$, note that the natural $\C^*$-equivariant structure on $\omega_X^{\otimes (-k)}$ is also trivial. On the other hand, $\C^*$ acts by inverse scaling on the fibers of $L = \Tot(\Ell)$. Hence, if $\omega_X^{\otimes (-k)}\{-\differential\chi_\ell\} \cong \Ell^{\otimes r}$ holds for some $r \in \Z$, we must have $r = \ell$. Therefore:
  \begin{align*}
  \mathscr N_{L^*}^\beta \neq 0 \quad &\Leftrightarrow \quad \omega_X^{\otimes (-k)}\{-\differential\chi_\ell\} \cong \Ell^{\otimes \ell} \text{ as $G'$-equivariant line bundles}. \\
  &\Leftrightarrow \quad \omega_X^{\otimes (-k)} \cong \Ell^{\otimes \ell} \text{ as $G$-equivariant line bundles}. \qedhere
  \end{align*}
\end{proof}

We now conclude with the final result of this section classifying when Fourier-transformed tautological systems are non-zero away from the origin. We work in the setup stated at the beginning of this section, i.e.,
$X$ is projective and admits a transitive action by a reductive algebraic group $G$, $L$ is $G$-equivariant,
$G':=\dC^*\times G$, and $G'$ acts on $L$ by letting the $\dC^*$-factor act by inverse scaling in the fibres.
Moreover, we now assume that $\Ell$ on $X$ is very ample. Consider the $G'$-representation $\rho \colon G' \to \operatorname{GL}(V)$ with $V := H^0(X,\Ell)^\vee$ and the equivariant closed embedding $X \hookrightarrow \P V$ defined by $|\Ell|$. Let $\hat X \subseteq V$ be the affine cone of $X$ in $V$. Notice that we have an isomorphism $\hat{X} \setminus \{0\} \cong L^{\vee,*}$, given by identifying $L^{\vee}$ with the blow-up of $\hat X$ at the origin. We write $i$ for the closed embedding of $\hat{X}\setminus \{0\}$ into $V\setminus \{0\}$. Together with the isomorphism $\textup{inv} \colon L^* \to L^{\vee,*}$ given by inverting fibers, we obtain a closed embedding $i' \colon L^* \hookrightarrow V \setminus \{0\}$ defined by $i':=i \circ \textup{inv}$, as shown in the following diagram.
\begin{equation}\label{diag:Maps}
\begin{tikzcd}
L & L^*\ar[swap,hook']{l}{j_L} \ar[swap]{dd}{\textup{inv}}[swap]{\cong} \ar[hook]{rrd}{i'} & & &  \\
&&\hat{X} \setminus \{0\} \ar[hook]{r}{i} & V \setminus \{0\} \\
\textup{Bl}_{\{0\}}(\hat{X})\cong L^\vee&  L^{\vee,*}\ar[hook']{l}{j_{L^\vee}} \ar{ur}{\cong}
&&&
\end{tikzcd}
\end{equation}

\begin{thm} \label{thm:restrictedTauHatDescription}
  Let
  $\beta \colon \mathfrak g' \to \C$ be a Lie algebra homomorphism with $\restr{\beta}{\mathfrak g} \equiv 0$. Let $k \in \Z_{>0}$ be such that $\pi_1(L^{*,\mathrm{an}}) \cong \Z/k\Z$ (see \cref{sec:OBeta}). Then
  \[
  \restr{\hat{\tau}(\rho, \hat{X}, \beta)}{V \setminus \{0\}} \cong
    \begin{cases}
    i'_+ \O_{L^*}^{\ell/k} & \quad\text{if $\exists \ell \in \Z: \beta(\mathbf e) = \ell/k$ and $\Ell^{\otimes \ell} \cong \omega_X^{\otimes (-k)}$ } \\
    & \quad\text{as $G$-equivariant line bundles,} \\ \\
    0 & \quad\text{otherwise.}
  \end{cases}\]
\end{thm}

\begin{proof}
  This follows directly from the work above by combining \cref{cor:FLTautSysAsDirectImage} (applied to $Y = \hat X \setminus \{0\}$) and \cref{prop:NBetaEqualsOBeta}.

  For convenience, we roughly summarize the main steps that led to the proof of \cref{thm:restrictedTauHatDescription}:
  \begin{itemize}
    \item The restriction of $\hat{\tau}(\rho, \hat{X}, \beta)$ to $V \setminus \{0\}$ is supported on $\hat X \setminus \{0\} \cong L^*$ and can be described as $i_+ \mathscr N_{\hat X \setminus \{0\}}^\beta$ (\cref{cor:FLTautSysAsDirectImage}). Here, $\mathscr N_{\hat X \setminus \{0\}}^\beta$ arises from a cyclic right $\mathscr D$-module constructed from the vector fields induced by the group action (\cref{def:NBeta}).
    \item The $\mathscr D$-module $\mathscr N_{\hat X \setminus \{0\}}^\beta$ is alternatively described as $\mathscr{D}_{\hat X \setminus \{0\}} \otimes_{\cA_{\hat X \setminus \{0\}}} (\omega_{\hat X \setminus \{0\}}\{\beta\})^\vee$ (\cref{prop:interpretationOfNBeta}), where $\cA_{\hat X \setminus \{0\}} = \O_{\hat X \setminus \{0\}} \otimes \mathcal U(\mathfrak g')$ and $(\omega_{\hat X \setminus \{0\}}\{\beta\})^\vee$ is the anticanonical bundle with a $\mathfrak g'$-module structure determined by $\beta$.
    \item Identifying $\hat X \setminus \{0\}$ with $L^*$, we can argue that $\mathscr{D}_{\hat X \setminus \{0\}} \otimes_{\cA_{\hat X \setminus \{0\}}} (\omega_{\hat X \setminus \{0\}}\{\beta\})^\vee$ is non-zero if and only if $\Ell$ is a $\ell/k$-th rational power of $\omega_X^\vee$ and is equipped with a suitable equivariant structure (\cref{prop:NBetaEqualsOBeta}). The geometric reason is that in this case $(\omega_{\hat X \setminus \{0\}}\{\beta\})^{\otimes k} \cong \O_{\hat X \setminus \{0\}}$. This statement is based on vanishing results for $\mathscr D$-modules constructed from equivariant line bundles in this way (\cref{prop:nonZeroIndependentOfPower}, \cref{prop:torsionGivesNonZero} and \cref{prop:nonTorsionGivesZero}).
    \item In the cases that $\mathscr N_{\hat X \setminus 0}^\beta$ is non-zero, it is a smooth $\mathscr D$-module of rank~$1$. As such, it is isomorphic to $\O_{L^*}^{\ell/k}$ (see \cref{sec:OBeta}) and one confirms that $\beta(\mathbf e) = \ell/k$ (see \cref{prop:NBetaEqualsOBeta}). \qedhere
          \end{itemize}
\end{proof}

\begin{rmk}
  If $G$ is a semisimple linear algebraic group, then we have $[\mathfrak g, \mathfrak g] = \mathfrak g$. This shows that $\restr{\beta}{\mathfrak g} \equiv 0$ holds for every Lie algebra homomorphism $\beta \colon \mathfrak g' \to \C$, so this condition in \cref{thm:restrictedTauHatDescription} is always fulfilled in this case.
\end{rmk}

\begin{rmk}
  For $G$ not semisimple, let us drop the assumption $\beta \equiv 0$. If $\restr{\hat{\tau}(\rho,\hat X, \beta)}{V \setminus \{0\}} \neq 0$, then   it is necessarily of the form $i'_+\O_{L^*}^{\ell/k}$ for some $\ell \in \Z$ as before, since $\mathscr N_{L^*}^\beta$ is a smooth $\mathscr D_{L^*}$-module of rank one.
    One can also show that $\mathscr N_{L^*}^\beta \cong \O_{L^*}^{\ell/k}$ forces $k\beta = \differential \chi$ for some group character $\chi \colon G' \to \C^*$.
  With the same arguments as in the proofs above, we then get:
  \[
  \restr{\hat{\tau}(\rho, \hat{X}, \beta)}{V \setminus \{0\}} \cong
    \begin{cases}
    i'_+ \O_{L^*}^{\ell/k} & \quad\text{if $\exists \ell \in \Z: \beta(\mathbf e) = \ell/k$ and $\Ell^{\otimes \ell}\{\chi\} \cong \omega_X^{\otimes (-k)}$ } \\
    & \quad\text{as $G$-equivariant line bundles,} \\ \\
    0 & \quad\text{otherwise.}
  \end{cases}\]
  Hence, in the general situation we still only get a non-zero $\restr{\hat\tau(\rho, \hat X, \beta)}{V \setminus \{0\}}$ for $\Ell$ being a rational power of the anticanonical bundle and for exactly one suitable $\beta$ uniquely determined by $\Ell$.
\end{rmk}

\begin{rmk}\label{rmk:RemarkToReferee}
  As has been pointed out to us by one of the referees, many results of this section can also be viewed through the lens of equivariant $\mathscr D$-modules as in \cite{BeilinsonBernstein_Jantzen}, realizing that $\hat \tau(\rho, \hat X, \beta)$ is a $\beta'$-twisted strongly $G'$-equivariant $\mathscr D_V$-module supported on $\hat X$, \cite[Definition, page~16]{Hot98}. As above, outside $\{0\}$ we may understand it as a direct image of a $\mathscr D_{L^*}$-module, twisted equivariant under a transitive $G'$-action. The fiber $F \cong \C^*$ of $L^*$ over a point of $X$ is a slice for the $G'$-action, which implies that the category of (twisted) $G'$-equivariant $\mathscr D$-modules is equivalent to that of (twisted) $P'$-equivariant $\mathscr D$-modules, where $P'$ is the stabilizer of $F$, see e.g. \cite[Proposition~4.5]{Lorincz-Walther}. Here, the action of $P'$ on the fiber $\C^*$ captures the global geometry of the $G'$-equivariant line bundle $L$ in a character $P' \to \C^*$. Through an explicit calculation on the level of twisted equivariant $\mathscr D_{\C^*}$-modules, one sees that there only exist $\beta'$-twisted equivariant $\mathscr D_{L^*}$-modules for $\beta(\mathbf e)\in \frac{1}{k} \Z$, each a direct sum of a unique simple $\mathscr D_{L^*}$-module. The question whether or not the restriction of $\hat \tau(\rho, \overline X, \beta)$ to $V \setminus \{0\}$ then translates to whether this simple object admits a $G'$-invariant section, which in turn reduces to a calculation when a certain equivariant line bundle is trivial, similar in spirit to \cref{prop:nonTorsionGivesZero} and the calculation in the proof of \cref{prop:NBetaEqualsOBeta} above.
\end{rmk}

\section{Representation theoretic criterion}\label{sec:FormuleBeta}

The purpose of this section is to give a necessary criterion for the non-vanishing of tautological systems in terms of the representation of $G$ on the space $V=H^0(X,\cL)^\vee$, at least in the case where $G$ is semisimple. On the one hand, this is a slight sharpening of one direction of \cref{thm:restrictedTauHatDescription}, which only concerns vanishing of $\hat{\tau}(\rho,\hat{X},\beta)$ on $V\backslash\{0\}$. On the other hand, the methods used
here heavily rely on the representation theory of semisimple Lie algebras, and this section is therefore in large parts logically independent of the rest of the paper.

\subsection{Formula for \texorpdfstring{$\beta(\mathbf{e})$}{beta(e)}}
We work in the same setup as in \cref{thm:restrictedTauHatDescription}. Additionally, we assume throughout this section
\begin{center}
    $G$ is semisimple.
\end{center}
This implies automatically that any Lie algebra homomorphism $\beta \colon \mathfrak g' \to \C$ satisfies $\restr{\beta}{\mathfrak g} = 0$, since $\mathfrak g = [\mathfrak g, \mathfrak g]$. In particular, the choice of $\beta$ is equivalent to the choice of a complex number $\beta(\mathbf{e})$. As we will see, the tautological system $\tau(\rho, \hat{X},\beta)$ will only be non-zero for particular values of $\beta(\mathbf e)$ that we can express in terms of the highest weight of the (necessarily irreducible by Borel--Weil, see e.g.~\cite{Ser54} or \cite[Cor.~to~Th.~V]{Bot57} or also \cref{th:borel-weil} below) $G$-representation $\rho$.

To state this formula, let $\mathfrak{t}$ be a Cartan subalgebra of $\mathfrak{g}$, let $\Phi^+\subseteq \mathfrak{t}^\vee$ be a choice of positive roots, and set
\[ \delta\coloneqq \frac{1}{2}\sum_{\lambda\in \Phi^+} \lambda.\]

\begin{thm}\label{th:reptheory-beta}

  Let $\mu$ be the highest weight of the irreducible $G$-representation $V\coloneqq H^0(X, \mathscr{L})^\vee$. If $\tau(\rho,\hat{X},\beta)$ is nonzero, then
  \[ \beta(\mathbf{e}) \in \Big\{0, \, \frac{2\langle \delta, \mu\rangle}{|\mu|^2}\Big\},\]
  where $\langle-,-\rangle$ is the inner product on $\mathfrak{t}^\vee$ dual to (the restriction to $\mathfrak{t}$ of) the Killing form, and $|\mu|=\sqrt{\langle\mu,\mu\rangle}$.
\end{thm}

\begin{cor}\label{cor:BetaNonNeg}
If $\tau(\rho, \hat X, \beta) \neq 0$, then $\beta(\mathbf e)$ is a non-negative rational number.
\end{cor}
\begin{proof}
  Since $\mu$ is the highest weight, the inner product of $\mu$ with all positive roots is a non-negative integer, hence $\langle \delta, \mu \rangle \geq 0$ and so $\frac{2\langle \delta, \mu\rangle}{|\mu|^2}\in\dQ_{\geq 0}$.
\end{proof}

In the remainder of this section, we will prove \cref{th:reptheory-beta} as well as a more geometric interpretation of it.

\begin{notation}
Throughout this section, we use the following notation.
\begin{itemize}
    \item $\mathfrak{g}'=\C\mathbf{e}\oplus \mathfrak{g}$ -- the Lie algebra of $G'$
    \item $T$ -- a maximal torus of $G$
    \item $B$ -- a Borel subgroup of $G$ containing $T$
    \item $\mathfrak{t}$ -- the Lie algebra of $T$

    \item $\Phi(M,T)$ -- the roots of an (affine) algebraic group $M$ relative to a subtorus $T$. This is the set of characters $\lambda\colon T\to \C^*$ of $T$ such that
    \[ \mathfrak{m}_\lambda \coloneqq \{\xi\in \mathfrak{m} \mid \operatorname{Ad}(t)\xi = \lambda(t)\xi\}\neq 0,\]
    where $\mathfrak{m}$ is the Lie algebra of $M$. (Cf.~\cite{Hum75})

    \item $\Phi\coloneqq \Phi(G,T)$ -- the root system of $G$ relative to $T$. As usual, we view this as a subset of $\mathfrak{t}^\vee$.

    \item $\Phi^+\coloneqq \Phi(B,T)$ -- the choice of positive roots corresponding to $B$

    \item $\Delta\subseteq \Phi^+$ -- the simple roots

    \item $\delta \coloneqq \displaystyle\frac{1}{2}\sum_{\alpha\in \Phi^+}\alpha$ -- the Weyl vector.

    \item $\mathsf{B}(-,-)$ -- the Killing form on $\mathfrak{g}$
    \item $\langle-,-\rangle$ -- the symmetric bilinear form on $\mathfrak{t}^\vee$ induced by the restriction to $\mathfrak{t}$ of the Killing form. Since $\mathfrak{g}$ is semisimple, $\langle-,-\rangle$ is nondegenerate.
    \qedhere
\end{itemize}
\end{notation}

Since $\mathfrak{g}$ is semisimple, there is a decomposition
\[ \mathfrak{g} = \left(\bigoplus_{\alpha\in \Phi^+} \mathfrak{g}_{-\alpha}\right)\oplus \mathfrak{t}\oplus \left(\bigoplus_{\alpha\in \Phi^+} \mathfrak{g}_\alpha\right).\]
Each $\mathfrak{g}_\alpha$ is one-dimensional. In fact, one can choose a generator $E_\alpha\in \mathfrak{g}_\alpha$ for each $\alpha\in \Phi$ such that $[E_\alpha,E_{-\alpha}]\eqqcolon H_\alpha\in \mathfrak{t}$, $\mathsf{B}(E_\alpha, E_{-\alpha})=1$, $[H_\alpha, E_\alpha]=2E_\alpha$, and $[H_\alpha, E_{-\alpha}]=-2E_{-\alpha}$. Note that the $H_\alpha$ might not form a basis for $\mathfrak{t}$, as there may be too many of them.

A straightforward argument shows that for each $\alpha\in \Phi^+$, $H_\alpha$ is the unique element of $\mathfrak{t}$ for which $\mathsf{B}(H, H_\alpha) = \alpha(H)$ for all $H\in \mathfrak{t}$. We can use this property to define $H_\lambda$ for all $\lambda\in \mathfrak{t}^\vee$---then the nondegenerate bilinear form $\langle-,-\rangle$ is given by
\[ \langle \lambda, \lambda'\rangle = \mathsf{B}(H_\lambda, H_{\lambda'}) = \lambda(H_{\lambda'}) = \lambda'(H_{\lambda}).\]

\begin{dfn}
  The (second order) \emph{Casimir element} is the element
  \[ C\coloneqq \sum_i A_iB_i \in \mathcal{U}(\mathfrak{g}),\]
  where $\{A_i\}$ is any basis for $\mathfrak{g}$, and $\{B_i\}$ is the dual basis under the Killing form. In particular, if $\{H_i\}$ is an orthonormal basis of $\mathfrak{t}$ with respect to the Killing form, then
  \[ C = \sum_i H_i^2 + \sum_{\alpha\in\Phi^+} E_\alpha E_{-\alpha} + \sum_{\alpha\in\Phi^+} E_{-\alpha}E_{\alpha}.\]
\end{dfn}
A straightforward exercise
shows that $C$ is in the center of $\mathcal{U}(\mathfrak{g})$.

\begin{lem}\label{irred-rep-cas}
    Let $U$ be an irreducible representation of $\mathfrak{g}$ with
        lowest
    weight $\lambda$ and
        lowest
    weight vector $v_\lambda$. Then
    \[ C\cdot v_\lambda = \left(|\lambda|^2 - 2\langle\delta,\lambda\rangle\right)v_\lambda.\]
\end{lem}
\begin{proof}
    We have
    \begin{align*}
        C\cdot v_\lambda &= \left(\sum_i H_i^2 + \sum_{\alpha\in\Phi^+} E_\alpha E_{-\alpha} + \sum_{\alpha\in\Phi^+} E_{-\alpha}E_{\alpha}\right)\cdot v_\lambda\\
        &= \left(\sum_i H_i^2 - \sum_{\alpha\in\Phi^+} H_\alpha + 2\sum_{\alpha\in\Phi^+} E_{\alpha}E_{-\alpha}\right)\cdot v_\lambda\\
        &= \left(\sum_i H_i^2 -2H_\delta\right)\cdot v_\lambda,\\
    \end{align*}
    since $E_\alpha$ kills $v_\lambda$. Now use that $ \sum_i H_i^2\cdot v_\lambda = \sum_i \lambda(H_i)v_i = |\lambda|^2v_i$ and $H_\delta\cdot v_\lambda = \lambda(H_\delta)v_\lambda$.
\end{proof}

For the proof of \cref{th:reptheory-beta}, we need
a few facts about differential operators on affine varieties, which we introduce now. The full power of the theory of such operators is not needed here, so we only touch on a very small bit of it.

Set
\[R\coloneqq \Gamma(V, \mathcal{O}_V), \quad D_V \coloneqq \Gamma(V, \mathscr{D}_V),\quad\text{and}\quad S\coloneqq R/I,\]
where $I$ is the defining ideal of $\hat{X}$.

Define
\begin{equation}\label{eq:subAlg}
\begin{array}{rcl}
    A   &\coloneqq& \{P\in D_V \mid P(I) \subseteq I\},\\ \\
     J   &\coloneqq &\{P\in D_V \mid P(R) \subseteq I\} = I\cdot D_V=\sum_{\alpha\in \N^n} I\partial^\alpha,\\
\end{array}
\end{equation}
and
\begin{equation}\label{eq:MapPsi}
  \begin{array}{rcl}
         \psi:A&\longrightarrow&\End_\C(S)  \\
         P&\longmapsto&(\overline{f}\mapsto\overline{P\bullet f})     \end{array}
\end{equation}

We will need the following fact in the proof of \cref{th:reptheory-beta}.

\begin{lem}[{\cite[Proposition 1.6.]{SmithStafford}}]\label{milicic-ker}
We have $\ker\psi=J$ and $\im\psi=D_S$, where $D_S$ is the ring of Grothendieck differential operators of $S$ over $\C$.
\end{lem}

Under the assumption that $G$ is semisimple, the definition of $\hat\tau(\rho, \hat{X},\beta)$ simplifies to
\[
\hat\tau(\rho, \hat{X}, \beta) = \mathscr{D}_V/\bigl(\mathscr{D}_VI + \mathscr{D}_V(Z(\xi) \mid \xi \in \mathfrak{g}) + \mathscr{D}_V(Z(\mathbf{e})-\dim V + \beta(\mathbf{e}))\bigr).
\]
Here, we denote by $Z(\xi)$ the vector field $Z_V(\xi)$ defined in \cref{ssec:VFfromGroupActions} and we will also denote by $Z$ the map $\mathcal U(\mathfrak g') \to D_V$ extending it.

Because $X$ is $G$-invariant, the ideal $I$ is $\mathfrak{g}'$-stable, i.e.\ $Z(\xi)(I)\subseteq I$ for all $\xi\in \mathfrak{g}'$. Hence, the map $Z$ induces a $\mathfrak{g}'$-module structure on $S$ for which the elements of $\mathfrak{g}'$ act via derivations. If $S_d$ is the $d$th graded component of $S$, then
\[ \xi\cdot S_d \subseteq S_d \quad \text{for all }\xi \in \mathfrak{g},\]
and
\begin{equation}\label{eq:eul-scales}
    \mathbf{e}\cdot f = -df\quad\text{for all }f\in S_d.
\end{equation}
Denote the induced map $\mathcal{U}(\mathfrak{g}') \to \End_\C(S)$ by $Z_S$.

\begin{lem}\label{C-vs-e-on-S}
    $Z_S(C) = Z_S(\mathbf{e})^2|\mu|^2 - 2Z_S(\mathbf{e})\langle\delta,\mu\rangle$.
\end{lem}
\begin{proof}
    By definition, $R_1=V^\vee$. The construction of the embedding $X\hookrightarrow \P V$ implies that $R_1=S_1$ also. Hence, if $x\in S_1\cong V^\vee$ is a
        lowest weight vector, it has
        lowest weight $-\mu$, because $\mu$ is the \emph{highest} weight of $V$.
        A straightforward argument shows that for all $d\in \mathbb{N}$, the element $x^d$ is a
        lowest weight vector of $S_d$ with
        lowest weight $-d\mu$.
        Hence, by \cref{irred-rep-cas},
    \[ C\cdot x^d = (|-d\mu|^2 - 2\langle\delta, -d\mu\rangle)x^d.\]
        Since $C$ is in the center of the universal enveloping algebra, it acts on the irreducible $\fg$-representation $S_d$ as a scalar, which then must be the factor on the right hand side. Now use that $\mathbf{e}$ acts on $S_d$ as multiplication by $-d$ (eq.~\eqref{eq:eul-scales}).
\end{proof}

By definition, the operators $Z(C)$ and $Z(\mathbf{e})^2|\mu|^2 - 2Z(\mathbf{e})\langle\delta,\mu\rangle$ are contained in the subalgebra $A$ from \eqref{eq:subAlg}. By \Cref{C-vs-e-on-S}, their difference is in the kernel of the map $\psi$ from  \eqref{eq:MapPsi}. Hence, by \cref{milicic-ker}, we know that
\[ Z(C) - \bigl(Z(\mathbf{e})^2|\mu|^2 - 2Z(\mathbf{e})\langle\delta,\mu\rangle\bigr) \in I D_V.\]
Applying the standard $D$-module transpose $(-)^\top$ and identifying $Z(\mathbf{e})$ with minus the Euler differential operator\footnote{The minus sign comes from the fact that the action of $\mathfrak{g}$ on the coordinate ring of $V$ is the contragredient action.} $-E$ gives
\begin{equation}\label{eq:ZZtop}
    (Z(C))^\top - \bigl((E+\dim{V})^2|\mu|^2 - 2(E+\dim{V})\langle\delta,\mu\rangle\bigr)\in D_V I,
\end{equation}
since $(-E)^\top=E+\dim V$ and $I$ is homogeneous.

\begin{lem}\label{ZZtop}
    $(Z(C))^\top = Z(C)$.
\end{lem}
\begin{proof}
    Because $\fg$ is semisimple, $[\fg,\fg]=\fg$. Hence, $\fg$ acts on $V$ via trace-free matrices. Let $\xi\in \fg$ act on $V$ via a square matrix $A=[a_{ij}]$. Then $Z(\xi)$ is the derivation $-\sum_{i,j}a_{ji}x_i\partial_{x_j}$. The transpose of this operator is then $\sum a_{ji}\partial_{x_j} x_i = \sum a_{ji} x_i \partial_{x_j} +\Tr(A) = -Z(\xi)$. As $C$ arises as evaluation on elements of $\fg$ of a homogeneous quadric, $(Z(C))^\top=Z(C)$.
\end{proof}

\begin{proof}[Proof of \cref{th:reptheory-beta}]
Combining \cref{ZZtop} and eq.~\eqref{eq:ZZtop} yields
\begin{equation*}
    Z(C) - \bigl((E+\dim{V})^2|\mu|^2 - 2(E+\dim{V})\langle\delta,\mu\rangle\bigr)\in D_V I.
\end{equation*}
Taking cosets in $\hat\tau(\rho, \hat{X},\beta)$, we find
\begin{align*}
    0 = \overline{Z(C)}
    &= \overline{(E+\dim{V})^2|\mu|^2 -  2(E+\dim{V})\langle\delta,\mu\rangle}\\
        &= \overline{(E+\dim{V})\bigl((E+\dim{V})|\mu|^2 -  2\langle\delta,\mu\rangle\bigr)}.
\end{align*}
On the other hand, the defining ideal of $\hat\tau(\rho, \hat{X},\beta)$ also contains
\[ Z(\mathbf{e}) -\dim{V} + \beta(\mathbf{e}) = -E - \dim{V} + \beta(\mathbf{e})\]
So, we deduce
\[ \beta(\mathbf{e}) \in \left\{0, \frac{2\langle\delta,\;\mu\rangle}{|\mu|^2}\right\}.\]

\end{proof}

\subsection{Geometric interpretation}

We now aim for the following geometric description of the quantity for $\beta(\mathbf e)$ from the previous section, showing the compatibility of the non-vanishing result of $\restr{\hat{\tau}(\rho,\hat X, \beta)}{V \setminus 0}$ in \cref{thm:restrictedTauHatDescription} with that of
$\hat{\tau}(\rho,\hat{X},\beta)$ in \cref{th:reptheory-beta}.

\begin{prop} \label{prop:repTheoreticDescription}
  If $\Ell^{\otimes \ell} \cong \omega_X^{\otimes (-k)}$ as $G$-equivariant line bundles for some integers $k, \ell \neq 0$, then
  \[\frac{2\langle \delta, \mu \rangle}{\langle \mu, \mu \rangle} = \frac{\ell}{k}.\]
\end{prop}

The proof of this proposition will be based on the correspondence between characters of $P$ and equivariant line bundles on $G/P$ (where $P$ is parabolic), and identifying the character corresponding to the canonical bundle $\omega_{G/P}$ (\cref{can-as-char} below). For this, we first recall some facts about parabolic subgroups.

We first recall some facts about parabolic subgroups.

\begin{lem}\label{para-facts}\hphantom{NEWLINE}
    \begin{enumerate}[label=\textnormal{(\alph*)}]
        \item\label{para-homog} If $X$ is a projective homogeneous $G$-space, then $X\cong G/P$ for some parabolic subgroup $P$ of $G$ containing $B$.

        \item\label{para-bij} There is an inclusion-preserving bijection between subsets $I\subseteq \Delta$ and parabolic subgroups $P_I$ containing $B$.\footnote{Although we won't need it, the actual definition of $P_I$ can be found directly above Th.~29.2 in \cite{Hum75}.}
        \item\label{para-roots} $\Phi(P_I, T) = \Phi^+\cup (\Phi^-\cap \Z I)$.
    \end{enumerate}
\end{lem}
\begin{proof}
    \ref{para-homog} This is standard. It uses that every parabolic subgroup contains a Borel subgroup (\cite[Cor.~21.3.B.]{Hum75}), and that all Borel subgroups are conjugate (\cite[Th.~21.3]{Hum75}).

    \ref{para-bij} \cite[Th.~29.3]{Hum75}.

    \ref{para-roots} \cite[Th.~30.1]{Hum75}.
\end{proof}

Let $P$ be a parabolic subgroup of $G$ containing a maximal torus $T$. Recall that the characters $\lambda\colon T\to \C^*$ which are extendable to $P$ correspond one-to-one with $G$-equivariant line bundles $L_{\lambda,P}$ on $G/P$; see, e.g., \cite[\S9.11]{Hotta} (although the argument there is for $P=B$, the same argument works verbatim with $N^{-}$ replaced by the unipotent radical of the parabolic subgroup of $G$ opposite to $P$). Note that there are two common conventions for this correspondence---we choose the convention for which $P$ acts on the fiber of $L_{\lambda,P}$ at $P$ as $b\cdot v = \lambda(b)v$.\footnote{The other convention is $b\cdot v = \lambda(b)^{-1}v$.} In this case, the sheaf of sections $\mathscr{L}_{\lambda,P}$ of $L_{\lambda,P}$ is given by
\begin{equation}\label{eq:borel-L-lambda}
    \Gamma(U, \mathscr{L}_{\lambda,P}) = \{f\in \Gamma(q^{-1}(U), \mathscr{O}_G) \mid f(gb) = \lambda(b)^{-1}f(g)\text{ for all }g\in G, b\in P\},
\end{equation}
where $q\colon G\to G/P$ is the quotient map. Since $L_{\lambda,P}$ is $G$-equivariant, there is a $G$-equivariant structure on $\mathscr{L}_{\lambda,P}$. Although we won't need to know this structure explicitly, it may help the reader to note that the induced action of $G$ on $\Gamma(G/P, \mathscr{L}_{\lambda,P})$ is given by
\begin{equation*}
    (g\cdot f)(g') = f(g^{-1}g')\qquad (g,g'\in G,\; f\in \Gamma(G/P, \mathscr{L}_\lambda)).
\end{equation*}

\bigskip
There are many proofs of the following theorem throughout the literature. It is often stated and proved only for $P=B$. However, it was originally proven for all parabolic subgroups, e.g.~in \cite{Ser54} or \cite[Cor.~to Th.~V]{Bot57}.
\begin{thm}[Borel--Weil]\label{th:borel-weil}
        If $-\lambda$ is a dominant weight which is extendable to the parabolic subgroup $P$, then $\Gamma(G/P,\mathscr{L}_{\lambda, P})^\vee$ is the irreducible representation of $G$ with highest weight $-\lambda$.
\end{thm}

\begin{lem}\label{can-as-char}
    Let $I\subseteq \Delta$ be a subset of the set of simple roots, and let $P_I$ be the corresponding parabolic subgroup. Define
    \begin{equation*}
        \delta_I \coloneqq \frac{1}{2}\sum_{\alpha\in \Phi^+\setminus \Z I} \alpha.
    \end{equation*}
    Then
    \begin{equation*}
        \omega_{G/P_I} \cong \mathscr{L}_{2\delta_I, P_I}.
    \end{equation*}
\end{lem}

\begin{proof}
    The following argument is based on the argument given in the MathOverflow post \cite{403574}.

    According to \cite[Lem.~1.4.9]{CG10}, $T^*(G/P_I)$ is the (unique) $G$-equivariant vector bundle on $G/P_I$ whose fiber over $P_I$ is the $P_I$-module
    \[ \mathfrak{p}_I^\perp\coloneqq \{\xi \in \mathfrak{g} \mid \langle \xi, x\rangle = 0 \text{ for all }x\in \mathfrak{p}_I\},\]
    where $P_I$ acts via the coadjoint action.
    But, letting $T$ be a maximal torus of $G$ contained in $P_I$, we have a sequence of $P_I$-isomorphisms
    \begin{align*}
        \mathfrak{p}_I^\perp &\cong (\mathfrak{g}/\mathfrak{p}_I)^\vee\\
        &\cong \left(\bigoplus_{\substack{\alpha\in \Phi\text{ not a root}\\\text{of }P_I\text{ rel.~to }T}}\mathfrak{g}_\alpha\right)^\vee\\
        &\cong \left(\bigoplus_{\alpha\in -(\Phi^+\setminus \Z I)} \mathfrak{g}_\alpha\right)\\
        &\cong \bigoplus_{\alpha\in \Phi^+\setminus \Z I} \mathfrak{g}_\alpha.
    \end{align*}
    Taking the determinant gives the $P_I$-equivariant line bundle whose fiber at $P_I$ is the $P_I$-module
    \begin{equation*}
        \bigotimes_{\alpha\in \Phi^+\setminus \Z I} \mathfrak{g}_\alpha.
    \end{equation*}
    The action of $P_I$ on this module is determined by the action of $T$, and the action of $t\in T$ is just multiplication by
    \[ \sum_{\alpha \in \Phi^+\setminus \Z I} \alpha(t) = 2\delta_I(t).\]
    Thus, $\omega_{G/P_I} \cong \mathscr{L}_{2\delta_I, P_I}$.
\end{proof}

We need one more technical lemma before beginning the proof of \cref{prop:repTheoreticDescription}.

\begin{lem}
    Let $I\subseteq \Delta$ be a subset of the simple roots. Then $\langle\delta,\delta_I\rangle = \langle\delta_I, \delta_I\rangle$.
\end{lem}
\begin{proof}
    Set $\delta_I' \coloneqq \delta - \delta_I = \frac{1}{2}\sum_{\alpha\in \Phi^+\cap \Z I} \alpha$. We want to show that $\langle\delta'_I, \delta_I\rangle=0$. To begin with, let $\Phi_I$ be $\Phi\cap \R I$ viewed as a subset of the vector space $\R I$. It is immediately clear that $\Phi_I$ is a root system (in $\R I$), and that $I$ forms a base of $\Phi_I$. Hence, $\delta'_I$ is the Weyl vector $\delta_{\Phi_I}$ of $\Phi_I$ (with respect to this base). Moreover, the inner product of two elements of $\R I$ is the same as in the ambient vector space $\mathfrak{t}^\vee$ of the root system $\Phi$. So, the coroots of $\Phi_I$ are the coroots $H_\alpha\coloneqq 2\alpha/\langle\alpha,\alpha\rangle$ of $\Phi$ for $\alpha\in \Phi_I$. Therefore, by \cite[Prop.~8.38]{Hal15},
    \[ \langle\delta'_I, H_\alpha\rangle = \langle\delta_{\Phi_I}, H_\alpha\rangle = 1\]
    for all $\alpha\in I$. Hence, for all $\alpha\in I$, we have
    \begin{align*}
        \frac{2\langle \alpha, \delta_I\rangle}{\langle\alpha,\alpha\rangle}
        &= \langle H_\alpha, \delta_I\rangle\\
        &= \langle H_\alpha, \delta - \delta'_I\rangle\\
        &= \langle H_\alpha, \delta\rangle - \langle H_\alpha, \delta'_I\rangle\\
        &= 1-1 = 0,
    \end{align*}
    where the final equality again uses \cite[Prop.~8.38]{Hal15}. Therefore, $\langle\alpha, \delta_I\rangle=0$ for all $\alpha\in I$ and hence for all $\alpha\in \R I$. In particular, $\langle \delta'_I, \delta_I\rangle=0$.
\end{proof}

\begin{proof}[Proof of \Cref{prop:repTheoreticDescription}]
Since $X$ is a projective $G$-homogeneous space, it is isomorphic by \Cref{para-facts} to $G/P_I$ for some $I$. By assumption, $\mathscr{L}^{\otimes \ell}\cong \omega_X^{\otimes (-k)}$ as $G$-equivariant line bundles. Therefore, since $\Pic(X)$ and therefore $\Pic^G(X)$ is torsion-free, and applying \cref{can-as-char},
\begin{equation*}
    \mathscr{L} \cong \mathscr{L}_{-\frac{k}{\ell}2\delta_I, P}.
\end{equation*}
Therefore, by Borel--Weil (\cref{th:borel-weil}), the $G$-representation $V=\Gamma(X, \mathscr{L})^\vee$ has highest weight
\[\mu=-(-\frac{2k}{\ell}\delta_I)=\frac{2k}{\ell}\delta_I.\] Then
\begin{align*}
    \frac{2\langle\delta, \mu\rangle}{\langle\mu,\mu\rangle}
    &= \frac{\ell\langle \delta, \delta_I\rangle}{k\langle\delta_I, \delta_I\rangle}
    = \frac{\ell\langle \delta_I, \delta_I\rangle}{k\langle\delta_I, \delta_I\rangle}
    = \frac{\ell}{k}
    = \beta(\mathbf{e}). \qedhere
\end{align*}
\end{proof}

We finish this section with the following well-known fact concering the anticanonical bundle of $X=G/P$, the proof of which we include for the convenience of the reader.
\begin{lem}\label{lem:Fano}
Assume only that $G$ is reductive. Then $X=G/P$ is a Fano variety.
\end{lem}
\begin{proof}
By \cite[II.4.4]{Jantzen}, a $G$-equivariant line bundle $\Ell_{-\lambda,P}$ on $G/P$ is ample if and only if
\[\langle \lambda, \alpha \rangle > 0 \qquad \text{for all $\alpha \in \Delta \setminus I$}.\]
By \cref{can-as-char}, the anticanonical bundle $\omega_X^{\vee}$ is $\Ell_{-2\delta_I,P}$, so we need to check that $\langle \delta_I, \alpha \rangle > 0$ holds for any $\alpha \in \Delta \setminus I$. The reflection $s_\alpha \colon \Phi \to \Phi$ given by
\[s_\alpha(\beta) = \beta-\frac{2\langle \alpha, \beta \rangle}{\langle \alpha, \alpha \rangle} \alpha\]
maps $\alpha$ to $-\alpha$ and permutes $\Phi^+ \setminus (\Z I \cup \{\alpha\})$ (this follows easily from the defining property of the set of simple roots $\Delta$). Hence, $s_\alpha(\delta_I) = \delta_I - \alpha$, which means that $\frac{2\langle \alpha, \delta_I\rangle}{\langle \alpha, \alpha \rangle} = 1 > 0$.
\end{proof}

\section{Tautological systems associated to homogeneous spaces} \label{sec:tautologicalSystems}

The purpose of this section is to gather all the previous results and to apply them to the special case where we are given a projective variety $X$ with a transitive group action together with a very ample equivariant line bundle.
We obtain a representation on the space of sections, and we can therefore consider the corresponding tautological system. The non-vanishing results from \cref{sec:nonVanishingCriteria} then apply. Moreover, we show in \cref{subsec:LocProp} a localization property of the corresponding Fourier transform $\hat{\tau}$, and in \cref{subsec:CoLocProp} a property that is in a certain sense dual to the first one, which is why we call it ``colocalization property''. These results, combined with the discussion in \cref{sec:FLOnVecBundles}, will finally give our main result
(\cref{theo:tau_is_MHM}) stating that the tautological system $\tau(\rho,\hat{X},\beta)$, if non-zero, underlies a complex pure resp.\ a mixed Hodge module, this is discussed in \cref{subsec:TautMHM}.

\subsection{Localization property of \texorpdfstring{$\widehat{\tau}$}{\textbackslash hat\textbackslash tau}}
\label{subsec:LocProp}

The purpose of this and of the following section is to prove a key property of the differential system $\hat{\tau}(\rho, \hat{X},\beta)$ concerning its relation to its restriction $i^+\hat{\tau}(\rho, \hat{X},\beta)$, where $i \colon \{0\}\hookrightarrow V$. In this section we only consider the case where $\beta(\mathbf{e})\in\dC\setminus \dZ$, whereas in the next section also the case where $\beta(\mathbf{e})\in \dZ$ is studied. For the moment, we are working in a slightly more general setup, therefore, we let temporarily  $V$ be any finite-dimensional vector space, and we consider the Euler operator $E$ on $V$ (i.e.\ the differential of the scaling action). For $\lambda\in \C$, define $\mathbf{Eig}(V,\lambda)$ to be the full subcategory of $\operatorname{Mod}(\Gamma(V, \mathscr{D}_V))$ consisting of modules $M$ satisfying
\begin{equation}
    M = \bigoplus_{\mu\in \lambda + \Z} M_\mu,
\end{equation}
where $M_\mu \coloneqq \ker(E-\mu)\subseteq M$.

\begin{prop}\label{prop:Eig-loc}
    Let $\mathscr{M}\in \operatorname{Mod}_{qc} (\mathscr{D}_V)$. If $\Gamma(V, \mathscr{M})\in \mathbf{Eig}(V,\lambda)$ for some $\lambda\in \C\setminus \Z$, then
    \[
    \mathscr{M}\cong j_+ j^+\mathscr{M},
    \]
    where $j$ denotes the open embedding of $V \setminus \{0\}$ into $V$.
\end{prop}
\begin{proof}
        Let $N := \dim V$ and choose coordinates $x_1,\dots,x_N$ on $V$. The distinguished triangle
    \[\mathrm{R}\Gamma_{\{0\}}(\mathscr{M}) \to \mathscr{M} \to j_+j^+\mathscr{M} \xrightarrow{+1}\]
    in $D_{qc}^b(\mathscr D_V)$ (see \cite[Prop.~1.7.1(i)]{Hotta}) shows that it suffices to prove $\mathrm{R}\Gamma_{\{0\}}(\mathscr{M}) = 0$ in order to conclude the claim. Since $V$ is affine and $\mathscr{M}$ is quasi-coherent, we actually just need to show
    \[ H^i_{\mathfrak{m}}(M)=0\quad\text{for all }i,\]
    where $M=\Gamma(V, \mathscr{M})$ and $\mathfrak{m}=(x_1,\ldots,x_N)$.

    Recall that $H^i_\mathfrak{m}(M)$ may be computed as the cohomology of the \v{C}ech complex
    \[ 0 \to M \to \bigoplus_{i} M_{x_i} \to \bigoplus_{i,j} M_{x_ix_j} \to \cdots \to M_{x_1\cdots x_N}\to 0.\]
    A straightforward application of the definition of eigenvector implies (a) that each term in this complex is also in $\mathbf{Eig}(V,\lambda)$, and (b) that $\mathbf{Eig}(V,\lambda)$ is closed under taking subquotients. Hence, each $H^i_\mathfrak{m}(M)$ is in $\mathbf{Eig}(V,\lambda)$.

    Since $H^i_\mathfrak{m}(M)$ is $\mathfrak{m}$-torsion, it remains to show that every $\mathfrak{m}$-torsion module in $\mathbf{Eig}(V,\lambda)$ is zero. Let $M'$ be one such module, and assume there is a nonzero $n\in M'$. Without loss of generality, we may assume that $n\in M'_\mu$ for some $\mu\in \lambda +\Z$ and that $\mathfrak{m}n=0$. Then
                            $$
    \mu n = E\cdot n
            = \sum_{i=0}^N x_i\partial_i \cdot n
            = \sum_{i=0}^N (\partial_i x_i -1) \cdot n = -N n.
    $$
    Thus, because $n\neq 0$, $\mu$ must be $-N$---in particular, $\mu$, and therefore also $\lambda$, must be an integer, which is false by assumption. Hence, $M'=0$.
\end{proof}
We draw a conclusion of the previous general result that concerns the Fourier transform of tautological systems as studied in \cref{cor:FLTautSysAsDirectImage}, where we only make the assumption that the boundary of the $G'$-orbit is reduced to the origin in the vector space $V$. This is of course satisfied in the case of interest like in the situation studied in \cref{thm:restrictedTauHatDescription}.
\begin{cor} \label{cor:localization}
  Let $\rho \colon G' \to \GL(V)$ be a finite-dimensional rational representation of an algebraic group of the form $G' = G \times \C^*$, where $\C^*$ acts by scaling elements of $V$. Let $Y \subseteq V$ be a $G'$-orbit and let $\overline{Y}$ be its closure. Assume that $\overline{Y}\setminus Y = \{0\}$.
  Let $\beta \colon \mathfrak g' \to \C$ be a Lie algebra homomorphism with $\beta(\mathbf e) \in \C \setminus \Z$. Then the $\mathscr{D}_V$-module $\hat\tau(\rho,\overline{Y},\beta)$ from \cref{def:tautSys} satisfies
  \[
  \hat\tau(\rho,\overline{Y},\beta) \cong j_+ j^+\hat \tau(\rho,\overline{Y},\beta),
  \] where $j$ denotes the open embedding of $V \setminus \{0\}$ into $V$.
\end{cor}

\begin{proof}
        By \cref{prop:Eig-loc}, it suffices to prove that $\Gamma(V, \hat\tau(\rho, \overline Y, \beta))\in \mathbf{Eig}(V,\beta(\mathbf{e}))$. To do this,     let
    \[ P= \sum_{\alpha\gamma} c_{\alpha\gamma} x^\alpha\partial^\gamma\]
    be a global section of $\mathscr{D}_V$. In $\Gamma(V, \hat\tau(\rho, \overline Y, \beta))$, we have
    \begin{align*}
        EP &= \sum_{\alpha\gamma} c_{\alpha\gamma} (|\alpha| - |\gamma|)x^\alpha\partial^\gamma + \sum_{\alpha\gamma} c_{\alpha\gamma} x^\alpha\partial^\gamma E\\
        &= \sum_{\alpha\gamma} c_{\alpha\gamma} (|\alpha| - |\gamma|)x^\alpha\partial^\gamma + \sum_{\alpha\gamma} \beta(\mathbf e)c_{\alpha\gamma} x^\alpha\partial^\gamma\\
        &= \sum_{\alpha\gamma} (\beta(\mathbf e) + |\alpha| - |\gamma|)c_{\alpha\gamma}x^\alpha\partial^\gamma\\
        &\in \bigoplus_{\mu \in \beta(\mathbf e) + \Z} \hat\tau_\mu.
    \end{align*}
    Thus, $\Gamma(V, \hat\tau(\rho, \overline Y, \beta))\in \mathbf{Eig}(V,\beta(\mathbf e))$.
\end{proof}

\subsection{Colocalization property of \texorpdfstring{$\widehat{\tau}$}{\textbackslash hat\textbackslash tau}}
\label{subsec:CoLocProp}

In this section, we consider a similar property as just studied, but which also includes the case where $\beta(\mathbf{e})\in \dZ$.
It turns out (see \cref{ex:Coloc} below) that in general the $\cD_V$-module $\hat{\tau}(\rho, \hat{X}, \beta)$ is not equal to the direct image of its restriction to $V\setminus \{0\}$, but to one cohomology group of the properly supported direct image. In the case where the value of $\beta$ on $\mathbf{e}$ is not an integer, this is consistent with the previous result since both direct images are equal then.

We work in the setup described before \cref{thm:restrictedTauHatDescription}, i.e.\ $X\subseteq \dP V$ is projective with affine cone $\hat{X}$ with vanishing ideal $\cI\subseteq \cO_V$. Consider the embeddings
$$
\begin{tikzcd}
\hat{X} \setminus \{0\} \ar[hook]{r} & V\setminus\{0\} \ar[hook]{r}{j} & V & \{0\} \ar[hook', swap]{l}{i_{\{0\}}}.
\end{tikzcd}
$$

Our main result in this section is the following.
\begin{thm}\label{theo:coloc}
If $\beta(\mathbf{e})\notin \dZ_{\leq 0}$, then
$\hat{\tau}(\rho, \hat{X}, \beta)$ is colocalized, in the sense that the canonical morphism
$$
H^0j_\dag j^+ \hat{\tau}(\rho, \hat{X}, \beta) \longrightarrow \hat{\tau}(\rho, \hat{X}, \beta)
$$
is an isomorphism in $\textup{Mod}_h(\cD_V)$.
\end{thm}

Before we discuss the proof of this theorem we show by example that integral parameters may correspond to systems that are  colocalized but not localized.  From here on and until the end of this  paragraph, in order to keep the notation light, we write $\hat\tau$ for the  $\cD_V$-module
$\hat{\tau}(\rho, \hat{X}, \beta)$ that appears in the theorem above.
\begin{exa}\label{ex:Coloc}
    Let $X$ be $\dP^1\times \dP^1$, where the group $G:=\SL_2\times\SL_2$ acts transitively via the action on each factor. Choose the projective embedding induced by the line bundle $\cO(1,1)$. The target space is $\dP V=\dP^3$ and $X$ is cut out by $f=x_{1,1}x_{2,2}-x_{2,1}x_{1,2}$. We write
    $E=x_{1,1}\partial_{1,1}+x_{1,2}\partial_{1,2}+
    x_{2,1}\partial_{2,1}+x_{2,2}\partial_{2,2}$.
    The interesting $\beta(\mathbf{e})$ (for which, according to \cref{thm:restrictedTauHatDescription}, the restriction of $\hat\tau$ to $V\setminus\{0\}$ is non-zero and has full support) equals $2$
    (so that $\beta'(\mathbf{e})=\trace(E)-\beta(\mathbf{e})=2$), and then the defining ideal of $\hat \tau$ is generated by $f$, $-E-\beta'(\mathbf{e})=-(E+2)$, and the operators
    \begin{gather*}
x_{2,1}\partial_{1,1}+x_{2,2}\partial_{1,2},\quad x_{1,1}\partial_{2,1}+x_{1,2}\partial_{2,2},\quad x_{1,1}\partial_{1,2}+x_{2,1}\partial_{2,2},\quad x_{1,2}\partial_{1,1}+x_{2,2}\partial_{2,1},\\
\theta_{1,1}+\theta_{1,2}+1,\quad
\theta_{2,1}+\theta_{2,2}+1,\quad
\theta_{1,1}+\theta_{2,1}+1,\quad
\theta_{1,2}+\theta_{2,2}+1,
\end{gather*}
    where we write $\theta_{i,j}$ for $x_{i,j}\partial_{i,j}$. Notice that the sum of $\theta_{1,1}+\theta_{1,2}+1$ and $\theta_{2,1}+\theta_{2,2}+1$ (or equally the sum of $\theta_{1,1}+\theta_{2,1}+1$ and $\theta_{1,2}+\theta_{2,2}+1$) equals $E+2$, but we prefer to work with this slightly redundant set of generators.

    Let $P=\partial_{1,1}\partial_{2,2}-\partial_{2,1}\partial_{1,2}$. It is an easy calculation using the above generators to see that the class of $x_{i,j}P$ is zero in $\hat\tau$, for $i,j\in\{1,2\}$. A computer computation shows that $P$ is not zero in $\hat\tau$, and so $\hat\tau$ contains a submodule $\cK$ of holonomic length one that is supported at the origin. In particular, we certainly have $\hat{\tau}\neq j_+j^+\hat{\tau}$ in this case.

    Inspection shows that there is a natural $\cD_V$-module map from $\hat\tau$ to the local cohomology sheaf $\cH=H^1_{\hat X}(\cO_V)$ that sends the coset of $1$ to the coset of $1/f$. Notice that this map is not surjective, since $\cH \cong \cO_V(*\hat{X})/\cO_V$ is generated by $1/f^2$, due to the fact that the Bernstein-Sato polynomial of $f$ is $(s+1)(s+2)$.

    The image of
    $\hat{\tau}\rightarrow \cH$ is the Kashiwara--Brylinski module $\cB$ attached to
        $f$ (i.e. the module obtained as
        $\im(H^0\hat{\iota}_{\dag}\cO_{\hat{X}\setminus\{0\}}
        \rightarrow
        H^0\hat{\iota}_{+}\cO_{\hat{X}\setminus\{0\}})\in\textup{Mod}(\cD_V)$), recall that
        $\hat{\iota} \colon  L^{\vee,*}\cong\hat{X}\setminus\{0\}\hookrightarrow V$ is the composition of the closed embedding $i \colon \hat{X}\setminus\{0\}\hookrightarrow V\setminus\{0\}$ with the canonical open embedding $j \colon V\setminus\{0\}\hookrightarrow V$ from above), and so $\cB$ is in particular simple and self-dual. The cokernel $\cC=\cH/\cB$ is the $\cD_V$-module generated by $1/f^2$; it is supported at the origin and of holonomic length one.  The   kernel is the module $\cK$ above.
        We thus arrive at the following sequence of $\cD_V$-modules.
        \[
            0\longrightarrow \cK \longrightarrow \hat\tau \longrightarrow \cH \longrightarrow \cC  \longrightarrow 0.
        \]
        It is automatic that $\bD \cK\cong \cC$ since both are length one and supported at the origin,  but one can also verify  that $\bD \cH \cong \hat\tau$. Moreover, it follows from the fact that $\cO_V(*\hat{X})$ is localized along $\hat{X}$ that it is also localized at $\{0\}$, i.e. that we have $j_+j^+\cO_V(*\hat{X})
        =\cO_V(*\hat{X})$. Then since $j_+j^+\cO_V=\cO_V$ we get that $j_+j^+\cH\cong\cH$, and thus the module $\bD \hat\tau$ also satisfies
        $$
            j_+j^+ \bD\hat\tau \cong \bD \hat\tau.
        $$
\end{exa}

The proof of \cref{theo:coloc} will be given after several intermediate steps. First we recall that we have the algebra
$\cA_V$ (see \cref{def:A}), which is the universal enveloping algebra of the Lie algebroid $\cO_V\otimes_\dC \fg'$. It comes with a (in general non-surjective) algebra homomorphism $\widetilde{Z}_V\colon\cA_V\rightarrow \cD_V$ which extends the map $Z_V$ as defined
in \cref{lem:defEquivVF}. Then we consider the left $\cA_V$-module
\[
\hat{\tau}^\cA:=\cA_V/\cA_V \cI+\cA_V\left(\xi-\beta'(\xi)\,|\,\xi\in\fg'\right).
\]
From the right exactness of the tensor product we get
\[
\hat{\tau} = \cD_V\otimes_{\cA_V} \hat{\tau}^\cA =
H^0\left(\cD_V\otimes^\dL_{\cA_V} \hat{\tau}^\cA\right),
\]
using that $\widetilde{Z}_V$ makes $\cD_V$ into a right $\cA_V$-module.
We first have the following comparison result.
\begin{lem}\label{lem:Spect-1}
If $H^k(\omega_V\otimes^\dL_{\cA_V}\hat{\tau}^\cA)=0$ for $k=0,-1$, then also $H^k(\DR \hat\tau)=H^k(\omega_V\otimes^\dL_{\cD_V} \hat{\tau})=0$ for $k=0,-1.$
\end{lem}
\begin{proof}
Consider the Grothendieck spectral sequence for the composition of functors
$\omega_V \otimes_{\cD_V} - $ and $\cD_V\otimes_{\cA_V} -$, with $E_2$-term
\[
    E_2^{p,q} = H^p\left(\omega_V \otimes_{\cD_V}^\dL H^q \left(\cD_V\otimes^\dL_{\cA_V} \hat{\tau}^\cA\right)\right) \Longrightarrow
    H^{p+q}\left(\omega_V\otimes^\dL_{\cA_V} \hat{\tau}^\cA\right).
\]
We clearly have that $E_2^{0,0}=H^0(\omega_V\otimes^\dL_{\cA_V} \hat{\tau}^\cA)$ and moreover, because we are dealing with the second page of a third quadrant spectral sequence, $E_2^{-1,0}$ injects into
$H^{-1}(\omega_V\otimes^\dL_{\cA_V} \hat{\tau}^\cA)$. Hence, under the assumption of the lemma, we obtain
\[
E_2^{0,0} = H^0(\omega_V \otimes^\dL_{\cD_V} \hat{\tau}) =0
\quad\quad
\textup{and}
\quad\quad
E_2^{-1,0} = H^{-1}(\omega_V \otimes^\dL_{\cD_V} \hat{\tau}) =0.\qedhere
\]
\end{proof}

Next consider the following adjuction triangle
\begin{equation}\label{eq:AdjTriang}
j_\dag j^+ \hat\tau \longrightarrow \hat\tau \longrightarrow
(i_{\{0\},+} i_{\{0\}}^\dag)\hat\tau[\dim V] \stackrel{+1}{\longrightarrow}
\end{equation}
and the associated exact sequence
\[
0
\longrightarrow H^{-1}((i_{\{0\},+} i_{\{0\}}^\dag)\hat\tau[\dim V])
\longrightarrow H^0(j_\dag j^+ \hat\tau)
\longrightarrow \hat\tau
\longrightarrow H^0((i_{\{0\},+} i_{\{0\}}^\dag)\hat\tau[\dim V])
\]
We would like to show that the left- and the rightmost terms in this sequence vanish.
Since clearly $i_{\{0\},+}$ is an exact functor, it suffice to show that
$$
H^k(i_{\{0\}}^\dag\hat{\tau}[\dim V])=0
$$
for $k=0,-1$. To that end, we apply the functor $a_{V,+}$ (where
$a_V\colon V\rightarrow \{pt\}$ is the projection) to the triangle
\eqref{eq:AdjTriang}, this yields
\begin{equation}\label{eq:AdjTriang-2}
a_{V,+}j_\dag j^+ \hat\tau \longrightarrow a_{V,+}\hat\tau \longrightarrow
i_{\{0\}}^\dag \hat\tau[\dim V] \stackrel{+1}{\longrightarrow}
\end{equation}
since
\[
a_{V,+} i_{\{0\},+} i_{\{0\}}^\dag \hat\tau[\dim V] \cong
(a_V\circ i_{\{0\}})_+ i_{\{0\}}^\dag \hat\tau[\dim V] \cong
a_{\{0\},+} i_{\{0\}}^\dag \hat\tau[\dim V] \cong
i_{\{0\}}^\dag \hat\tau[\dim V]
\]
as elements in $D^b(\dC)$.

Now we have the following piece of the associated cohomology sequence of the triangle \eqref{eq:AdjTriang-2}
\begin{equation}\label{eq:CohomSequenceAv}
H^{-1} a_{V,+} \hat\tau
\longrightarrow H^{-1} i_{\{0\}}^\dag \hat\tau[\dim V]
\longrightarrow H^0 a_{V,+} j_\dag j^+ \hat\tau
\longrightarrow H^0 a_{V,+} \hat\tau
\longrightarrow H^0 i_{\{0\}}^\dag\hat\tau[\dim V]
\longrightarrow 0.
\end{equation}
Here zero on the rightmost term comes from the vanishing
\[
H^1 a_{V,+} j_\dag j^+ \hat\tau=0,
\]
which holds since both functors $a_{V,+}$ and $j_\dag $ are right exact.
We now claim
\begin{lem}
The map
\[
H^{-1} a_{V,+} \hat\tau
\longrightarrow H^{-1} i_{\{0\}}^\dag \hat\tau[\dim V]
\]
is an isomorphism.
\end{lem}
\begin{proof}
It can be shown more generally that under the assumption made here, we have an isomorphism
\[
a_{V,+} \hat\tau
\longrightarrow i_{\{0\}}^\dag \hat\tau[\dim V]
\]
in $D^b(\dC)$. In order to see this, we apply \cite[Lemma 4.4]{AviDualProjRestrGKZ} (which is based on an earlier result in \cite[Lemma 3.3]{ReichWalth-Duco}),
when seeing $a_V \colon V\rightarrow\{pt\}$ as a bundle over the point $\{pt\}$. Then it is clear that this map is fibered in the sense of \cite[Definition 4.1]{AviDualProjRestrGKZ}. It therefore remains to check that the $\cD_V$-module $\hat\tau$ is twistedly $\dC^*$-quasi-equivariant (as defined in \cite[Definition 4.2]{AviDualProjRestrGKZ}). This is a condition that depends only on the restriction
$j^+\hat\tau$, and this restriction has support on $\hat{X}\setminus \{0\}$.
Recall that we have  the isomorphism $L^* \cong \hat{X}\setminus\{0\}$, obtained from composing the restriction to $L^{\vee,*}$ of the blow-up $L^\vee \rightarrow \hat{X}$ with the fiberwise isomorphism $\textup{inv} \colon L^*\stackrel{\cong}{\rightarrow} L^{\vee,*}$.
It is therefore sufficient
to show that $\iota^+\hat{\tau}$ is
twistedly $\dC^*$-quasi-equivariant with respect to the
$\dC^*$-action in the fibres of $L\rightarrow X$, where $\iota$ is the composition of $j \colon V\setminus \{0\}\hookrightarrow V$ with the closed embedding $\hat{X}\setminus\{0\}\hookrightarrow V\setminus\{0\}$ and with the above isomorphism $L^{*}\cong \hat{X}\setminus\{0\}$.

It follows from \cref{thm:restrictedTauHatDescription} that this restriction is either zero, in which case the equivariance property we are after is trivially satisfied, or else equals $\cO_{L^*}^{\ell/k}$. It is an easy exercise to check (e.g., locally over trivializing neighborhoods) that $\cO_{L^*}^{\ell/k}$ is
twistedly $\dC^*$-quasi-equivariant.
\end{proof}

By using the exact sequence \eqref{eq:CohomSequenceAv} as well as the previous lemma, \cref{theo:coloc} is proved once we have shown that $H^k(a_{V,+}\hat\tau)=0$ for $k=0,-1$. But clearly $a_{V,+}\hat\tau = a_{V,*}\DR(\hat\tau)$ since $a_V$ is an affine morphism. Therefore, by \cref{lem:Spect-1}, we are left to show the following.
\begin{prop}\label{prop:VanishCohomA}
Using the above notation, and always assuming that
$\beta(\mathbf{e})\notin \dZ_{\leq 0}$, we have
\[
H^k(\omega_V\otimes^\dL_{\cA_V}\hat{\tau}^\cA)=0
\]
for $k=0,-1$.
\end{prop}

For this, we will need some further preparations. We start with an algebraic property of the left $\cA_V$-module
$\cA_V/\cA_V\cI$.
\begin{lem}
\label{lem:IleftAmod}
\begin{enumerate}
\item
$\cI\subseteq\cO_V$ has naturally the structure of a left $\cA_V$-module (and consequently, also $\cO_{\hat{X}}$ has).

\item
For any $\xi\in\fg'$, we have
\[
\cA_V\cdot \cI \cdot \xi \subseteq\cA_V\cdot \cI
\]
as subsets of $\cA_V$. Consequently, $\cA_V\cdot \cI$ is a two-sided ideal, and $\cA_V/\cA_V\cI$
is also a right $\cA$-module (i.e., it is a sheaf of  rings).
\end{enumerate}

\end{lem}
\begin{proof}
\begin{enumerate}
\item
Clearly, $\cO_V$ is a left $\cA_V$-module
through $\tilde{Z}_V \colon \cA_V\rightarrow \cD_V$. We need to show that this left action leaves $\cI$ invariant. Let $\xi\in \fg'$ and let $g\in \cI$ be given.
Consider the following piece
of the dual to the conormal sequence of $\hat{X}\subseteq V$
\[
0\longrightarrow
{\cD\!}er_{\dC}(\cO_{\hat{X}})
\longrightarrow
{\cD\!er}_{\dC}(\cO_V)\otimes_{\cO_V}\cO_{\hat{X}}
\stackrel{\alpha}{\longrightarrow} {H\!}om_{\cO_V}(\cI,\cO_{\hat{X}}),
\]
Since $\hat{X}\subseteq V$ is a $G'$-variety, $Z_V(\xi)$ descends to a derivation of $\cO_{\hat{X}}$, i.e., it lies in the kernel of the map $\alpha$. Therefore $Z_V(\xi)(g)\in \cI$.
\item
Since $\cA_V$ is the universal enveloping algebra of the Lie algebroid $\cO_V\otimes_\dC \fg'$, for any $g\in \cO_V$, the commutator
\[
\xi\cdot g-g\cdot\xi
\]
must be equal to the result of applying the anchor map to $\xi$, and then applying the correspondig derivation to $g$. But the anchor map $\cO_V\otimes_\dC\fg'\rightarrow \Theta_V$ is nothing but the scalar extension of $Z_V$, so that $\xi\cdot g-g\cdot\xi=Z_V(\xi)(g)$, which lies in $\cI$ by point 1. Consequently
\[
g\cdot \xi=\xi\cdot g-Z_V(\xi)(g) \in \cA_V\cdot \cI
\]
for $g\in \mathcal{I}$, as required.
\end{enumerate}
\end{proof}

We next consider a homological construction that can be considered as a generalization of both the Spencer complex in $\cD$-module theory (see, e.g. \cite[Lemma 1.5.27.]{Hotta}) and of the Euler-Koszul complex as defined in the theory of hypergeometric differential systems (\cite[Section 4]{MMW}) and which is closely related to Lie algebra cohomology resp.\ homology (see, e.g., \cite[Section VII.4]{HiltonStammbach}). We therefore call it the Euler-Koszul-Chevalley-Eilenberg-Spencer complex. Notice that a similar construction is also considered in several works in the theory of free divisors, see, e.g. \cite[Section 1.1.2.]{CN08} (one could also cite the much more classical reference \cite[Section 4]{Rinehart}).

Let first $\cN$ be a right $\cA_V$-module. Define
\[
\D \cS^{-\ell}(\cN) :=
\D \cN \otimes_{\cO_V}
\bigwedge_{\cO_V}^\ell \left(\cO_V\otimes_\dC \fg'\right)
=
\cN \otimes_\dC \bigwedge^\ell \fg',
\]
where the differential is as follows
\begin{equation}\label{eq:DiffEKCESComplex}
\begin{array}{rcl}
\D \delta^{-\ell} \colon  \cS^{-\ell}(\cN) &\longrightarrow& \cS^{-\ell+1}(\cN) \\ \\
m \otimes (\xi_1 \wedge \ldots \wedge \xi_\ell) &\longmapsto & \D \sum_{i=1}^\ell (-1)^{i-1} m (\xi_i-\beta'(\xi_i)) \otimes (\xi_1 \wedge \ldots \wedge  \widehat{\xi}_i \wedge\ldots \wedge \xi_\ell)+\\
&& \D \sum\limits_{1 \leq i < j \leq \ell} (-1)^{i+j} m \otimes ([\xi_i,\xi_j] \wedge \xi_1 \wedge \ldots \wedge \widehat{\xi}_i \wedge \ldots \wedge \widehat{\xi}_j \wedge \ldots \wedge \xi_\ell).
\end{array}
\end{equation}
where the right $\cA_V$-module structure on $\cN$ is used in the first term of the differential when writing $m (\xi_i-\beta'(\xi_i))$.
In general, $\cS^\bullet(\cN)$ will be a complex of sheaves of $\dC$-vector spaces only.

We will apply this construction several times, but in particular in the following more special situation.
Let $\mathcal{M}$ be a left $\mathscr{A}_V$-module (e.g.\ $\mathcal{O}_{\hat{X}}$). Consider the sheaf
\[ \mathscr{A}_V\otimes_{\mathcal{O}_V}\mathcal{M}.\]
We view this sheaf as an $(\mathscr{A}_V,\mathscr{A}_V)$-bimodule as follows: The left $\mathscr{A}_V$-action is given by
\[ b(a\otimes m) = ba\otimes m.\]
The right action is induced by
\begin{align*}
    (a\otimes m)f &= af \otimes m \qquad (f\in \mathcal{O}_V)\\
    (a\otimes m)\xi &= a\xi \otimes m - a \otimes \xi m\qquad (\xi \in \mathfrak{g}').
\end{align*}
It is easy to check that this construction extends to a functor from left $\cA_V$-modules to $(\cA_V,\cA_V)$-bimodules.

Note that this uses the natural right action on the tensor product of a right $\cA_V$-module (here, $\cA_V$) and a left $\cA_V$-module (here, $\cM$) over $\O_V$ in the realm of Lie algebroids (see e.g.~\cite[Proposition 2.2.1]{Che99}) analogous to the case $\mathscr D_V$-modules. The left action however, is the trivial one only acting on the first factor. This ensures that both actions describe a bimodule structure. One could also consider on $\cA_V \otimes_{\cO_V} \cM$ the left $\cA_V$-module structure as in \cite[Section (2.1)]{Huebschmann} obtained from the left $\cA_V$-module structures on $\cA_V$ and $\cM$ respectively and complement it with the trivial right $\cA_V$-module structure (by right action on the first factor) to obtain a bimodule structure. In fact, these two $(\cA_V,\cA_V)$-bimodule structures are isomorphic, as can be seen by an argument similar to \cite[Corollary A.12]{Narvaez-Duality-BS}

To consider a specific example, we can take $\cM:=\cO_{\hat{X}}$, which is a left $\cA_V$-module by \cref{lem:IleftAmod} above.
Let $\psi\colon \mathscr{A}_V\otimes_{\mathcal{O}_V}\mathcal{O}_{\hat{X}}\to \mathscr{A}_V/\mathscr{A}_V\mathcal{I}$, $a\otimes \overline{g}\mapsto \overline{a\cdot g}$ be the canonical isomorphism of left $\mathscr{A}_V$-modules. Now  by the previous construction, the left hand side is also a right $\cA_V$-module, and by invoking \cref{lem:IleftAmod} again, so is the right hand side. Then the morphism is also an isomorphism of right $\mathscr{A}_V$-modules: Since $\mathfrak{g}'$ kills the element 1 of $\mathcal{O}_{\hat{X}}$, we have (for $a\in \mathscr{A}_V$, $\overline{g}\in\cO_{\hat{X}}$ and $\xi\in \mathfrak{g}'$)
    \begin{align*}
        \psi((a\otimes \overline{g})\xi) &= \psi((ag\otimes 1)\xi) \\
        &= \psi(ag\xi\otimes 1 - ag\otimes (\xi\cdot 1))\\
        &= \psi(ag\xi\otimes 1)\\
        &= \overline{ag\xi}\\
        &= \overline{ag}\cdot \xi\\
        &= \psi(a\otimes \overline{g})\xi,
    \end{align*}
    as claimed.

We now apply the construction of the complex $\cS^\bullet(-)$ (taking as input a right $\cA_V$-module $\cN$) to the particular case where $\cN:=\cA_V\otimes_{\cO_V}\cM$, i.e., we put for all $\ell\in\dZ$
\[
\cC^{-\ell}(\cM):=\cS^{-\ell}(\cA_V\otimes_{\cO_V}\cM),
\]
yielding a complex $(\cC^\bullet, \delta)$. It is readily checked that since $\cA_V\otimes_{\cO_V}\cM$ is also a left $\cA_V$-module, the differentials $\delta^{-\ell}$ are now left $\cA_V$-linear. Again, it is an easy exercise to see that this construction is functorial, so that $\cC^\bullet(-)$ yields an exact functor from the category of left $\cA_V$-modules to the category of complexes of left $\cA_V$-modules.

Pursuing the above example where $\cM=\cO_{\hat{X}}$, we see immediately that
\[
H^0(\cC^\bullet(\cO_{\hat{X}})) \cong \hat\tau^\cA.
\]

We also have the following important homological property of this complex.

\begin{lem}
For any left $\cA_V$-module $\cM$,
$\cC^\bullet(\cM)$ is a resolution of $H^0(\cC^\bullet(\cM))$ by left $\cA_V$-modules
(which in general are not $\cA_V$-free though).
In particular, for $\cM=\cO_{\hat{X}}$, we obtain
that $\cC^\bullet(\cO_{\hat{X}})$ is a resolution of $\hat{\tau}^\cA$ by left $\cA$-modules.
\end{lem}
\begin{proof}
We follow a standard strategy by filtering
$\cC^\bullet(\cM)$ by degree using the natural
filtration on $\cA_V$. More precisely, using
$\cA_V=\cO_V\otimes_\dC\cU\fg'$, we set $F_k\cA_V:=\cO_V\otimes_\dC F_k(\cU\fg')$, where $F_\bullet(\cU\fg')$ is the standard filtration on the universal enveloping algebra. By the Poincar\'e-Birkhoff-Witt theorem, we have
\[
\gr^F_\bullet \cA_V \cong \cO_V\otimes_\dC \textup{Sym}^\bullet(\fg').
\]
We consider the induced filtration
$F_\bullet(\cA_V \otimes_{\cO_V} \cM) = F_\bullet(\cA_V) \otimes_{\cO_V} \cM$ on the left $\cA_V$-module
$\cA_V \otimes_{\cO_V} \cM$. Then we have the following isomorphism of $\cO_V\otimes_\dC \textup{Sym}^\bullet(\fg')$-modules
\[
\begin{array}{rcl}
\D \gr^F_\bullet(\cA_V \otimes_{\cO_V} \cM) & \cong &\D \gr^F_\bullet(\cA_V) \otimes_{\cO_V} \cM
\cong
\left(\cO_V\otimes_\dC \textup{Sym}^\bullet(\fg')\right) \otimes_{\cO_V} \cM
\\ \\
&\cong& \D
\textup{Sym}^\bullet(\fg') \otimes_{\dC} \cO_V \otimes_{\cO_V} \cM
\cong
\textup{Sym}^\bullet(\fg') \otimes_{\C} \cM.
\end{array}
\]
Notice that the action of elements from $\textup{Sym}^\bullet(\fg')$ on
$\textup{Sym}^\bullet(\fg') \otimes_{\C} \cM$ is only via the first factor.
Then we consider the filtration on $\cC^\bullet(\cM)$ defined as
\[
F_k \cC^{-\ell}(\cM) :=
F_{k-\ell}(\cA_V \otimes_{\cO_V} \cM) \otimes_\dC \bigwedge^\ell \fg'.
\]
This makes $F_\bullet\cC^\bullet(\cM)$ into a filtered complex (this is the reason one needs to shift the filtration on the various terms of the complex, since otherwise the first summand in the formula defining the differential of $\cC^\bullet(\cM)$, i.e. in formula \eqref{eq:DiffEKCESComplex} applied to the case $\cN=\cA_V\otimes \cM$, would not respect this filtration), and by the usual arguments one checks that
\[
\gr^F_\bullet \cC^\bullet(\cM)
\cong \textup{Kos}^\bullet(\cM \otimes_\dC \textup{Sym}^\bullet(\fg'), (\xi_1,\ldots,\xi_{\dim(\fg')})),
\]
for some basis $(\xi_1,\ldots,\xi_{\dim(\fg')})$ of the Lie algebra $\fg'$. Since clearly
$\xi_1,\ldots,\xi_{\dim(\fg')}$ is a regular sequence on $\cM\otimes_\dC \textup{Sym}^\bullet(\fg')$, we obtain $H^i(\gr_\bullet^F \cC^\bullet(\cM)) = 0$ for $i<0$. Then by a general argument (see, e.g. \cite[Theorem 4.3.5]{SST}) it follows that
\[
H^0(\gr^F_\bullet(\cC^\bullet(\cM)))=\gr^F_\bullet H^0(\cC^\bullet(\cM)).
\]
We have therefore shown that $\gr^F_\bullet \cC^\bullet(\cM)$ is a resolution of
$\gr^F_\bullet H^0(\cC^\bullet(\cM))$, but then
the original complex $\cC^\bullet(\cM)$ is a resolution of $H^0(\cC^\bullet(\cM))$, which is the first statement of the lemma.
Since, as remarked above, we have
$H^0\cC^\bullet(\cO_{\hat{X}}) \cong \hat\tau^\cA$,
$\cC^\bullet(\cO_{\hat{X}})$ is a resolution of
$\hat\tau^\cA$
by left $\cA_V$-modules.
\end{proof}
The terms of the complex $\cC^\bullet(\cM)$ are not $\cA_V$-free in general. This is cured by the following construction.

\begin{lem}
  There exists a finite resolution $(\cF^\bullet(\cM), d)\twoheadrightarrow \cM$ by left $\cA_V$-modules that are  free over $\cO_V$.
\end{lem}
\begin{proof}
    We first construct via induction an infinite resolution $\mathcal{G}^\bullet$ of $\mathcal{M}$ by left $\cA_V$-modules that are     free (but possibly of infinite rank) over $\mathcal{O}_V$.

    Let $W^0$ be the $\mathfrak{g}'$-submodule generated by a     global $\mathcal{O}_V$-generating set of $\mathcal{M}$. Then $\mathcal{G}^0:=\mathcal{O}_V\otimes_\C W^0$ is a left $\cA_V$-module via
    \begin{align*}
        f\cdot (g\otimes w) &= (fg)\otimes w \qquad (f\in \mathcal{O}_V),\\
        \xi\cdot (g\otimes w) &= (Z_V(\xi)(g))\otimes w + g \otimes (\xi\cdot w) \qquad (\xi\in \mathfrak{g}').
    \end{align*}
    The obvious map $\mathcal{G}^0\to \mathcal{M}$ is surjective and $\cA_V$-linear. Repeating this procedure with $\ker(\mathcal{G}^0\to \mathcal{M})$, and continuing in that way, we get an infinite resolution $\mathcal{G}^\bullet$ of $\mathcal{M}$ of the required type.

    We now construct $\mathcal{F}^\bullet(\mathcal{M})$: Since $\mathcal{O}_V$ has finite global dimension (say equal to $n$), $\im(\mathcal{G}^{-n}\to\mathcal{G}^{-n+1})$ is $\mathcal{O}_V$-projective (see, e.g., \cite[\href{https://stacks.math.columbia.edu/tag/00O5}{Lemma 00O5}]{stacks-project}) and therefore $\mathcal{O}_V$-free. Thus,
    \[
        \mathcal{F}^{-i}(\mathcal{M})
        := \begin{cases}
            \mathcal{G}^{-i}, &\text{if }i<n,\\
            \im(\mathcal{G}^{-n}\to\mathcal{G}^{-n+1}), &\text{if }i=n,\\
            0, &\text{if }i>n.
        \end{cases}
    \]
    with the differential induced from $\mathcal{G}^\bullet$ works.
\end{proof}
\begin{rmk}
For what follows, a resolution of modules that have possibly infinite rank over $\cO_V$ as just constructed is sufficient. However, it is actually possible to obtain a resolution by finite rank $\cO_V$-modules under the additional assumption that $\cM$ is graded,
and finitely generated over $\cO_V$ by homogeneous elements such that the grading is compatible with the left $\cA_V$-structure on $\cM$. This is in particular the case for $\cM=\cO_{\hat{X}}$, which is the only case that we will use below. Namely, under these assumptions, the $\fg'$-submodule $W^0$ constructed in each step is then necessarily contained in a finite number of homogeneous components of $\cM$, i.e. in a finite dimensional vector space. This suffices to obtain a free $\cO_V$-module of finite rank $\cG^0$ as above,
which is again graded in a compatible way with the left $\cA_V$-action, and then one argues again by induction.
\end{rmk}

In the sequel, we specialize to the case $\cM=\cO_{\hat{X}}$.
According to the previous lemma, by applying the functor
$\cC^\bullet(-)$ to the $\cO_V$-free resolution $\cF^\bullet(\cO_{\hat{X}}) \twoheadrightarrow \cO_{\hat{X}}$ by left $\cA_V$-modules,
we  obtain the double complex
\[
\cK^{\bullet,\bullet} :=
\cC^\bullet(\cF^\bullet(\cO_{\hat{X}}))
\]
and its associated total complex $\textup{Tot}^\bullet(\cK^{\bullet,\bullet})$. Then
$\textup{Tot}^\bullet(\cK^{\bullet,\bullet})$ provides a resolution of $\hat\tau^\cA$ by free left $\cA_V$-modules (of possibly infinite rank). Therefore, we have
\[
\omega_V\otimes^\dL_{\cA_V}\hat\tau^\cA\cong
\omega_\cA \otimes_{\cA_V} \textup{Tot}^\bullet(\cK^{\bullet,\bullet})
\cong
\textup{Tot}^\bullet(\omega_V \otimes_{\cA_V}\cK^{\bullet,\bullet})
\]

Consider the spectral sequence  associated to the double complex $\omega_V \otimes_{\cA_V}\cK^{\bullet,\bullet}$
with $E_1$-term given by first taking vertical cohomology, i.e.
\[
E_1^{p,q}:=H^q(\omega_V\otimes_{\cA_V}
\cC^p(\cF^\bullet(\cO_{\hat{X}}))) \Longrightarrow
H^{p+q}(\omega_V\otimes^\dL_{\cA_V}\hat\tau^\cA).
\]
Then we have the following
\begin{lem}
The above sequence collapses at the $E_1$-term, and we have
\[
\omega_V\otimes^\dL_{\cA_V}\hat\tau^\cA
\simeq  (\cS^\bullet(\omega_V\otimes_{\cO_V}\cO_{\hat{X}}), \delta),
\]
where we consider the right module structure on $\cN=\omega_V\otimes_{\cO_V}\cO_{\hat{X}}$ coming from the tensor product of the right $\cA_V$-module $\omega_V$ with the left $\cA_V$-module $\cO_{\hat{X}}$.
Explicitly, we have
\[
\cS^\ell(\omega_V\otimes_{\cO_V}\cO_{\hat{X}}):=
\omega_V/\cI\omega_V\otimes_\dC \bigwedge^{-\ell}\fg',
\]
and where the differentials are
\[
\begin{array}{rcl}
\delta^{-\ell} \colon
\D\frac{\omega_V}{\cI\omega_V}\otimes_\dC \bigwedge^{-\ell}\fg'
&\longrightarrow&
\D\frac{\omega_V}{\cI\omega_V}\otimes_\dC \bigwedge^{-\ell+1}\fg'
\\ \\
(f\cdot\vol) \otimes (\xi_1 \wedge \ldots \wedge \xi_\ell) &\longmapsto &
\SC \sum\limits_{i=1}^\ell (-1)^{i-1}
(-\Lie_{Z_V(\xi_i)}-\beta'(\xi_i))(f\cdot\vol) \otimes (\xi_1 \wedge \ldots \wedge  \widehat{\xi}_i \wedge\ldots \wedge \xi_\ell)+\\
&& \SC \sum\limits_{1 \leq i < j \leq \ell} (-1)^{i+j} (f\cdot\vol )\otimes ([\xi_i,\xi_j] \wedge \xi_1 \wedge \ldots \wedge \widehat{\xi}_i \wedge \ldots \wedge \widehat{\xi}_j \wedge \ldots \wedge \xi_\ell)
\end{array}
\]
\end{lem}
\begin{proof}
According to the above construction, we have
\begin{align*}
    &\D H^q(\omega_V\otimes_{\cA_V}
\cC^p(\cF^\bullet(\cO_{\hat{X}})))\\
    &\quad= \D\frac{\D \ker\left(\omega_V\otimes_{\cA_V}(\cA_V\otimes_{\cO_V}\cF^q(\cO_{\hat{X}}))\otimes_\dC\bigwedge^p\fg',\id\otimes\id\otimes d^q\otimes\id\right)}{\D \textup{im}\left(\omega_V\otimes_{\cA_V}(\cA_V\otimes_{\cO_V}\cF^{q-1}(\cO_{\hat{X}}))\otimes_\dC\bigwedge^p\fg',\id\otimes\id\otimes d^{q-1}\otimes\id\right)}\\
    &\quad= \D \omega_V\otimes_{\cO_V} H^q(\cF^\bullet(\cO_{\hat{X}}))\otimes_\dC\bigwedge^p \fg'\\
    &\quad=\begin{cases}
        0, &\text{if }q<0,\\
        \omega_V\otimes_{\cO_V}\cO_{\hat{X}}\otimes_\dC\bigwedge^p \fg', &\text{if }q=0.
    \end{cases}
\end{align*}
(recall that $d$ is the differential of the complex $\cF^\bullet(\cO_{\hat{X}})$)
from which it is obvious that the spectral sequence collapses, and that the induced differential $\delta^{-\ell} \colon \cS^\ell(\omega_V\otimes_{\cO_V}\cO_{\hat{X}})=E_1^{\ell,0}\rightarrow E_1^{\ell+1,0}=\cS^{\ell+1}(\omega_V\otimes_{\cO_V}\cO_{\hat{X}})$ is as indicated.
\end{proof}

Using all these preliminaries, we finally obtain the vanishing of the two de Rham cohomology groups we are interested in.
\begin{proof}[Proof of \cref{prop:VanishCohomA}]
It remains to show that, under the assumptions of the proposition (in particular, $\beta(\mathbf{e})\notin \dZ_{\leq 0}$), we have $H^k(\cS^\bullet(\omega_V\otimes_{\cO_V} \cO_{\hat{X}}))=0$ for $k=0,-1$. Let us first notice that the complex $\cS^\bullet(\omega_V\otimes_{\cO_V}\cO_{\hat{X}})$
is naturally graded by the grading of $\cO_{\hat{X}}$ and of $\cO_V$ (by putting $\deg(\vol):=\dim(V)$) and by setting $\deg(\fg'):=0$.
Then it is easily verified that the morphism $Z_V$ is homogeneous of degree $0$, and therefore also the differentials $\delta^{-\ell}$ are so. Consequently, it suffices to calculate the cohomology of the graded parts of $\cS^\bullet(\omega_V\otimes_{\cO_V}\cO_{\hat{X}})$.

The relevant maps in this complex are as follows:
\[
\begin{array}{rcl}
\D\delta^{-1} \colon \frac{\omega_V}{\cI\omega_V}\otimes_\dC \fg'
&\longrightarrow&
\D\frac{\omega_V}{\cI\omega_V}
\\ \\
(f\cdot \vol)\otimes \xi & \longmapsto & (-\Lie_{Z_V(\xi)}-\beta'(\xi))(f\cdot \vol)
\\ \\ \\ \\
\D\delta^{-2} \colon \frac{\omega_V}{\cI\omega_V}\otimes_\dC \bigwedge^2 \fg'
&\longrightarrow&
\D\frac{\omega_V}{\cI\omega_V}\otimes_\dC \fg'
\\ \\
(f\cdot \vol)\otimes\vartheta\wedge\eta &\longmapsto &
(-\Lie_{Z_V(\vartheta)}-\beta'(\vartheta))(f\cdot \vol)\otimes \eta
+
(\Lie_{Z_V(\eta)}+\beta'(\eta))(f\cdot \vol)\otimes \vartheta \\
&& -(f\cdot \vol)\otimes[\vartheta,\eta] \\ \\
&&\D=
\delta^{-1}((f\cdot\vol)\otimes\vartheta)\otimes\eta - \delta^{-1}((f\cdot\vol)\otimes\eta)\otimes\vartheta
 -(f\cdot \vol)\otimes[\vartheta,\eta].
\end{array}
\]
In order to describe these morphisms, first  notice that for any $\theta\in\fg'$, we have
\[
\begin{array}{rcl}
\D \Lie_{Z_V(\theta)}(\vol)& =&
\D \Lie_{-\sum_{i,j}d\rho(\theta)_{ji}x_i\partial_{x_j}}(\vol)=
-\sum_{i,j}d\rho(\theta)_{ji}\Lie_{x_i\partial_{x_j}}(\vol)\\ \\
&=&\D
-\sum_{i,j}d\rho(\theta)_{ji}\delta_{ij} \cdot \vol=
-\textup{trace}(d\rho(\theta)) \cdot \vol.
\end{array}
\]
We thus get
\begin{equation}\label{eq:CancelLieVolTrace}
\begin{array}{rcl}
\D \left(\Lie_{Z_V(\theta)}+\beta'(\theta)\right)(f\cdot \vol) &=&
\big(Z_V(\theta)(f)-f\cdot\textup{trace}(d\rho(\theta))+\underbrace{\beta'(\theta)}_{=\textup{trace}(d\rho(\theta))-\beta(\theta)}\cdot f\big)\cdot\vol\\ \\
&=& \D\left(Z_V(\theta)(f)-\beta(\theta)f\right)\cdot \vol
\end{array}
\end{equation}
After these preliminaries, let us first show that $H^0(\cS^\bullet(\omega_V\otimes_{\cO_V} \cO_{\hat{X}}))=0$, i.e., that the morphism $\delta^{-1}$ is surjective.
According to \cref{lem:equivVFOnVectSp}, for $\xi=\mathbf{e}\in\fg'$, we have
\[
Z_V(\mathbf{e})=-E:=-\sum_{i=1}^{\dim V} x_i\partial_{x_i},
\]
where $x_1,\ldots,x_{\dim(V)}$ are coordinates on $V$.
We thus have
\[
\begin{array}{rcl}
\D \delta^{-1}\left((f\cdot\vol)\otimes\mathbf{e}\right) &=&
\D (-\Lie_{Z_V(\mathbf{e})}-\beta'(\mathbf{e}))(f\cdot\vol)=
(E(f)+\beta(\mathbf{e})f)\cdot\vol
\end{array}
\]
Since $E(f) =d \cdot f$ for $f$ homogeneous of (non-negative) degree $d$, the fact that $\beta(\mathbf{e})\notin\dZ_{\leq 0}$  shows that $\delta^{-1}$ is surjective, hence $H^0(\mathcal S^\bullet(\omega_V\otimes_{\cO_V} \cO_{\hat{X}})) = 0$.

The vanishing of $H^{-1}(\cS^\bullet(\omega_V\otimes_{\cO_V} \cO_{\hat{X}}))$ will similarly be shown in each degree of the complex. Therefore, suppose that  we have homogeneous elements $f_i\in\cO_{\hat{X}}$ all of which have the same degree $d\in\dZ_{\geq 0}$ and $\xi_i\in \fg'$ for $i=1,\ldots,r$ such that
\[
\delta^{-1}\bigg(
\sum_{i=1}^r (f_i\cdot \vol)\otimes \xi_i\bigg)=0.
\]
By assumption, we have $d+\beta(\mathbf{e})\neq 0$. Then it follows (using $[\mathbf{e}, \xi_i]=0$) that
\[
\begin{array}{l}
\D \delta^{-2}\sum_{i=1}^r
(\frac{f_i}{d+\beta(\mathbf{e})}\cdot\vol)\otimes \mathbf{e}\wedge \xi_i=
 \sum_{i=1}^r
\delta^{-2}\bigg(
\frac{f_i}{d+\beta(\mathbf{e})}\cdot\vol\otimes \mathbf{e}\wedge \xi_i\bigg)\\ \\
= \D
\sum_{i=1}^r\bigg(\delta^{-1}\big((\frac{f_i}{d+\beta(\mathbf{e})}\cdot\vol)\otimes \mathbf{e}\big)\otimes \xi_i\bigg) -
\sum_{i=1}^r\delta^{-1}\bigg((\frac{f_i}{d+\beta(\mathbf{e})}\cdot\vol)\otimes \xi_i\bigg)\otimes \mathbf{e}
\\ \\
= \D
\sum_{i=1}^r\bigg(\delta^{-1}\big((\frac{f_i}{d+\beta(\mathbf{e})}\cdot\vol)\otimes \mathbf{e}\big)\otimes \xi_i\bigg)
-\frac{1}{d+\beta(\mathbf{e})}
\underbrace{\sum_{i=1}^r\delta^{-1}\bigg((f_i\cdot\vol)\otimes \xi_i\bigg)}_{=0}\otimes \mathbf{e}
\\ \\
=\D
\sum_{i=1}^r \bigg(\frac{E(f_i)+\beta(\mathbf{e})f_i}{d+\beta(\mathbf{e})}\cdot \vol\otimes \xi_i
\bigg),
\end{array}
\]
(recall that the elements $f_i\in \cO_{\hat{X}}$ were chosen to be homogeneous of degree $d$, so that $E(f_i)=d\cdot f_i$, and, consequently, $\frac{E(f_i)+\beta(\mathbf{e})f_i}{d+\beta(\mathbf{e})}=f_i$) so that $\sum_{i=1}^r (f_i\cdot \vol)\otimes \xi_i\in\textup{im}(\delta^{-2})$, thus showing $H^{-1}(\cS^\bullet(\omega_V\otimes_{\cO_V}\cO_{\hat{X}}))=0$.
\end{proof}

The next statement summarizes the results obtained so far in \cref{sec:tautologicalSystems}. We consider  the situation as described before \cref{thm:restrictedTauHatDescription} and we would like to describe how the $\cD_V$-module $\hat{\tau}(\rho, \hat{X},\beta)$ is related to its restriction to $V\setminus\{0\}$.
We only state the results under the simplifying assumption that $G$ is semisimple, since this is the main case of interest and since it allows us to use the results proved in \cref{sec:FormuleBeta}.
Recall that under the assumption that $G$ (and consequently its Lie algebra $\fg$) is semisimple, we necessarily have $\beta_{|\fg} = 0$.
\begin{cor}\label{cor:LocColocHatTau}
In the above situation, assume that $\tau(\rho,\hat{X},\beta)\neq 0$. Then we have:
\begin{enumerate}
    \item $\beta(\mathbf{e})\in\dQ_{\geq 0}$.
    \item
    If $\beta(\mathbf{e}) = 0$, then $\tau(\rho, \hat{X},\beta)$ is a free $\cO_W$-module of finite positive rank.

    \item
    If $\beta(\mathbf{e}) \in \dQ_{>0}$, then
    we have an isomorphism in $\textup{Mod}_h(\cD_V)$
    $$
        H^0j_\dag j^+ \hat{\tau}(\rho, \hat{X}, \beta) \stackrel{\cong}{\longrightarrow} \hat{\tau}(\rho, \hat{X}, \beta)
    $$

    \item
    If $\beta(\mathbf{e}) \in \dQ_{>0}\setminus \dZ_{>0}$, then
    we have isomorphisms in $\textup{Mod}_h(\cD_V)$
    $$
        \begin{array}{rcl}
            j_+ j^+ \hat{\tau}(\rho, \hat{X}, \beta) &\stackrel{\cong}{\longrightarrow} & \hat{\tau}(\rho, \hat{X}, \beta),
           \\ \\
                 & \textup{and} \\ \\
            j_+j^+\hat{\tau}(\rho, \hat{X}, \beta) & \cong & j_\dag j^+ \hat{\tau}(\rho, \hat{X}, \beta)
        \end{array}
    $$
    in particular, we have $H^i(j_\star j^+\hat{\tau}(\rho, \hat{X}, \beta))=0$ for $i\neq 0$ and for $\star\in\{+,\dag\}$ in this case. Furthermore, $\tau(\rho,\hat{X},\beta)$ is simple.
\end{enumerate}

\end{cor}

\begin{proof}
\begin{enumerate}
\item
This is exactly the statement of \cref{cor:BetaNonNeg}.

\item
It has been shown in \cref{lem:Fano} that $X$ is Fano. This implies by \cref{thm:restrictedTauHatDescription} that if $\beta(\mathbf{e})=0$, then necessarily $\hat{\tau}(\rho,\hat{X},\beta)_{|V\backslash\{0\}}=0$. Therefore, $\hat{\tau}(\rho,\hat{X},0)$ has support in the origin in $V$ and consequently, $\tau(\rho,\hat{X},0)$ is a free $\cO_W$-module (of positive rank by the assumption $\tau(\rho,\hat{X},\beta)\neq 0$). Notice that the non-vanishing of $\tau(\rho,\hat{X},\beta)\neq 0$ is automatic if $\beta(\mathbf{e})=0$: in this case, any constant function on $W$ is a (classical) solution to $\tau(\rho,\hat{X},0)$, since it is annihilated by any operator in the denominator.
\item
This follows directly from \cref{theo:coloc}.
\item
The isomorphism
$j_+ j^+ \hat{\tau}(\rho, \hat{X}, \beta) \cong \hat{\tau}(\rho, \hat{X}, \beta)$ is exactly the content of \cref{cor:localization} (applying it for $Y=\hat{X}\setminus \{0\}$ and $\overline{Y}=\hat{X}$). Moreover, the second isomorphism is obviously true if $\hat{\tau}(\rho, \hat{X}, \beta)=0$. Otherwise, we must have by the first isomorphism that
$j^+\hat{\tau}(\rho, \hat{X}, \beta)\neq 0$, but then by
\cref{thm:restrictedTauHatDescription} we know that
$$
j^+\hat{\tau}(\rho, \hat{X}, \beta) \cong
i'_+ \O_{L^*}^{\ell/k}
$$
with $\beta(\mathbf{e})=\ell/k$. Since $i'$ is proper, we are therefore left to show that
\begin{equation}\label{eq:pluseqdag}
j_+ i'_+ \O_{L^*}^{\ell/k} \cong j_\dag i'_\dag \O_{L^*}^{\ell/k},
\end{equation}
but this follows from the proof of \cref{prop:directImageOfOBetaMHM}, points 1. and 3., by noticing that we have $\hat{\iota}=j\circ i' \circ \textup{inv}^{-1} \colon L^{\vee,*}\hookrightarrow V$.  It follows from Eq. \eqref{eq:pluseqdag} that $\hat{\tau}(\rho, \hat{X}, \beta)$ is an intermediate extension of $i'_+ \O_{L^*}^{\ell/k}$.  Since $ i'_+ \O_{L^*}^{\ell/k}$ is simple, we conclude that $\hat{\tau}(\rho, \hat{X}, \beta)$ is simple as well. Since $\FL^V$ is an equivalence of categories the claim follows. \end{enumerate}
\end{proof}

\begin{rmk}\label{rmk:RemarkRefColoc}
  Similar to \cref{rmk:RemarkToReferee}, it is possible to alternatively embrace equivariant viewpoints to approach the (co-)localization properties of $\hat\tau(\rho, \overline Y, \beta)$, as pointed out by a referee. One observation is that $\mathscr D_V$-modules supported on the origin can only be strongly $\beta'$-equivariant for $\beta(\mathbf e) \in \Z$ and only have $G'$-invariant sections for $\beta(\mathbf e) = 0$, which can be used to recover \cref{th:reptheory-beta}. A more detailed analysis also leads to \cref{theo:coloc}.
\end{rmk}

\subsection{Tautological systems as mixed Hodge modules}
\label{subsec:TautMHM}

The purpose of this section is to finally achieve the functorial construction of tautological systems
announced in the introduction (more specifically, in  \cref{theo:Main}), by combining the results in
 \cref{sec:TwistCOhomHypSec}, the description of $\hat\tau_{|V\setminus \{0\}}$ from \cref{thm:restrictedTauHatDescription} as well as the localization resp.\ colocalization properties of $\hat\tau$ summarized in \cref{cor:LocColocHatTau} above.

Let us recall once again the setup we are working with: We let $X$ be a projective variety and we consider a transitive action of a reductive connected algebraic group $G$ on $X$.
We let $L \rightarrow X$ be a very ample $G$-equivariant line bundle. We extend the group action on $X$ and $L$ to a an action of $G' := \C^* \times G$ by letting the $\C^*$-factor act trivially on $X$ and by inverse scaling in the fibers of $L$.
We consider the $G'$-representation $V:=H^0(X,\mathscr{L})^\vee$ and the equivariant closed embedding $X\hookrightarrow \P V$ defined by $|\mathscr{L}|$.
Let $\hat{X}\subseteq V$ be the affine cone of $X$ in $V$, and we have an isomorphism $\hat{X} \setminus \{0\} \cong L^{\vee,*}$ by identifying $L^{\vee}$ with the blow-up of $\hat X$ at the origin. We write $\hat\iota \colon L^{\vee,*} \cong \hat X \setminus \{0\} \to V$  for the locally closed embedding given as the composition of
the closed embedding $i \colon \hat{X}\setminus \{0\} \hookrightarrow V\setminus \{0\}$ with the canonical open embedding
$j \colon V\setminus \{0\} \hookrightarrow V$.
Together with the isomorphism $\textup{inv} \colon L^* \to L^{\vee,*}$ given by inverting fibers, we obtain a locally closed embedding $\iota \colon L^* \hookrightarrow V$ defined by $\iota:=\hat\iota\circ \textup{inv}$.

For the convenience of the reader, we summarize the various maps that occur by extending the diagram \eqref{diag:Maps} from \cref{ssec:canSheafOnLineBdl}.
\begin{equation}\label{diag:Maps2}
\begin{tikzcd}
L & L^*\ar[swap,hook']{l}{j_L} \ar[swap]{dd}{\textup{inv}}[swap]{\cong} \ar[hook]{rrd}{i'} \ar[bend left]{rrrd}{\iota} & & &  \\
&&\hat{X} \setminus \{0\} \ar[hook]{r}{i} & V \setminus \{0\} \ar[hook]{r}{j} & V
\\
\textup{Bl}_{\{0\}}(\hat{X})\cong L^\vee&  L^{\vee,*}\ar[hook']{l}{j_{L^\vee}} \ar{ur}{\cong}
\ar[bend right]{rrru}{\hat{\iota}}&&&
\end{tikzcd}
\end{equation}

We let, as before,
$\beta \colon \mathfrak g' \to \C$ be a Lie algebra homomorphism satisfying $\restr{\beta}{\mathfrak g} \equiv 0$. Denote $W := V^\vee$.

\begin{thm}\label{theo:tau_is_MHM}
Under the above hypotheses, the following statements hold.
\begin{enumerate}
    \item
                Assume that
        $\beta(\mathbf{e}) = \ell/k \in \Q \setminus \Z$. The tautological system $\tau(\rho,\hat{X},\beta)$ is non-zero if and only if $\Ell^{\otimes \ell} \cong \omega_X^{\otimes (-k)}$ as $G$-equivariant line bundles. In this case, we have
        isomorphisms
        \[
        \tau(\rho,\hat X, \beta) \cong \FL^V(\iota_+ \O_{L^*}^{\ell/k}) \cong \FL^V(\iota_\dag \O_{L^*}^{\ell/k})
        \]
        in $\textup{Mod}(\mathscr{D}_W)$, and the $\mathscr{D}_W$-module $\tau(\rho,\hat X, \beta) $ underlies a complex pure Hodge module on $W$ of weight $\dim(X)+\dim(W)$. Moreover, $\tau(\rho,\hat{X},\beta)$ is simple, and, consequently, the local system associated to
        $\tau(\rho,\hat{X},\beta)_{| W \backslash \textup{Sing}(\tau(\rho,\hat{X},\beta))}$ is irreducible.
   \item
        If $\beta(\mathbf{e})\in \dZ_{>0}$,
        then $\tau(\rho,\hat{X},\beta)$ is non-zero if and only if $\Ell^{\otimes \beta(\mathbf{e})} \cong \omega_X^\vee$ as $G$-equivariant line bundles, in which case we have  an isomorphism
        $$
            \tau(\rho, \hat{X},\beta) \cong \FL^V(H^0\iota_\dag \cO_{L^*})
        $$
        in $\textup{Mod}(\cD_W)$. Then the $\mathscr D_W$-module $\tau(\rho, \hat{X},\beta)$ underlies an element of $\MHM(W)$ with weights in $\{\dim(W)+\dim(X),\,\dim(W)+\dim(X)+1\}$.
\end{enumerate}
\end{thm}

\begin{proof}
Under the assumptions of the theorem, we have in both cases that
\[
  j^+\hat\tau(\rho, \hat{X}, \beta) \cong
  \begin{cases}
    i'_+ \O_{L^*}^{\ell/k} &\text{if }     \Ell^{\otimes \ell} \cong \omega_X^{\otimes (-k)}, \\
        0 & \text{otherwise}
  \end{cases}
\]
by \cref{thm:restrictedTauHatDescription} (with $\ell = \beta(\mathbf e)$, $k = 1$ in case $2$), notice that we had implicitly identified $L^*$ with $\hat{X}\setminus \{0\}$ in \cref{thm:restrictedTauHatDescription}.

We now distinguish the two cases.
\begin{enumerate}
    \item
    Since
    $\ell/k\notin \dZ$, we know from
    \cref{cor:localization}
        that
    $$
    \hat{\tau}(\rho, \hat{X},\beta)\cong j_+j^+\hat{\tau}(\rho, \hat{X},\beta).
    $$
    Therefore, since $\iota=j\circ i'$, we conclude that
    $$
    \hat{\tau}(\rho, \hat{X},\beta)\cong \begin{cases} \iota_+ \cO_{L^*}^{\ell/k} & \text{if } \Ell^{\otimes \ell} \cong \omega_X^{\otimes (-k)}, \\
    0 &\text{otherwise}.
    \end{cases}
    $$
    As we have $\FL^V(\hat\tau(\rho,\hat X, \beta)) = \tau(\rho,\hat X, \beta)$ by \cref{def:tautSys}, we obtain
    $$
    \tau(\rho,\hat X, \beta) \cong \begin{cases}
    \FL^V(\iota_+ \O_{L^*}^{\ell/k})
    &\text{if } \Ell^{\otimes \ell} \cong \omega_X^{\otimes (-k)}, \\
    0 &\text{otherwise},
    \end{cases}
    $$
    as required. The fact that
    $\FL^V(\iota_+ \O_{L^*}^{\ell/k}) \cong \FL^V(\iota_\dag \O_{L^*}^{\ell/k})$ is simply the $\cD$-module version of \cref{prop:directImageOfOBetaMHM}, point 3. Then it follows as in \cref{cor:LocColocHatTau} that $\tau(\rho,\hat{X},\beta)$ is a simple $\cD_W$-module, and in particular that the local system (and the monodromy representation) of its restriction to its smooth part is irreducible.

    For the second statement, assume $\Ell^{\otimes \ell} \cong \omega_X^{\otimes (-k)}$ and recall from
    \cref{prop:directImageOfOBetaMHM} that
    $$
    \FL^V(\hat\iota_+ \O_{L^{\vee,*}}^{-\ell/k}) \cong a_{W,+} ev^\dag j_{L,\dag}\cO_{L^*}^{\ell/k}
    $$ as elements in $D^b_h(\mathscr{D}_W)$, using the notations from \cref{prop:directImageOfOBetaMHM}.     Since we have $\textup{inv}_+ \O_{L^{\vee,*}}^{-\ell/k} \cong \O_{L^*}^{\ell/k}$, we get
    $$
    \FL^V(\iota_+ \O_{L^*}^{\ell/k}) \cong a_{W,+} ev^\dag j_{L,\dag}\cO_{L^*}^{\ell/k}
    $$
    in $D^b_h(\mathscr D_W)$. However, as we have just proved, this is actually a single degree complex isomorphic to the tautological system $\tau(\rho,\hat X, \beta)$. Hence it follows from the second statement of \cref{prop:directImageOfOBetaMHM} that this $\mathscr{D}_W$-module underlies the pure complex Hodge module
    $$
     {^H\!\!}\cM_L^{\ell/k}=H^0( {^H\!\!}\cM_L^{\ell/k}) =a_{W,*} ev^* j_{L,!} {^H\!}\underline{\dC}_{L^*}^{\ell/k}[\dim W -1]
    $$
    which has weight $\dim(X)+\dim(W)$.

    \item Since $\beta(\mathbf e) \notin \Z_{\leq 0}$, we know from \cref{theo:coloc} that
    \[\hat\tau(\rho, \hat X, \beta) \cong \begin{cases} H^0 j_\dag j^+ \hat\tau(\rho, \hat X, \beta)  &\text{if } \Ell^{\otimes \beta(\mathbf e)} \cong \omega_X^\vee, \\
    0 & \text{otherwise.}
    \end{cases}\]
    Hence,
    \[\hat\tau(\rho, \hat X, \beta) \cong \begin{cases} H^0 j_\dag i'_+ \O_{L^*} &\text{if } \Ell^{\otimes \beta(\mathbf e)} \cong \omega_X^\vee, \\
    0 & \text{otherwise,}
    \end{cases}\]
    using $\O_{L^*}^{\beta(\mathbf{e})} \cong \O_{L^*}$ since $\beta(\mathbf e) \in \Z$. Since $i'$ is a closed embedding, we have $i'_+ \cong i'_\dag$, so, using $\iota = j \circ i'$, we conclude the first statement.

    The second statement then follows again from \cref{prop:directImageOfOBetaMHM}, points 1. and 2. More precisely, we had shown there that $\FL^V(\hat{\iota}_\dag \cO_{L^{\vee,*}})$ underlies ${^{H,!}\!\!}\cM_L\in\MHM(W)$, so that
    $$
        \tau(\rho,\hat{X},\beta)
        \cong
        \FL^V(H^0 \iota_\dag \cO_{L^*})
        \cong
        \FL^V(H^0 \hat{\iota}_\dag \cO_{L^{\vee,*}})
        \cong
        H^0\FL^V(\hat{\iota}_\dag \cO_{L^{\vee,*}})
    $$
    underlies $H^0({^{H,!}\!\!}\cM_L)\in\MHM(W)$. The weight estimate then follows directly from \cref{prop:directImageOfOBetaMHM}, point 4. for the case $k=0$.    \qedhere
\end{enumerate}
\end{proof}

As a corollary, we solve the holonomic rank problem from \cite[Conjecture 1.3.]{Bloch_HolRankProb} in general (i.e. for all homogeneous spaces and all possible equivariant line bundles that give rise to non-zero tautological systems). Recall from the discussion before \cref{prop:TautIsDirectImage} that $\cU:=(W\times X) \setminus ev^{-1}(0) \subseteq W\times X$ and that $a_\cU\colon \cU\rightarrow W$ denotes the restriction of the first projection $a_W \colon W\times X\rightarrow W$.
Moreover, for any $\lambda\in W$, we write $i_\lambda \colon \{\lambda\}\hookrightarrow W$ for the corresponding closed embedding, we let $U_\lambda\subset X$ be the complement of the zero locus of the section $\lambda \colon  X\rightarrow L$, and we denote by $\underline{\dC}^\beta_\lambda$ the complex local system on $U_\lambda$ that underlies the pure complex Hodge module $\lambda^*_{|U_\lambda}\, ^{H\!}\underline{\dC}^\beta_{L^*}[-1]$.
\begin{cor}\label{cor:holRankProb}
\begin{enumerate}
    \item
        Under the assumptions of \cref{theo:tau_is_MHM}, point 1., i.e., $\beta(\mathbf{e}) =\ell/k\in\dQ\setminus\dZ$ and $\Ell^{\otimes \ell} \cong \omega_X^{\otimes (-k)}$ as $G$-equivariant line bundles, we have isomorphisms in $\textup{Mod}_h(\mathscr{D}_W)$
        $$
            \tau(\rho,\hat X, \beta) \cong a_{\cU,\dag} {ev_{|\cU}^+}\cO_{L^*}^{-\ell/k} \cong a_{\cU,+} {ev_{|\cU}^+}\cO_{L^*}^{-\ell/k}.
        $$
        As a consequence, we have an isomorphism of vector spaces

        \begin{equation}\label{equ:solHolRank1a}
            H^m (i_\lambda^+ \tau(\rho, \hat X, \beta)) \cong H^{\dim(X) +m}(U_\lambda,\,\underline{\dC}^{-\ell/k}_\lambda)
        \end{equation}
        resp.
        \begin{equation}\label{equ:solHolRank1b}
            H^m (i_\lambda^\dag \tau(\rho, \hat X, \beta)) \cong H_c^{\dim(X) + m}(U_\lambda,\,\underline{\dC}^{-\ell/k}_\lambda)
        \end{equation}
        for all $m \in \Z$ and for all $\lambda\in W$.
                                            \item
        If we assume that the hypotheses of \cref{theo:tau_is_MHM}, point 2., hold true (i.e., $\beta(\mathbf{e})\in \dZ_{>0}$ and $\Ell^{\otimes \beta(\mathbf e)} \cong \omega_X^\vee$ as $G$-equivariant line bundles), then we have an isomorphism
        $$
        \tau(\rho,\hat X, \beta) \cong H^0 a_{\cU,+} {ev_{|\cU}^+}\cO_{L^*}.
        $$
        In particular, we obtain for all $\lambda\in W$ an isomorphism
        \begin{equation}\label{equ:solHolRank2}
            H^0 (i_\lambda^+ \tau(\rho, \hat X, \beta)) \cong H^{\dim(X)}(U_\lambda,\,\dC).
        \end{equation}

    \item The holonomic rank of $\tau(\rho,\hat X, \beta)$ is given in the two cases as
        $$
        \dim H_c^{\dim(X)}(U_\lambda, \underline{\dC}_\lambda^{-\ell/k})\simeq \dim H^{\dim(X)}(U_\lambda, \underline{\dC}_\lambda^{-\ell/k}) \quad\quad\quad\textup{ if } \beta(\mathbf{e})\in\dQ\setminus\dZ
        ,
        $$
        resp.\
        $$
        \dim H^{\dim(X)}(U_\lambda, \dC) \quad\quad\quad\textup{ if }\beta(\mathbf{e})\in\dZ_{>0},
        $$
        for any value $\lambda\in W$ that lies outside the singular locus of $\tau(\rho,\hat X, \beta)$. \end{enumerate}
\end{cor}
Notice that $H^0 (i_\lambda^+ \tau(\rho, \hat X, \beta))$ is the space dual to the space of (classical) solutions of $\tau(\rho, \hat{X},\beta)$ at the point $\lambda$, so that that \cref{equ:solHolRank1a} and \cref{equ:solHolRank2}  also comprise and generalize \cite[Corollary 2.3]{HuangLianZhu}.
\begin{proof}
\begin{enumerate}
    \item
        Using the previous \cref{theo:tau_is_MHM}, the first statement is exactly the $\mathscr{D}$-module version of  \cref{prop:TautIsDirectImage}, 1. Similarly,
        the second statement follows from \cref{prop:TautIsDirectImage}, 2.
    \item
        The first statement is obtained by combining \cref{theo:tau_is_MHM} with \cref{prop:TautIsDirectImage}, 1. In order to get the second one, we apply the functor $H^0i_\lambda^+$ to the isomorphism $\tau(\rho,\hat X, \beta) \cong H^0 a_{\cU,+} {ev_{|\cU}^+}\cO_{L^*}$. This shows that $H^0 i_\lambda^+ \tau(\rho,\hat X, \beta)$ sits at the origin of the $E_2$-term of the (third quadrant) Grothendieck spectral sequence for the composition of the functors $i_\lambda^+$ and $a_{\cU,+}$. Therefore, it is isomorphic to the $(0,0)$-spot of the abutment, which is $
            H^0 i_\lambda^+ a_{\cU,+} {ev_{|\cU}^+}\cO_{L^*}$.         Now the remainder of the argument is close to the proof of part 2. of \cref{prop:TautIsDirectImage}, namely,
        consider the diagram
        $$
         \begin{tikzcd}
            U_\lambda=(\{\lambda\}\times X)\cap\mathcal{U} \ar{rr}{a^{\cU_\lambda}} \ar[swap]{dd}{i_\lambda\times \textup{id}} \ar[bend right=60, swap]{dddd}{\lambda_{|U_\lambda}}&& \{\lambda\} \ar[swap]{dd}{i_\lambda} \\ \\
            \mathcal{U}  \ar[swap]{dd}{ev_{|\mathcal{U}}}
            \ar{rr}{a_{\mathcal{U}}}&& W \\ \\L^*
          \end{tikzcd}
          $$
        then base change shows that
        $$
        i_\lambda^+ a_{\cU,+} {ev_{|\cU}^+}\cO_{L^*}\cong a^{\cU_\lambda}_+(i_\lambda\times \id)^+
        {ev_{|\cU}^+}\cO_{L^*} \cong a^{\cU_\lambda}_+ \lambda_{U_\lambda}^+\cO_{L^*} \cong a^{\cU_\lambda}_+ \cO_{U_\lambda}.
        $$
        Taking zeroth cohomology, we finally obtain
        $$
            H^0 i_\lambda^+ \tau(\rho,\hat X, \beta) \cong H^0 i_\lambda^+ a_{\cU,+} {ev_{|\cU}^+}\cO_{L^*}
            \cong
            H^0 a^{\cU_\lambda}_+ \cO_{U_\lambda}
            = H^{\dim(X)}(U_\lambda,\dC).
        $$
    \item
       This follows from point 1. resp.\ 2. since the holonomic rank is the fibre dimension
       of $\tau(\rho,\hat{X},\beta)$ at any $\lambda\in W$ outside the singular locus. For such points $\lambda$ we also have $i_\lambda^+=i_\lambda^\dag$ and this is then an exact functor.
\end{enumerate}
\end{proof}

\bibliographystyle{amsalpha}
\newcommand{\etalchar}[1]{$^{#1}$}
\providecommand{\bysame}{\leavevmode\hbox to3em{\hrulefill}\thinspace}
\providecommand{\MR}{\relax\ifhmode\unskip\space\fi MR }
\providecommand{\MRhref}[2]{%
  \href{http://www.ams.org/mathscinet-getitem?mr=#1}{#2}
}
\providecommand{\href}[2]{#2}

\vspace*{1cm}

\nd
Paul G\"orlach\\
Otto-von-Guericke-Universität Magdeburg\\
Fakult\"at f\"ur Mathematik\\
Institut f\"ur Algebra und Geometrie\\
Universitätsplatz 2\\
39106 Magdeburg\\
Germany\\
paul.goerlach@ovgu.de\\

\nd
Thomas Reichelt\\
Lehrstuhl f\"ur Mathematik VI \\
Institut f\"ur Mathematik\\
Universit\"at Mannheim,
A 5, 6 \\
68131 Mannheim\\
Germany\\
thomas.reichelt@math.uni-mannheim.de\\

\nd
Christian Sevenheck\\
Fakult\"at f\"ur Mathematik\\
Technische Universit\"at Chemnitz\\
09107 Chemnitz\\
Germany\\
christian.sevenheck@mathematik.tu-chemnitz.de\\

\nd
Avi Steiner\\
Fakult\"at f\"ur Mathematik\\
Technische Universit\"at Chemnitz\\
09107 Chemnitz\\
Germany\\
avi.steiner@gmail.com\\

\nd
Uli Walther\\
Purdue University\\
Dept.\ of Mathematics\\
150 N.\ University St.\\
West Lafayette, IN 47907\\
USA\\
walther@math.purdue.edu\\

\end{document}

\end{document}